\documentclass[a4,12pt,twoside]{book}
 \usepackage{amsgen, amstext, amsbsy, amsopn, amsthm, mathtools, amsfonts, amssymb, amscd, amsmath, euscript, enumerate, url, verbatim, calc, bigints}
 \usepackage[all]{xy}
 \usepackage{color}
\usepackage{fancyhdr}
\usepackage{makeidx}
\makeindex 
\pagebreak
 \usepackage{latexsym}
 \usepackage{graphics}
 \usepackage{color}

\newcommand{\C}{\mathbb{C}} 
\newcommand{\Z}{\mathbb{Z}}
\newcommand{\F}{\mathbb{F}}
\newcommand{\G}{\mathbb{G}}
\newcommand{\GL}{{\rm GL}}
\newcommand{\SL}{{\rm SL}}
        
\newcommand{\trace}{\text{trace}}
\newcommand{\ind}{{\rm ind}}
\newcommand{\Ind}{{\rm Ind}}
\newcommand{\Hom}{{\rm Hom}}
\newcommand{\Rep}{{\rm Rep}}
\newcommand{\meas}{{\rm meas}}
\newcommand{\vol}{{\rm vol}}
\newcommand{\tr}{{\rm tr}}
\newcommand{\Ad}{{\rm Ad}}
\newcommand{\Wh}{{\rm Wh}}
\newcommand{\Int}{{\rm Int}}
\newcommand{\val}{{\rm val}}
\newcommand{\g}{\mathfrak{g}}
\newcommand{\bg}{\boldsymbol{\mathfrak{g}}}
\newcommand{\supp}{{\rm supp}}
\newcommand{\Ext}{{\rm Ext}}

\makeatletter
\def\oversortoftilde#1{\mathop{\vbox{\m@th\ialign{##\crcr\noalign{\kern3\p@}%
      \sortoftildefill\crcr\noalign{\kern3\p@\nointerlineskip}%
      $\hfil\displaystyle{#1}\hfil$\crcr}}}\limits}

\def\sortoftildefill{$\m@th \setbox\z@\hbox{$\braceld$}%
  \braceld\leaders\vrule \@height\ht\z@ \@depth\z@\hfill\braceru$}
\makeatother
\theoremstyle{plain}
 \newtheorem{theorem}{Theorem}[section]
 
 \newtheorem{corollary}[theorem]{Corollary}
 \newtheorem{lemma}[theorem]{Lemma}
 \newtheorem{proposition}[theorem]{Proposition}

 \newtheorem{workinghypothesis}[theorem]{Working Hypothesis}
 \newtheorem{Refined Question}[theorem]{Refined Question}
  \newtheorem{question}[theorem]{Question}

 \theoremstyle{definition}
 
 \newtheorem{remark}[theorem]{Remark}

 \newtheorem{definition}[theorem]{Definition}

 \theoremstyle{remark}

\begin{document}
\thispagestyle{empty}

\begin{center}
%{\bf
{\LARGE {\bf Branching laws on the metaplectic \\ cover of ${\rm GL}_{2}$}} \\ \vspace{0.5 in}
{\LARGE A Thesis} \vspace{0.7 in} \\
{\LARGE Submitted to the} \\
{\LARGE Tata Institute of Fundamental Research, Mumbai}\\
{\LARGE for the degree of Doctor of Philosophy} \\
{\LARGE in Mathematics} \vspace{0.7 in} \\
{\LARGE by} \vspace{0.7 in} \\
{\LARGE Shiv Prakash Patel} \vspace{2.5 cm} \\
{\LARGE School of Mathematics} \\
{\LARGE Tata Institute of Fundamental Research} \\
{\LARGE Mumbai} \\
{\LARGE May, 2014}
%}
\end{center}

\thispagestyle{empty}
\cleardoublepage

\thispagestyle{empty}
%\addcontentsline{toc}{chapter}{Declaration}
\begin{center} 
 {\bf \LARGE Declaration}
\end{center}

\noindent This thesis is a presentation of my original research work.
Wherever contributions of others are involved, every effort is
made to indicate this clearly, with due reference to the literature,
and acknowledgement of collaborative research and discussions.\\

\noindent The work was done under the guidance of Professor Dipendra Prasad,
at the Tata Institute of Fundamental Research, Mumbai. \\ \vspace{0.1 in}

\begin{flushright}
  Shiv Prakash Patel $~~~~~~~~~~~$\\ \vspace{1.2 in}
\end{flushright}

{\noindent}In my capacity as supervisor of the candidate's thesis, I certify
that the above statements are true to the best of my knowledge. \\ \vspace{0.1 in}

{\noindent}Prof. Dipendra Prasad \\ 
%\vspace{1.2 cm}

{\noindent}Date:

\thispagestyle{empty}
\cleardoublepage

\thispagestyle{empty}
\vspace{5 cm}
\begin{center}
{\large \it Dedicated to\\
\vspace{1 cm}
My Parents}
\end{center}

\thispagestyle{empty}
\cleardoublepage

\setcounter{page}{0}
\thispagestyle{empty}

\pagenumbering{roman}
\setcounter{page}{1}

\thispagestyle{plain}
\addcontentsline{toc}{chapter}{Acknowledgements}
\begin{center}
{\large \bf Acknowledgements} 
\end{center}
It gives me an immense pleasure to thank my advisor Professor Dipendra Prasad for introducing me to the representation theory of $p$-adic groups. This thesis would not have been possible without his constant encouragement and numerous help. Professor Prasad has always been very kind and generous to me, whether it is mathematics or any non-academic matter.

I am very fortunate to be in a vibrant atmosphere at TIFR, which has so many mathematician and very good teachers of mathematics. I thank Professors R. Sujatha, N. Fakhruddin, S. Bhattacharya, A. Nair, Raja Sridharan, I. Biswas for their first year courses. The lecture courses given by Professors D. Prasad, C. S. Rajan, N. Nitsure, M. S. Raghunathan, T. N. Venkataramana, Sandeep Varma have been very fruitful. My special thanks are due to Professor Sandeep Varma by whom I have been greatly benefited both in mathematics and non-academic matters. It has always been a pleasure talking to him. In fact, Chapter 4 would not have been possible in the present form without his numerous help at various moments. Interactions with Professors T. N. Venkataramana, C. S. Rajan, Ravi A. Rao and E. Ghate were very helpful. 

I have been greatly benefited by several visitors at TIFR in last few years, namely Professors A. Silberger, S. Spallone, Bin Yong Sun, D. Manderscheid, W. Casselman, V. Heiermann, Manish Thakur, A Raghuram, Wen-Wei Li etc. I also thank Professor Wee Teck Gan for several suggestions. Several visits of Venket to TIFR have been very helpful.

Thanks are also due to my teachers at IIT Bombay and at University of Allahabad, namely Professors S. R. Ghorpade, G. K. Srinivasan, R. Raghunathan, U. K. Anandvardhanan at IIT Bombay and  Professors R. P. Shukla, Ramji Lal, B. K. Sharma, K. K. Azad at University of Allahabad. The MT \& TS program organized by the NBHM has been very helpful in my studies and I thank Professor S. Kumaresan for giving me an opportunity to be a part of it.
\newpage
I would like to thank the office staff of School of Mathematics at TIFR for the pleasant stay who have made all administrative issues very smooth and simple. Thanks to V. Nandgopal and Vishal for computer related helps.
  
I would like to thank my parents for their unconditional love and support. I also thank my brothers and sisters and their families for their support which has been very helpful during difficult times.

Finally, I would like to thank all my friends at TIFR and outside, specially my batch-mates in TIFR, who have shared their moments, thoughts, philosophy and have guided me at several occasions.

\addcontentsline{toc}{chapter}{Contents}
\tableofcontents
\clearpage

\thispagestyle{plain}

\thispagestyle{plain}
 
\setlength{\headheight}{17.5pt}
\pagestyle{fancy}

\fancyhf{}

\fancyhead[LE,RO]{\thepage}
\fancyhead[RE]{Chapter \thechapter, Section \thesection}
\fancyhead[LO]{\nouppercase{\rightmark}}
\renewcommand{\headrulewidth}{0pt}
\fancypagestyle{plain}{ %
\fancyhf{} % remove everything
\renewcommand{\headrulewidth}{0pt} % remove lines as well
\renewcommand{\footrulewidth}{0pt}}

\pagenumbering{arabic}
\setcounter{page}{0}
\thispagestyle{plain}

\chapter{Introduction}
\section{Some background and history}
Let $G$ be a group and $H$ a subgroup of $G$. Let $\pi_1$ and $\pi_2$ be irreducible representations of $G$ and $H$ respectively. By `branching laws' one refers to rules describing the space $\Hom_{H}( \pi_1, \pi_2)$ of $H$-equivariant linear maps from $\pi_1$ to $\pi_2$. The sixties saw the birth of the celebrated Langlands conjectures that predicted deep connections between the representation theory of reductive groups (over local fields and over the adeles of a global field), which may be called the `automorphic or harmonic analytic side', and the study of Galois representations, which may be referred to as the `arithmetic side'. In the nineties, B. Gross and D. Prasad \cite{GrossPrasad92} systematically investigated branching laws for certain classical groups and came up with a conjectural answer to many branching problems in terms of objects on the arithmetic side, known as the Langlands parameters and `$\epsilon$-factors'. This has been extended to cover many cases in \cite{GGP12}.

In another direction, Shimura's work on modular forms of half-integral weight \cite{Shimura73}, and various follow-ups had suggested that an analogue of Langlands' program should be there not only for reductive groups over local and global fields, but also for certain covering groups of these. Since then many people have studied representation theory for these covering groups, and recently there have been many attempts (e.g., \cite{Weissman09}, \cite{Weissman13}) to adapt Langlands' conjectures to the setting of covering groups. 
%The work of J.-L. Waldspurger, cf. \cite{Wald91}, contains some of the most important results on the representation theory of metaplectic group associated to ${\SL_{2}$. 
This task  appears formidable but also seems to be a potential source for a rich theory. Thus, it seems natural to investigate branching laws in the context of covering groups, and this is what this thesis endeavors to do.

One of the first observations of Gross and Prasad was that the restriction problem for a ($p$-adic) pair $(G, H)$ should not be studied in isolation but together with that for various pairs $(G', H')$, as $(G', H')$ runs over groups closely related to $(G, H)$, known as the (pure) inner forms of $G, H$. In an early work \cite{Prasad92}, Prasad studied the restriction problem for the pairs $({\rm GL}_2(E), {\rm GL}_2(F))$ and $({\rm GL}_2(E), D_F^{\times})$, where $F$ is a non-Archimedian local field of characteristic zero, $D_F$ is the unique quaternion division algebra \index{quaternion division algebra} over $F$ and $E$ is a quadratic extension of $F$. In hindsight, this case already captures many of the subtleties that show up in the general case. 

For $p$-adic groups, where most representations are infinite dimensional, precise questions about branching have been formulated only in contexts where we have a theorem of `multiplicity one', or at least `finite multiplicity'. More precisely, one restricts to pairs $(G, H)$ of $p$-adic groups and a class of pairs $(\pi_1, \pi_2)$ of representations that ensure that the dimension of $\Hom_{H}(\pi_1, \pi_2)$ is 0 or 1, or at least finite. 

Prasad proved a multiplicity one theorem for the pair $({\rm GL}_{2}(E), {\rm GL}_{2}(F))$ and gave a classification of pairs $(\pi_1, \pi_2)$ of irreducible `admissible' representations $\pi_{1}$ of ${\rm GL}_{2}(E)$ and $\pi_2$ of ${\rm GL}_{2}(F)$ such that 
\[
\Hom_{\GL_{2}(F)}(\pi_1, \pi_2) \neq 0.
\]
Further, he showed that there is a certain dichotomy relating the restriction problem for the pairs $(\GL_{2}(E), \GL_{2}(F))$ and $(\GL_{2}(E), D_{F}^{\times})$. More precisely, the following theorem was proved in \cite{Prasad92}.

\begin{theorem} [Prasad] \label{main:DP}
Let $\pi_{1}$ and $\pi_{2}$ be irreducible infinite dimensional representations of $\GL_{2}(E)$ and $\GL_{2}(F)$ respectively. Assume that the central character of $\pi_{1}$ restricted to the center of $\GL_{2}(F)$ is the same as the central character of $\pi_{2}$. Then
\begin{enumerate}
\item for a principal series representation $\pi_{2}$ of $\GL_{2}(F)$, we have
\[
\dim \Hom_{\GL_{2}(F)} (\pi_{1}, \pi_{2}) = 1,
\]
\item for a discrete series representation $\pi_{2}$ of $\GL_{2}(F)$, letting $\pi_{2}'$ be the finite dimensional irreducible representation of $D_{F}^{\times}$ associated to $\pi_{2}$ by the Jacquet-Langlands correspondence, \index{Jacquet-Langlands correspondence} we have
\[
\dim \Hom_{\GL_{2}(F)} (\pi_{1}, \pi_{2}) + \dim \Hom_{D_{F}^{\times}} (\pi_{1}, \pi_{2}') = 1.
\]
\item There is a criterion in terms of a certain epsilon factor attached to $\pi_{1}$, $\pi_{2}$ that determines when $\dim \Hom_{\GL_{2}(F)}(\pi_{1}, \pi_{2}) = 1$.
\end{enumerate}
\end{theorem}
\section{Description of the problem}
The first covering group to be considered is $\widetilde{\GL_2(E)}$, a two fold cover of $\GL_{2}(E)$ known as the metaplectic cover of $GL_2(E)$, which  will be defined in Section \ref{definition: 2-fold cover} by an explicit (Kubota) cocycle with values in $\mu_{2} = \{ \pm 1 \}$, giving rise to an exact sequence of groups
\[
1 \rightarrow \mu_{2} \rightarrow \widetilde{\GL_{2}(E)} \rightarrow \GL_{2}(E) \rightarrow 1.
\]
To consider the restriction problem for $(\widetilde{\GL_{2}(E)}, \GL_{2}(F))$ in analogy with that for the pair $(\GL_{2}(E), \GL_{2}(F))$, we need $\GL_{2}(F)$ and $D_{F}^{\times}$ to be realized in a suitable manner as subgroups of $\widetilde{\GL_{2}(E)}$. The first question to be analysed is whether $\GL_{2}(F)$ and $D_{F}^{\times}$ are subgroups of $\widetilde{\GL_{2}(E)}$. More specifically, the question is whether the covering $\widetilde{\GL_{2}(E)}$ splits when restricted to $\GL_{2}(F)$ and $D_{F}^{\times}$. We shall discuss this question in Chapter \ref{splitting questions}. It turns out to be rather easy to prove this for $\GL_{2}(F)$ but not so for $D_{F}^{\times}$. Actually, we consider $\C^{\times}$-covering of $\GL_{2}(E)$ obtained from $\widetilde{\GL_{2}(E)}$, namely $\widetilde{\GL_{2}(E)} \times_{\mu_{2}} \C^{\times}$, and it is this covering that splits when restricted to $D_{F}^{\times}$, see Theorem \ref{main theorem: splitting}. An admissible representation of $\widetilde{\GL_{2}(E)}$ (respectively, $\widetilde{\GL_{2}(E)} \times_{\mu_{2}} \C^{\times}$) is called genuine if the action of $\mu_{2}$ (respectively, $\C^{\times}$) is non-trivial (respectively, $\C^{\times}$ acts by identity) on the representation space. It is clear that the category of genuine representations of $\widetilde{\GL_{2}(E)}$ and that of $\widetilde{\GL_{2}(E)} \times_{\mu_{2}} \C^{\times}$ are equivalent. Hence there is no harm in replacing the group $\widetilde{\GL_{2}(E)}$ by $\widetilde{\GL_{2}(E)} \times_{\mu_{2}} \C^{\times}$. We shall abuse the notation and often write $\widetilde{\GL_{2}(E)}$ for $\widetilde{\GL_{2}(E)} \times_{\mu_{2}} \C^{\times}$. After Theorem \ref{main theorem: splitting}, we will know that $\GL_{2}(F)$ and $D_{F}^{\times}$ are subgroups of $\widetilde{\GL_{2}(E)}$ but not canonically, as there are many `inequivalent' embeddings of these two subgroups inside $\widetilde{\GL_{2}(E)}$. In fact, the set of splittings of the map $p : \widetilde{\GL_{2}(E)} \rightarrow \GL_{2}(E)$ restricted to either of $\GL_{2}(F)$ or $D_{F}^{\times}$ is a principal homogeneous space over the character group of $F^{\times}$. Before we begin the study of restriction of representations from $\widetilde{\GL_{2}(E)}$ to $\GL_{2}(F)$ and to $D_{F}^{\times}$, we need to fix an splitting of these two subgroups inside $\widetilde{\GL_{2}(E)}$. We also require that these fixed embeddings of $\GL_{2}(F)$ and $D_{F}^{\times}$ inside $\widetilde{\GL_{2}(E)}$ are compatible in the sense that they satisfy a `technical' condition, see ``Working Hypothesis \ref{working hypothesis}" in  Section \ref{introduction: DP-metaplectic} formulated by D. Prasad. We are not able to prove this hypothesis at the moment and hence we assume it, and put it to use in Chapter \ref{chapter:restriction}. \\

For $X \subset \GL_{2}(E)$, let $\tilde{X}$ denote the inverse image of $X$ in $\widetilde{\GL_{2}(E)}$. Let $Z$ be the center of $\GL_{2}(E)$. It is important to note that the subgroup $\tilde{Z}$ is an abelian group containing the center of $\widetilde{\GL_{2}(E)}$ but it is not the center of $\widetilde{\GL_{2}(E)}$. The center of $\widetilde{\GL_{2}(E)}$ is $\tilde{Z^2}$. Let $\pi$ be an irreducible admissible genuine representation of $\widetilde{\GL_{2}(E)}$. In the study of the restriction of a representation of $\widetilde{\GL_{2}(E)}$ to the subgroups $\GL_{2}(F)$ and $D_{F}^{\times}$, the space of Whittaker functionals of representations of $\widetilde{\GL_{2}(E)}$ plays an important role. Let $\omega_{\pi}$ be the central character of $\pi$. Define $\Omega(\omega_{\pi}) = \{ \mu : \tilde{Z} \rightarrow \C^{\times} \mid \mu|_{\tilde{Z^2}} = \omega_{\pi} \}$. Sometimes we regard $\Omega(\omega_{\pi})$ as a $\tilde{Z}$-module  $\oplus_{\mu \in \Omega(\omega_{\pi})} \mu$. Let $\psi$ be a non-trivial additive character of $E$. The $\psi$-twisted Jacquet functor $\pi_{N, \psi}$ is a finite dimensional completely reducible $\tilde{Z}$-module, each character of $\tilde{Z}$ appearing with multiplicity $\leq 1$ (by \cite[~Theorem 4.1]{GHP79}). With these notations, we wish to prove the following theorem which is analogous to Theorem \ref{main:DP}.
\begin{theorem} \label{introduction:DP-mataplectic}
Let $\pi_{1}$ be an irreducible admissible genuine representation of $\widetilde{\GL_2(E)}$ and let $\pi_2$ be an infinite dimensional irreducible admissible representation of $\GL_2(F)$. Assume that the central characters $\omega_{\pi_{1}}$ of $\pi_{1}$ and $\omega_{\pi_{2}}$ of $\pi_{2}$ agree on $E^{\times 2} \cap F^{\times}$. Fix a non-trivial additive character $\psi$ of $E$ such that $\psi|_{F} = 1$. Then:
\begin{enumerate}
 \item For a principal series representation $\pi_2$ of $\GL_2(F)$, (except for a few pairs $(\pi_{1}, \pi_{2})$ for a given $\pi_{1}$ to be described explicitly in Chapter \ref{chapter:restriction}) we have
 \[
  \dim \Hom_{\GL_2(F)} \left( \pi_{1}, \pi_{2} \right) = \dim \Hom_{Z(F)}( (\pi_{1})_{N, \psi}, \omega_{\pi_{2}}).
 \]
 \item For a principal series representation $\pi_1$ of $\widetilde{\GL_2(E)}$ and a discrete series representation $\pi_2$ of $\GL_2(F)$, let $\pi_2'$ be the finite dimensional representation of $D_{F}^{\times}$ associated to $\pi_2$ by the Jacquet-Langlands correspondence. Then (except for a few pairs $(\pi_{1}, \pi_{2})$ for a given $\pi_{1}$ to be described explicitly in Chapter \ref{chapter:restriction}) we have
 \[
  \dim \Hom_{\GL_2(F)} \left( \pi_{1}, \pi_{2} \right) + \dim \Hom_{D_{F}^{\times}} \left( \pi_{1}, \pi_{2}' \right) =  [E^{\times} : F^{\times}E^{\times 2}].
 \]
 \item For an irreducible admissible genuine representation $\pi_1$ of $\widetilde{\GL_2(E)}$ and an irreducible supercuspidal representation $\pi_2$ of $\GL_2(F)$, let $\pi_{1}'$ be an admissible genuine representation of $\widetilde{\GL_{2}(E)}$ which has the same central character as $\pi_{1}$ and $(\pi_{1})_{N, \psi} \oplus (\pi_{1}')|_{N, \psi} = \Omega(\omega_{\pi_{1}})$ as $\tilde{Z}$-modules. Let $\pi_2'$ be the finite dimensional representation of $D_{F}^{\times}$ associated to $\pi_2$ by the Jacquet-Langlands correspondence. Then
\[
 \dim \Hom_{\GL_2(F)} \left( \pi_{1} \oplus \pi_{1}', \pi_{2} \right) + \dim \Hom_{D_{F}^{\times}} \left( \pi_{1} \oplus \pi_{1}', \pi_{2}' \right) =  [E^{\times} : F^{\times}E^{\times 2}].
\] 
\end{enumerate}
\end{theorem}
\section{The strategy of proofs}
The strategy to prove this theorem is similar to that in \cite{Prasad92}, which we briefly recall here. Part 1 of Theorem \ref{introduction:DP-mataplectic} is proved by looking at the Kirillov model of an irreducible admissible genuine representation of $\widetilde{\GL_{2}(E)}$ and its Jacquet module restricted to $\GL_{2}(F)$. Part 2 of Theorem \ref{introduction:DP-mataplectic} is proved using Mackey theory. Part 3 of Theorem \ref{introduction:DP-mataplectic} is proved using a trick of Prasad in \cite{Prasad92}, where we `transfer' results of principal series representations (as in Part 2) to the representations which do not belong to principal series (Prasad `transfers' the results from a principal series representation to a discrete series representation). This is done by using character theory and an analogue of a result of Casselman and Prasad \cite[Theorem 2.7]{Prasad92} for $\widetilde{\GL_{2}(E)}$. The theorem of Casselman-Prasad is as follows.
\begin{theorem} \label{Intro: Casselman-Prasad}
Let $\pi_{1}$ and $\pi_{2}$ be two irreducible admissible infinite dimensional representations of $\GL_{2}(E)$ which have the same central character. Then the virtual representation $\pi_{1} - \pi_{2}$ of $\GL_{2}(E)$ restricted to any compact modulo central subgroup of $\GL_{2}(E)$ is finite dimensional.
\end{theorem}
 Let $\Theta_{\pi}$ denote the character of an admissible representation $\pi$ (see Section \ref{prelim:character expansion}).
Prasad gives a proof of Theorem \ref{Intro: Casselman-Prasad} by observing that $\Theta_{\pi_{1}} - \Theta_{\pi_{2}}$ is an everywhere smooth function on $\GL_{2}(E)$. We need an analogue of this result for $\widetilde{\GL_{2}(E)}$, which is as follows:
\begin{theorem} \label{intro: CP-mataplectic}
Let $\Pi_{1}$ and $\Pi_{2}$ be two irreducible admissible genuine representations of $\widetilde{\GL_{2}(E)}$ with the same central character and such that $(\Pi_{1})_{N, \psi} cong (\Pi_{2})_{N, \psi}$ as $\tilde{Z}$-modules, where $\psi$ is a non-trivial additive character of $E$. Then $\Theta_{\Pi_{1}} - \Theta_{\Pi_{2}}$ is a smooth function on $\widetilde{\GL_{2}(E)}$.
\end{theorem}
This theorem is an application of a theorem of F. Rodier \cite{Rod75}, generalized by C. M{\oe}glin and J.-L. Waldspurger \cite{MW87} and is extended by this author to the setting of covering groups (see Theorem \ref{main:theorem}). We prove the theorem of M{\oe}glin-Waldspurger for any covering group $\tilde{G}$ of a connected reductive group $G$ in Chapter \ref{MW-covering}. This is an important part of the thesis and we give an overview of this result below.\\

\section{A theorem of M{\oe}glin-Waldspurger for covering groups}
Let {\bf G} be a connected reductive group defined over $E$ and $G = {\bf G}(E)$. Let $\bg = \text{Lie}({\bf G})$ be the Lie algebra of  ${\bf G}$ and $\g = \bg(E)$. A theorem of F. Rodier \cite{Rod75} for connected reductive split groups relates the dimension of a certain space of non-degenerate Whittaker forms of an irreducible admissible representation $(\pi, W)$ to a certain coefficient in the character expansion $\Theta_{\pi}$ of $\pi$ around the identity. Rodier, for his proof, had to assume that the residual characteristic is large enough. The theorem of Rodier was generalised by C. M{\oe}glin and J.-L. Waldspurger \cite{MW87} in several directions, yielding in particular a statement for arbitrary connected reductive group over $p$-adic field of odd residual characteristic. The theorem of M{\oe}glin-Waldspurger is a more precise statement about certain coefficients in the character expansion around identity and certain spaces of `degenerate' Whittaker forms (see Section \ref{degenerate_W_forms} for the definition of degenerate Whittaker forms). In the case of even residual characteristic, the theorem of M{\oe}glin-Waldspurger has been recently proved by S. Varma \cite{San14}. We generalize this theorem of M{\oe}glin-Waldspurger to the setting of a locally compact topological central extension of an arbitrary connected reductive group defined over a $p$-adic field of arbitrary residual characteristic in Chapter \ref{MW-covering}. \\

Let $\mu_{r} := \{ z \in \C^{\times} : |z|^{r} = 1 \}$. Let $\tilde{G}$ be a locally compact topological central extension of $G$ by $\mu_{r}$. Let $Y \in \bg$ be a nilpotent element and $\varphi : \mathbb{G}_{m} \rightarrow {\bf G}$ be a one parameter subgroup of ${\bf G}$ satisfying
\begin{equation} \label{condition: Y and phi}
\Ad(\varphi(s))Y = s^{-2}Y.
\end{equation}
Let $(\pi, W)$ be an irreducible admissible genuine representation of $\tilde{G}$. We fix a non-trivial additive character $\psi$ of $E$ with conductor $\mathfrak{O}_{E}$, where $\mathfrak{O}_{E}$ is ring of integers of $E$. Associated to a pair $(Y, \varphi)$ as in \ref{condition: Y and phi} one can define a certain space $\mathcal{W}_{(Y, \varphi)}$, called the space of degenerate Whittaker forms of $(\pi, W)$ relative to $(Y, \varphi)$ (see Section \ref{degenerate_W_forms}). Let $\mathcal{N}_{\Wh}(\pi)$ denote the set of nilpotent orbits $\mathcal{O}$ of $\g$ for which there exists an element $Y \in \mathcal{O}$ and a one parameter subgroup $\varphi$ satisfying  Equation \ref{condition: Y and phi}, such that $\mathcal{W}_{(Y, \varphi)} \neq 0$. Recall that the Harish-Chandra-Howe character expansion (as extended by Wen-Wei Li in \cite{WWLi} in the setting of a covering group) of $(\pi, W)$ around the identity is a sum $\sum_{\mathcal{O}} c_{\mathcal{O}} \widehat{\mu_{\mathcal{O}}}$, where $\mathcal{O}$ varies over the set of nilpotent orbits of $\bg$, each $c_{\mathcal{O}}$ is a complex number and $\widehat{\mu_{\mathcal{O}}}$ denotes the Fourier transform of a suitably normalized invariant measure $\mu_{\mathcal{O}}$ on $\mathcal{O}$. Let $\mathcal{N}_{\tr}(\pi)$ denote the set of nilpotent orbits $\mathcal{O}$ of $\g$ such that the corresponding coefficient $c_{\mathcal{O}}$ is non-zero in the character expansion of $\pi$ around the identity. There is a natural partial order on the set of nilpotent orbits in $\g$: $\mathcal{O}_{1} \leq \mathcal{O}_{2}$ if $\bar{\mathcal{O}_{1}} \subset \bar{\mathcal{O}_{2}}$. Let ${\rm Max}(\mathcal{N}_{\Wh}(\pi))$ and ${\rm Max}(\mathcal{N}_{\tr}(\pi))$ denote the sets of maximal elements in $\mathcal{N}_{\Wh}(\pi)$ and $\mathcal{N}_{\tr}(\pi)$ respectively, with respect to the above partial order. Then we prove the following theorem which is a generalization of the main theorem of M{\oe}glin-Waldspurger in Chapter I of \cite{MW87}.
\begin{theorem} 
Let $\pi$ be an irreducible admissible genuine representation of $\tilde{G}$. Then
\[
 {\rm Max}(\mathcal{N}_{\Wh}(\pi)) =  {\rm Max}(\mathcal{N}_{\tr}(\pi)).
\]
Moreover, if $\mathcal{O}$ is an element in either of these sets, then for any $(Y, \varphi)$ as above with $Y \in \mathcal{O}$ we have
\[
c_{\mathcal{O}} = \dim \mathcal{W}_{(Y, \varphi)}.
\]
\end{theorem}
\section{A question}
Let us come back to Theorem \ref{introduction:DP-mataplectic}. In the proof of part 2 of the theorem, the number $[E^{\times} : F^{\times} E^{\times 2}]$ is related to the fact that $(\pi_{1})_{N, \psi} = \Omega(\omega_{\pi_{1}})$ for a principal series representation $\pi_{1}$ (see Proposition \ref{whittaker models of principal series}). But for a representation $\pi_{1}$, which is not a principal series, $(\pi_{1})_{N, \psi}$ is a proper $\tilde{Z}$-submodule of $\Omega(\omega_{\pi_{1}})$. To `compensate', we add another genuine representation $\pi_{1}'$ which has the same central character and satisfies $(\pi_{1})_{N, \psi} \oplus (\pi_{1}')|_{N, \psi} = \Omega(\omega_{\pi_{1}})$. Then we can utilize Theorem \ref{intro: CP-mataplectic} for $\pi_{1} \oplus \pi_{1}'$ and a suitable principal series representation $Ps$. It is not clear in general how to describe a `natural' $\pi_{1}'$ for a given $\pi_{1}$ with the above properties. The question of describing a `natural' $\pi_{1}'$ for a given $\pi_{1}$ reduces to the following question about the representations of $\widetilde{\SL_{2}(E)}$. 
\begin{question}
Let $\tau$ be an irreducible admissible genuine representation of $\widetilde{\SL_{2}(E)}$.  Is there a `natural' choice of an genuine admissible representation of finite length $\tau'$ with a central character (not necessarily irreducible) such that $\omega_{\tau} = \omega_{\tau'}$, and $\tau$ admits a non-zero $\psi$-Whittaker functional if and only if $\tau'$ does not admit a non-zero $\psi$-Whittaker functional for any non-trivial additive character $\psi$ of $E$ ? 
\end{question}

We remark that the Waldspurger involution defined on the set of isomorphism classes of irreducible admissible genuine representations of $\widetilde{\SL_{2}(E)}$, written as $\tau \mapsto \tau_{W}$, has the property that $\tau$ admits a non-zero $\psi$-Whittaker functional if and only if $\tau_{W}$ does not admit a non-zero $\psi$-Whittaker functional for any non-trivial character $\psi$ of $E$, but the central characters of $\tau$ and $\tau_{W}$ are different. However, the question above requires the central characters of $\tau$ and $\tau'$ to be the same. The fact that $\tau$ and $\tau_{W}$ have different central characters also makes it difficult to extend the Waldspurger involution from $\widetilde{\SL_{2}(E)}$ to $\widetilde{\GL_{2}(E)}$. \\

This question will be discussed in Chapter \ref{consequences:Waldspurger}, where we provide $\pi_{1}'$ for certain representations $\pi_{1}$. We are not able to construct a `natural' $\pi_{1}'$ for all irreducible admissible genuine supercuspidal representations $\pi_{1}$ of $\widetilde{\GL_{2}(E)}$ (equivalently, $\tau'$ for all irreducible admissible genuine supercuspidal representation $\tau$ of $\widetilde{\SL_{2}(E)}$). It is not clear if the inability to do so is a reflexion on us, or if there is a more fundamental reason for this. \\

In Chapter \ref{consequences:Waldspurger}, we also consider the question of restriction of an irreducible admissible genuine representation of $\widetilde{\GL_{2}(E)}$ to $\widetilde{\SL_{2}(E)}$ and show that this restriction may not satisfy  `multiplicity one'. In fact, we prove that the multiplicity can be either 0, 1, 2 or 4. The results in Chapter \ref{consequences:Waldspurger} are consequences of a theorem of Waldspurger in \cite{Wald91} which involve the so called Waldspurger involution and $\theta$-correspondence.

\chapter{Preliminaries}
\section{Linear algebraic groups}
A linear algebraic group ${\bf G}$ over a field $k$ is a closed subgroup of $\GL_{N}$ for some non-negative integer $N$. A linear algebraic group is called a torus if it is isomorphic to $(\G_{m})^{n}$ over $\bar{k}$ for some non-negative integer $n$, where $\G_{m}=\GL_{1}$. The radical ${\rm Rad}({\bf G})$ of ${\bf G}$ is defined to be the identity component of the maximal normal solvable subgroup of ${\bf G}$. A maximal connected solvable closed subgroup ${\bf B}$ of ${\bf G}$ is called a Borel subgroup. A closed subgroup ${\bf P}$ of ${\bf G}$ is called a parabolic subgroup if ${\bf G}/{\bf P}$ is a projective algebraic variety. Any Borel subgroup ${\bf B}$ is a parabolic subgroup of ${\bf G}$. An element $x \in {\bf G}$ is called unipotent if for any algebraic injective morphism $i : {\bf G} \hookrightarrow \GL_{N}$, $(i(x) - {\rm Id})^{N}=0$. A linear algebraic group is called unipotent if every element in it is a unipotent element. The unipotent radical of a linear algebraic group ${\bf G}$ is the subvariety of unipotent elements in ${\rm Rad}({\bf G})$ which can be shown to be a subgroup. 
%Let ${\bf B}$ be a Borel and ${\bf U}$ the unipotent radical of ${\bf B}$. Then ${\bf B}/{\bf U}$ is a torus.
The group ${\bf G}$ is called reductive \index{reductive group} (resp. semi-simple \index{semi-simple group}) if ${\rm Rad}({\bf G})$ is a torus (resp. trivial). Equivalently ${\bf G}$ can be defined to be reductive if its unipotent radical is trivial.   A reductive linear algebraic group ${\bf G}$ defined over $k$ is called quasi-split if there exists a Borel subgroup ${\bf B}$ of ${\bf G}$ which is defined over $k$. Moreover ${\bf G}$ is called split if it is quasi-split and ${\bf B}/{\bf U}$ is a split torus (i.e. isomorphic to $(\mathbb{G}_{m})^{n}$ over $k$), where ${\bf U}$ is the unipotent radical of ${\bf B}$ and $n$ is a non-negative integer.

\section{Covering groups and genuine representations}
Let $A$ be a finite abelian group. Let $E$ be a non-Archimedian local field of characteristic zero and ${\bf G}$ a connected reductive group defined over $E$. A locally compact topological central extension \index{central extension} $\tilde{G}$ of $G={\bf G}(E)$ by $A$ gives rise to the following exact sequence of groups
\[
1 \rightarrow A \rightarrow \tilde{G} \xrightarrow{p} G \rightarrow 1,
\]  
where image of $A$ lies in the center of $\tilde{G}$. Such an extension can be described by a suitable element in $H^{2}(G, A)$, where $A$ is considered to be a $G$-module with trivial action. If $\beta : G \times G \longrightarrow A$ is a 2-cocycle \index{cocycle} corresponding to the central extension $\tilde{G}$, then the group $\tilde{G}$ can be described more explicitly; namely $\tilde{G}$ can be identified with $G \times A$ as a set such that modulo this identification the multiplication in $\tilde{G}$ is described by $(g_1, a_1) \cdot (g_2, a_2) = (g_1 g_2, a_1 a_2 \beta(g_1, g_2))$. We refer to these groups as \index{covering groups} covering groups. A covering group $\tilde{G}$ is locally compact and totally disconnected like $G$. We remark that the topology on $\tilde{G}$ is not obtained by transferring the product topology on $G \times A$. It is well known that the covering $\tilde{G} \rightarrow G$ splits when restricted to a small enough open subgroup of $G$ and let $\mathcal{U}$ be such a subgroup. If we fix a splitting $s : \mathcal{U} \rightarrow \tilde{G}$ then we define $s(\mathcal{U})$ to be an open subgroup of $\tilde{G}$ and $s$ being homeomorphic onto its image, which in turn defines a topology on $\tilde{G}$ making it a topological group and it is this topology on $\tilde{G}$ with which we work.\\

Let $\tilde{x} \in \tilde{G}$ and $y \in G$. For $\tilde{y} \in p^{-1}(\{ y \})$, the element $\tilde{y} \tilde{x} \tilde{y}^{-1}$ is independent of the choice of $\tilde{y}$ in $p^{-1}(\{ y \})$. We abuse the notation and write $y \tilde{x} y^{-1}$ for $\tilde{y} \tilde{x} \tilde{y}^{-1}$.\\

Let $(\pi, W)$ be a representation of $\tilde{G}$, where $W$ is a complex vector space. The representation $(\pi,W)$ is called smooth if the stabilizer of every element $w \in W$ is an open subgroup of $\tilde{G}$.  The representation $(\pi,W)$ is called admissible \index{admissible} if $(\pi,W)$ is smooth and $\pi^{K} := \{ w \in W \mid \pi(k)w = w, \forall k \in K \}$ is finite dimensional for all open compact subgroup $K$ of $\tilde{G}$. If $(\pi,W)$ is an irreducible admissible representation of $\tilde{G}$ with central character $\omega_{\pi}$, then $\omega_{\pi}(A)$ is a finite cyclic (sub)group $\mu_{r}$ of $\C^{\times}$. Such a representation will factor through a representation of $\tilde{G}/\ker(\omega_{\pi}|_{A})$, which can be identified with a central extension of $G$ by $\mu_{r}$. Thus it suffices to consider only those central extensions of $G$ for which $A=\mu_{r}$ for $r \geq 1$. From now onward, we will consider only such extensions. A representation $(\pi,W)$ of $\tilde{G}$ is called genuine \index{genuine representation} if the action of $\mu_{r}$ is given by scalar multiplication, which makes sense as $\mu_{r} \subset \C^{\times}$.

\section{Character expansion and Whittaker functionals} \label{prelim:character expansion}
In this section, we recall some facts about the character distribution of an admissible genuine representation of locally compact topological central extension $\tilde{G}$ of $G= {\bf G}(E)$ by $\mu_{r}$ with $r \geq 1$, where ${\bf G}$ is a connected reductive group $G$ defined over $E$, see \cite[~Chapter 2]{WWLi14}. Let $\mathcal{C}_{c}^{\infty}(\tilde{G})$ be the space of smooth (locally constant) functions with compact support and let $f \in \mathcal{C}_{c}^{\infty}(\tilde{G})$. Let $\widehat{\mu_{r}} := \Hom(\mu_{r}, \C^{\times})$ and for $\xi \in \widehat{\mu_{r}}$ let
\[
\mathcal{C}_{c, \xi}^{\infty}(\tilde{G}) := \{ f \in \mathcal{C}_{c}^{\infty}(\tilde{G}) \mid f(\epsilon \tilde{x}) = \xi(\epsilon) f(\tilde{x}), \forall \epsilon \in \mu_{r} \text{ and } \forall \tilde{x} \in \tilde{G} \}.
\]
We have the following canonical decomposition
\[
\mathcal{C}_{c}^{\infty}(\tilde{G}) = \bigoplus_{\xi \in \widehat{\mu_{r}}} \mathcal{C}_{c, \xi}^{\infty} (\tilde{G}).
\]
Let $\Rep(\tilde{G})$ be the set of isomorphism classes of irreducible admissible representations of $\tilde{G}$. For $\xi \in \widehat{\mu_{r}}$, let $\Rep_{\xi}(\tilde{G}) = \{ \pi \in \Rep(\tilde{G}) \mid \pi(\epsilon) = \xi(\epsilon) \text{id} \}$.
Let $\xi, \chi \in \widehat{\mu_{r}}$. Let $\pi \in \Rep_{\chi}(\tilde{G})$ and $f \in \mathcal{C}_{c, \xi}^{\infty} (\tilde{G})$. Then the operator $\pi(f) : \pi \longrightarrow \pi \mbox{ given by } v \mapsto \int_{\tilde{G}} f(\tilde{x})\pi(\tilde{x})v \, d\tilde{x}$ is zero except for $\xi = \bar{\chi}$. Moreover, $\pi(f)$ has finite rank as $\pi$ is admissible and hence its trace is well defined. Then
\[
 f \mapsto \trace(\pi(f))
\]
defines a distribution on $\tilde{G}$, called the character distribution \index{character distribution} of $\pi$. Let $gen$ be the genuine character of $\mu_{r}$, i.e. $gen(\epsilon) = \epsilon$. If $\pi \in \Rep_{gen} (\tilde{G})$ then the character distribution is determined by its restriction to the subspace $\mathcal{C}_{c, \overline{gen}}^{\infty} (\tilde{G})$ of $\mathcal{C}_{c}^{\infty} (\tilde{G})$, the so called space of anti-genuine functions on $\tilde{G}$. From \cite[~Theorem 4.3.2]{WWLi}, the character distribution of an irreducible admissible genuine representation of $\tilde{G}$ is represented by a locally integrable function $\Theta_{\pi}$ on $\tilde{G}$, i.e.
\[
\trace(\pi(f)) = \int_{\tilde{G}} \Theta_{\pi}(\tilde{x})f(\tilde{x}) \, d\tilde{x}.
\]
This function $\Theta_{\pi}$ is conjugation invariant in the sense that it satisfies $\Theta_{\pi}(x^y) = \Theta_{\pi}(x)$. The function $\Theta_{\pi}$ is called the character of the representation $\pi$. The function $\Theta_{\pi}$ is known to be smooth at each regular semi-simple element in $\tilde{G}$. We have an asymptotic description (or Harish-Chandra-Howe character expansion) of $\Theta_{\pi}$ in a neighbourhood of any singular semisimple element of $\tilde{G}$. In a neighbourhood of the identity, the Harish-Chandra-Howe character expansion \index{character expansion} is of the form 
\[
\sum_{\mathcal{O}} c_{\mathcal{O}} \widehat{\mu_{\mathcal{O}}}
\]
where $\mathcal{O}$ runs over the set of $\Ad(G)$-orbits of nilpotents elements in the Lie algebra $\g$, the $c_{\mathcal{O}} \in \C$ are constants and $\widehat{\mu_{\mathcal{O}}}$ is the Fourier transform \index{Fourier transform} of a suitably chosen $\Ad(G)$-invariant measure on the orbit $\mathcal{O}$. More precisely, there exists a neighbourhood $\mathcal{U}$ of the identity such that if $f \in C_{c, \overline{gen}}^{\infty} (\tilde{G})$ be such that support of $f$ lies in $\mathcal{U}$, then 
\[
\Theta_{\pi}(f) = \sum_{\mathcal{O}} c_{\mathcal{O}} \int_{\mathcal{O}} \widehat{f \circ \exp}(X) \, d\mu_{\mathcal{O}}.
\]
Note that we have implicitly used the fact that the covering $\tilde{G}$ splits when restricted to small enough open subgroup. In fact, there exists an exponential map $\exp : L \rightarrow \tilde{G}$, where $L$ is a sufficiently small open set containing $0$ in the Lie algebra $\g$ \cite{Serre65}.

Let ${\bf N}$ be a maximal unipotent subgroup of ${\bf G}$. The covering restriction of the covering $\tilde{G} \longrightarrow G$ to ${\bf N}$ splits in a unique way \cite{MW95}.  Let $\chi$ be a non-degenerate character of $N={\bf N}(E)$. Then the pair $(N, \chi)$ is called a non-degenerate  \index{Whittaker datum} Whittaker datum. 
\begin{definition}
Let $(N, \chi)$ be a non-degenerate Whittaker datum. A non-zero linear functional $\ell : W \longrightarrow \C$ is called a Whittaker functional with respect to $(N, \chi)$ if it satisfies the following condition:
\[
\ell(\pi(n)v) = \chi(n) \ell(v), \text{ for all } n \in N \text{ and } v \in W.
\]
\end{definition}

\section{Two fold covers of $\SL_{2}(E)$ and $\GL_{2}(E)$} \label{definition: 2-fold cover}
We give an explicit description of two fold covers of $\SL_{2}(E)$ and $\GL_{2}(E)$ by describing an explicit 2-cocycle defining each of these covers \cite{Kubota69}. For $g = \left( \begin{matrix} a & b \\ c & d \end{matrix} \right) \in \GL_2(E)$, set
\[
  x(g) = \left\{ 
  \begin{array}{l l}
    c & \quad \text{if $c \neq 0$}\\
    d & \quad \text{if $c=0$.}\\
  \end{array} \right.
\] 
Define
\begin{equation} \label{SL2 cocycle}
 \beta(g_{1}, g_{2}) = (x(g_1), x(g_2))(-x(g_1)^{-1}x(g_2), x(g_1 g_2))
\end{equation}
for $g_1, g_2 \in \SL_2(E)$, where $(*,*)$ denotes quadratic Hilbert symbol of the field $E$. For $g \in \GL_2(E)$, denote by $p(g)$the element of $\SL_{2}(E)$ which satisfies $g = \left( \begin{matrix} 1 & 0 \\ 0 & det(g) \end{matrix} \right) p(g)$. For $y \in E^{\times}$ write $g^{y}=\left( \begin{matrix} 1 & 0 \\ 0 & y^{-1} \end{matrix} \right) g \left( \begin{matrix} 1 & 0 \\ 0 & y \end{matrix} \right)$. Define
\[
 v(y, g) = \left\{
 \begin{array}{l l}
  1 & \quad c \neq 0 \\
  (y,d) & \quad c=0.
 \end{array} \right.
\]
The formula of \ref{SL2 cocycle} can be extended to $\GL_{2}(E) \times \GL_{2}(E)$ by
\begin{equation} \label{GL2 cocycle}
 \beta(g_1, g_2)=\beta(p(g_1)^{\det(g_2)}, p(g_2))v(\det(g_2), p(g_1)).
\end{equation}
It can be verified that $\beta : \GL_{2}(E) \times \GL_{2}(E) \longrightarrow \{ \pm 1 \}$ defined by Equation \ref{GL2 cocycle} is a 2-cocycle. We define $\widetilde{\GL_{2}(E)}$ to be $\GL_{2}(E) \times \{ \pm 1 \}$ as a set, but with the group law given by 
\[
(g_1, \epsilon_1)(g_2, \epsilon_2) = (g_1 g_2, \epsilon_1 \epsilon_2 \beta(g_1, g_2))
\]
for $g_1, g_2 \in \GL_{2}(E)$ and $\epsilon_1, \epsilon_2 \in \{ \pm 1 \}$. This gives the following short exact sequence of groups
\[
 1  \rightarrow \{ \pm 1 \} \rightarrow \widetilde{\GL_2(E)} \xrightarrow{p} \GL_2(E) \longrightarrow  1.
\]
One can verify that the restriction of the cocycle to the subgroup of upper triangular matrices is given as follows:
\begin{equation} \label{cocycle-borel}
\beta \left( \left( \begin{matrix} a_{1} & x \\ 0 & a_{2} \end{matrix} \right), \left( \begin{matrix} b_{1} & y \\ 0 & b_{2} \end{matrix} \right) \right) = (a_{1}, b_{2}).
\end{equation}
For any subset $X$ of $\GL_{2}(E)$, let $\tilde{X}$ be the inverse image of $X$ in $\widetilde{\GL_{2}(E)}$. Let $A=A(E)$ be the group of diagonal matrices in $\GL_{2}(E)$ and $N(E)$ the group of upper triangular unipotent matrices in $\GL_{2}(E)$. Write $B(E)=A(E) \cdot N(E)$ for the group of upper triangular matrices in $\GL_{2}(E)$. Then from Equation \ref{cocycle-borel}, it is clear that the covering of $\GL_{2}(E)$ splits when restricted to $N(E)$, since the cocycle is identically 1. Moreover, $\widetilde{N(E)} = N(E) \times \{ \pm 1 \}$ as groups and hence we will regard $N(E)$ as a subgroup of $\widetilde{\GL_{2}(E)}$ with the obvious splitting. The group $\tilde{A} \cong \widetilde{B(E)}/N(E)$ is not abelian. By the non-degeneracy of the quadratic Hilbert symbol, it follows that the subgroup $\tilde{A}^2 = \left\{ \left( \left( \begin{matrix} a^2 & 0 \\ 0 & b^2 \end{matrix} \right), \epsilon \right): a, b \in E^{\times}, \epsilon \in \{\pm 1 \} \right\}$ of $\tilde{A}$ is the center of $\tilde{A}$. Further, $\tilde{A^2} \cong A^2 \times \{ \pm 1\}$. We have the following short exact sequence of groups
\[
1 \longrightarrow \{ \pm 1 \} \longrightarrow \tilde{A} \longrightarrow A \longrightarrow 1.
\]
The following proposition will make it easier to describe the genuine representations of the group $\tilde{A}$.
\begin{proposition} \label{representation: heisenberg}
Let $G$ be a locally compact topological group with center $Z(G)$ of finite index. Let $Z_{1}(G)$ be a normal abelian subgroup of $G$ containing $Z(G)$ such that $[G : Z(G)] = [Z_{1}(G) : Z(G)]^{2}$. Note that the inner conjugation action of $G$ on $Z_{1}(G)$ induces an action of $G/Z_{1}(G)$ on $\widehat{Z_{1}(G)}$ the group of character of $Z_{1}(G)$. Assume that this action of $G/Z_{1}(G)$ on $\widehat{Z_{1}(G)}$ is transitive on the set of characters of $Z_{1}(G)$ with a given non-trivial restriction on $Z(G)$. Let $\chi$ be a non-trivial character of $Z(G)$, and $\chi_{1}$ a character of $Z_{1}(G)$ with $\chi_{1}|_{Z_{1}(G)} = \chi$. Then $\ind_{Z_{1}(G)}^{G}(\chi_{1})$ is an irreducible representation of $G$. Moreover, an irreducible representation $\pi$ of $G$ with a non-trivial central character $\chi$ is $\ind_{Z_{1}(G)}^{G} (\chi_{1})$ where $\chi_{1}$ is a character of $Z_{1}(G)$ such that $\chi_{1}|_{Z_{1}(G)} = \chi$.
\end{proposition}
\begin{proof}
The proof of the first assertion follows from noting that $\ind_{Z_{1}(G)}^{G} (\chi_{1})$ must be irreducible since any $G$-module containing $\chi_{1}$ must contain $\chi_{1}^{g}$ for all $g \in G$. The second assertion follows from taking any character of $Z_{1}(G)$ appearing in the irreducible representation $\pi$.
\end{proof}
Let $Z=Z(E)$ be the center of $\GL_{2}(E)$. Recall that center of $\tilde{A}$ is $\tilde{A^2}$.
\begin{lemma} \label{tau tilde res 2 Z tilde}
The group $G = \tilde{A}$ with $Z_{1}(G) = \tilde{Z} \tilde{A^2}$ and $\chi$ any genuine character of $\tilde{A^2}$ satisfies the hypothesis in Proposition \ref{representation: heisenberg}, i.e. the induced action of $\tilde{A}/\tilde{Z} \tilde{A^2}$ on $\widehat{\tilde{Z} \tilde{A^2}}$ is transitive on the set of characters of $\tilde{Z} \tilde{A^2}$ whose restriction to $\tilde{A^2}$ is a given genuine character.
\end{lemma}
\begin{proof}
As quadratic Hilbert symbol satisfies $(a, b) = (b, a)$, it can be easily verified that $\tilde{Z} \tilde{A^{2}}$ is a maximal abelian subgroup of $\tilde{A}$. Clearly $[\tilde{A} : \tilde{A^2}] = [\tilde{Z} \tilde{A^2} : \tilde{A^2}]^{2} = [E^{\times} : E^{\times 2}]^{2}$. Let $\chi$ be a genuine character of $\tilde{A^{2}}$. Let $\mu_{1}$ and $\mu_{2}$ be two extensions of $\chi$ to $\tilde{Z} \tilde{A^2}$. Then $\mu_{1} \mu_{2}^{-1}$ is trivial on $\tilde{A^2}$ and hence descends to a quadratic character of $Z$. There are $[E^{\times} : E^{\times 2} ]$ quadratic characters of $E^{\times}$ given by $x \mapsto (x, a)$ where $a \in E^{\times}$ is determined modulo $E^{\times 2}$. So there exists $a \in E^{\times}/E^{\times 2}$ such that $\mu_{2}(\tilde{z}) = (a, z) \mu_{1}(\tilde{z})$ for all $\tilde{z} \in \tilde{Z}$ with $p(\tilde{z}) =z$. As the character $\mu_{1}$ is a genuine character of $\tilde{Z} \tilde{A^2}$, it can be easily verified that $\mu_{2} = \mu_{1}^{g(a)}$ where $g(a)= \left( \begin{matrix} a & 0 \\ 0 & 1 \end{matrix} \right)$ is a representative of $\tilde{A}/\tilde{Z}\tilde{A^{2}} \cong A/ZA^{2}$. Thus the induced action of $\tilde{A}/ \tilde{Z} \tilde{A^2}$ on the set of characters of $\tilde{Z} \tilde{A^2}$ which extend the character $\chi$ of $\tilde{A^2}$ is transitive.
\end{proof}
As $\tilde{A^{2}} \cong A^{2} \times \{ \pm 1 \}$ as groups, genuine characters of $\tilde{A^2}$ are in obvious bijection with characters of $A^{2}$. Then the following corollary is immediate.
\begin{corollary} \label{gen rep of tilde{A}}
The set of genuine irreducible representations of $\tilde{A}$ is parametrized by the set of characters of $A^2$. The dimension of an irreducible genuine representation of $\tilde{A}$ is $[E^{\times} : E^{\times 2}]$. $\hfill \square$
\end{corollary}
 
Hence, any irreducible genuine representation of $\tilde{A}$ can be constructed as follows. Let $\chi_1, \chi_2$ be a pair of characters of $E^{\times}$. Define a character $\chi$ of $\tilde{A}^2$ given by 
\begin{equation} \label{explicit chi} 
\chi\left( \left( \begin{matrix} a^2 & 0 \\ 0 & b^2 \end{matrix} \right), \epsilon \right)=\epsilon \chi_1(a^2) \chi_2(b^2).
\end{equation}
Choose any extension of this character to $\tilde{Z} \tilde{A^2} = \widetilde{ZA^2}$ and denote this extended character by the same letter $\chi$. Let $\tilde{\tau} = \ind_{\widetilde{ZA^{2}}}^{\tilde{A}} (\chi)$. By Proposition \ref{representation: heisenberg}, we know that $\tilde{\tau}$ is irreducible. By the same proposition any irreducible genuine representation of $\tilde{A}$ is of this type. We note that $\tilde{\tau}$ does not depend on the choice of the character of $\tilde{Z} \tilde{A^2}$ which extends the character $\chi$. The following lemma is immediate.

\begin{lemma} \label{restriction to tilde Z}
Let $\tilde{\tau} = \ind_{\widetilde{ZA^2}}^{\tilde{A}} (\chi)$. Then $\tilde{\tau}|_{\tilde{Z}}$ contains all the possible characters $\mu$ of $\tilde{Z}$ such that $\mu|_{\tilde{Z^2}} = \chi|_{\tilde{Z^2}}$. Moreover, $\tilde{\tau}|_{\tilde{Z}}$ is an $[E^{\times} : E^{\times 2}]$ dimensional representation which is a direct sum of distinct characters of $\tilde{Z}$.
\end{lemma}

\section{Representations of $\widetilde{\GL_{2}(E)}$ } \label{Reps of meta GL2}
The first observation about admissible genuine representation of $\widetilde{\GL_{2}(E)}$ is that they are all infinite dimensional. Indeed suppose $(\pi, W)$ is a finite dimensional admissible representation of $\widetilde{\GL_{2}(E)}$. Since $\pi$ is admissible, the kernel of $\pi : \widetilde{\GL_{2}(E)} \rightarrow \GL(W)$ is an open normal subgroup of $\widetilde{\GL_{2}(E)}$. In particular, $\ker(\pi)$ contains $N(E)$, $w N(E) w^{-1}$ and $\widetilde{\SL_{2}(E)}$, where $w = \left( \begin{matrix} 0 & -1 \\ 1 & 0 \end{matrix} \right)$. Thus $\ker(\pi)$ contains $\mu_{2}$ and hence $\pi$ cannot be genuine. \\
We first describe the principal series representations, which are analogous to principal series representations of $\GL_{2}(E)$ \cite{Gelbart76}. Recall that although $\tilde{Z}$ is abelian, it does not lie in the center of $\widetilde{\GL_{2}(E)}$. The center of $\widetilde{\GL_{2}(E)}$ is $\tilde{Z^{2}}$.  Let $(\tilde{\tau}, V)$ be an irreducible genuine representation of $\tilde{A}$. Extend this representation to a representation of $\widetilde{B(E)}$ by defining the action of $N(E)$ on $V$ to be trivial. Then the normalised induction $\Ind_{\widetilde{B(E)}}^{\widetilde{\GL_{2}(E)}}(\tilde{\tau})$ is called a principal series. As in the case of $\GL_{2}(E)$, there is an analogous criterion for the irreducibility of a principal series representation. If a principal series is reducible, it is of length two, and both the Jordan-H{\"o}lder factors are infinite dimensional (as these are genuine representations) unlike for $\GL_{2}(E)$. We recall the criterion of irreducibility of a principal series now. Let $\tilde{\tau} = \Ind_{\widetilde{ZA^2}}^{\tilde{A}}(\chi)$ be an irreducible representation of $\tilde{A}$, where $\chi$ is as given in Equation \ref{explicit chi}. From \cite{Gelbart80}, the principal series representation $\Ind_{\widetilde{B(E)}}^{\widetilde{\GL_{2}(E)}} (\tilde{\tau})$ is irreducible if and only if $\chi_{1}^{2}/\chi_{2}^{2} \neq | \cdot |^{ \pm 1}$, where $\chi_{1}, \chi_{2}$ are characters of $E^{\times}$ satisfying $\chi \left( \left( \begin{matrix} a^2 & 0 \\ 0 & b^2 \end{matrix} \right), \epsilon \right) = \epsilon \chi_{1}(a^2) \chi_{2}(b^2)$ and $| \cdot |$ is the normalised absolute value on $E$. \\

An irreducible admissible genuine representation of $\widetilde{\GL_{2}(E)}$ which is not a Jordan-H{\"o}lder factor of a principal series is called a supercuspidal representation. Thus there are two types of irreducible admissible genuine representations of $\widetilde{\GL_{2}(E)}$, those which arise as  Jordan-H{\"o}lder factors of principal series representations on the one hand supercuspidal representations on the other. \\

Another way to look at an irreducible representation of $\widetilde{\GL_{2}(E)}$ is via a representation of $\widetilde{\SL_{2}(E)}$, which will be useful to us later, e.g. in Section \ref{metaplectic Cass-Prasad} and Chapter \ref{consequences:Waldspurger}. Let 
\[
\GL_{2}(E)_{+} := \{ g \in \GL_{2}(E) : \det(g) \in E^{\times 2} \} = Z \cdot \SL_{2}(E).
\]
\begin{lemma}
 The centralizer of $\tilde{Z}$ in $\widetilde{\GL_{2}(E)}$ is $\widetilde{\GL_{2}(E)}_{+}$.
\end{lemma}
\begin{proof}
 Consider the following map
 \[
\phi :  \tilde{Z} \times \widetilde{\GL_{2}(E)} \longrightarrow \{ \pm 1 \}
 \]
given by
\[
 \phi(\tilde{z}, \tilde{g}) := \tilde{z}\tilde{g}\tilde{z}^{-1}\tilde{g}^{-1}.
\]
Note that this map is a `bi-character' and $\phi(\tilde{z}, \tilde{g})$ depends only on $p(\tilde{z})$ and $p(\tilde{g})$. To prove the proposition we shall prove that the right kernel of the map $\phi$ is $\widetilde{\GL_{2}(E)_{+}}$. \\
{\bf Observation 1: } As $\tilde{Z^{2}}$ is the center of $\widetilde{\GL_{2}(E)}$ so it is the left kernel $\phi$. \\
{\bf Observation 2: } As $\tilde{Z}$ is abelian it lies in the right kernel of the map $\phi$. Moreover, the commutator of $\GL_{2}(E)$ is $\SL_{2}(E)$ so $\widetilde{\SL_{2}(E)}$ also lies in the right kernel of $\phi$. And the group generated by $\tilde{Z}$ and $\widetilde{\SL_{2}(E)}$ is $\widetilde{\GL_{2}(E)}_{+}$. So centralizer of $\tilde{Z}$ in $\widetilde{\GL_{2}(E)}$ contains $\widetilde{\GL_{2}(E)}_{+}$.\\
From these observations we conclude that the map $\phi$ factors through
\[
 \bar{\phi} : \tilde{Z}/\tilde{Z^2} \times \widetilde{\GL_{2}(E)}/\widetilde{\GL_{2}(E)_{+}} \longrightarrow \{ \pm 1 \}.
\]
We now write this map $\bar{\phi}$ more explicitly using self-explanatory notation. 
\[
 \tilde{Z}/\tilde{Z^2} \cong \left\{ \left( \begin{matrix} a & 0 \\ 0 & a \end{matrix} \right) : a \in E^{\times}/E^{\times 2} \right\}, \hspace{2mm}
 \widetilde{\GL_{2}(E)}/\widetilde{\GL_{2}(E)_{+}} \cong \left\{ \left( \begin{matrix} 1 & 0 \\ 0 & a \end{matrix} \right) : a \in E^{\times}/E^{\times 2} \right\}
\]
Both the sets involved in $\bar{\phi}$ are isomorphic to $E^{\times}/E^{\times 2}$ and using the description of the Kubota cocycle $\beta$ on diagonal elements we get 
\[
 \bar{\phi} \left( \left( \begin{matrix} a & 0 \\ 0 & a \end{matrix} \right), \left( \begin{matrix} 1 & 0 \\ 0 & b \end{matrix} \right) \right) = (a,b).
\]
From the non-degeneracy of the quadratic Hilbert symbol, it follows that the right kernel of the map $\phi$ is $\widetilde{\GL_{2}(E)_{+}}$.
\end{proof}
We have that $\widetilde{\GL_{2}(E)_{+}} = \tilde{Z} \cdot \widetilde{\SL_{2}(E)}$ and that the center of $\widetilde{\GL_{2}(E)_{+}}$ is $\tilde{Z}$. Note that $\tilde{Z} \cap \widetilde{\SL_{2}(E)} = \widetilde{ \{ \pm 1 \} }$, which is the center of $\widetilde{\SL_{2}(E)}$ and the index $[ \widetilde{\GL_{2}(E)} : \widetilde{\GL_{2}(E)_{+}} ] = [E^{\times} : E^{\times 2} ] < \infty$.

\begin{definition}
Let $\tau$ be an irreducible admissible genuine representation of $\widetilde{\SL_{2}(E)}$ and $\mu$ a genuine character of $\tilde{Z}$. We say that $\mu$ and $\tau$ are compatible if the central character of $\tau$ (i.e. $\tau$ restricted to $\widetilde{ \{ \pm 1 \}}$) is the same as $\mu|_{\widetilde{ \{ \pm 1 \} }}$. 
\end{definition}
If $\mu$ and $\tau$ are compatible, we can define an irreducible representation of $\widetilde{\GL_{2}(E)_{+}}$ on the space of $\tau$ with central character $\mu$ and on which $\widetilde{\SL_{2}(E)}$ acts by $\tau$. Denote this representation by $\mu \tau$ and consider
\begin{equation} \label{induction from SL2 to GL2}
\pi := \ind_{\widetilde{\GL_{2}(E)_{+}}}^{\widetilde{\GL_{2}(E)}}(\mu \tau).
\end{equation}
For $a \in E^{\times}$, let $\mu^{a}$ denote the genuine character of $\tilde{Z}$ defined by
\begin{equation}
\mu^{a}(x, \epsilon):= (x,a) \mu(x, \epsilon) \hspace{1cm} \forall x \in E^{\times}, \epsilon \in \{ \pm 1 \} .
\end{equation}
By the commutation relation in $\tilde{A}$, it follows that conjugation by $diag(a, 1) \in \GL_{2}(E)$ on a genuine character $\mu$ of $\tilde{Z}$ takes $\mu$ to $\mu^{a}$. By non-degeneracy of the quadratic Hilbert symbol, if $a$ represents a non-trivial coset of $E^{\times}/E^{\times 2}$, then $x \mapsto (x,a)$  is a non-trivial character of $E^{\times}$. It follows that $\mu = \mu^{a}$ if and only if $ a \in E^{\times 2}$. One may choose the representatives of the quotient $\widetilde{{\rm GL}_{2}(E)} / \widetilde{{\rm GL}_{2}(E)}_{+}$ to be $g(a) := \left( \left( \begin{matrix} a & 0 \\ 0 & 1 \end{matrix} \right), 1 \right)$ for $a \in E^{\times}$ representing cosets of $E^{\times}/E^{\times 2}$. If we write $(\mu \tau)^{a}$ for the conjugate representation of $\mu \tau$ by the element $g(a)$, then it follows that 
\begin{equation}
\mu \tau \ncong (\mu \tau)^{g(a)} \cong \mu^{a} \tau^{g(a)}
\end{equation}
as representations of $\widetilde{{\rm GL}_{2}(E)}_{+}$ if $a \notin E^{\times}$, since the central characters of $\mu \tau$ and $(\mu \tau)^{g(a)}$, namely $\mu$ and $\mu^{a}$, are different. By Clifford theory, the representation $\pi$ of $\widetilde{{\rm GL}_{2}(E)}$ defined by equation \ref{induction from SL2 to GL2} is irreducible. Moreover, for all $a \in E^{\times}$ we have 
\begin{equation}
\pi := \ind_{\widetilde{{\rm GL}_{2}(E)}_{+}}^{\widetilde{{\rm GL}_{2}(E)}} (\mu \tau) \cong \ind_{\widetilde{{\rm GL}_{2}(E)}_{+}}^{\widetilde{{\rm GL}_{2}(E)}} (\mu \tau)^{a} 	\label{restriction:1}
\end{equation} 
and 
\begin{equation} 
\pi|_{\widetilde{{\rm GL}_{2}(E)}_{+}} \cong \bigoplus_{a \in E^{\times}/E^{\times 2}} (\mu \tau)^{a} \label{restriction:2}
\end{equation} 
and 
\begin{equation}
\pi|_{\widetilde{{\rm SL}_{2}(E)}} \cong \bigoplus_{a \in E^{\times}/E^{\times 2}} \tau^{a}. \label{restriction:3}
\end{equation} 
Conversely, using Frobenius reciprocity and the fact that there exists an irreducible $\widetilde{\SL_{2}(E)}$-subrepresentation of an irreducible  admissible genuine representation of $\widetilde{\GL_{2}(E)}$, it is easy to prove that any irreducible admissible genuine representation of $\widetilde{{\rm GL}_{2}(E)}$ arises as in Equation \ref{induction from SL2 to GL2} for some choice of $\mu$ and $\tau$.  
\begin{remark}
From the analysis above, it follows that $\pi$ restricted to $\widetilde{\GL_{2}(E)_{+}}$ is multiplicity free. Later we will see in Section \ref{higher multiplicity} that the restriction of $\pi$ to $\widetilde{\SL_{2}(E)}$ may not be multiplicity free.
\end{remark}

\section{Whittaker functionals for $\widetilde{\GL_{2}(E)}$}
Let $\psi$ be a non-trivial character of $E$. Following \cite{GHP79}, we recall the definition of a $\psi$-Whittaker functional of a representation $(\pi, W)$ of any of $\widetilde{\GL_{2}(E)}, \GL_{2}(E), \widetilde{\SL_{2}(E)}$ or $\SL_{2}(E)$. We identify $N(E)$ with $E$ as a topological group in obvious way.
\begin{definition}
A linear functional $\Lambda : W \longrightarrow \mathbb{C}$ is called a $\psi$-Whittaker functional if it satisfies the following:
\begin{equation} \label{whittaker functional}
\Lambda \left( \pi \left( \begin{matrix} 1 & n \\ 0 & 1 \end{matrix} \right)v \right) = \psi(n) \Lambda(v), \forall \, n \in E \text{ and } v \in V.
\end{equation} 
The representation $(\pi, W)$ is called $\psi$-generic if admits a non-zero $\psi$-Whittaker functional.
\end{definition}
It is known that an irreducible admissible infinite dimensional representation of any of $\widetilde{\GL_{2}(E)}, \GL_{2}(E), \widetilde{\SL_{2}(E)}$ or $\SL_{2}(E)$ is $\psi$-generic for {\it some} non-trivial character $\psi$ of $E$. Moreover, all infinite dimensional irreducible admissible representations of $\widetilde{{\rm GL}_2(E)}$ or $\GL_{2}(E)$ are $\psi$-generic for {\it any} non-trivial character $\psi$, see \cite{GHP79}. In particular, genuine representations of $\widetilde{{\rm {\rm GL}}_{2}(E)}$ are $\psi$-generic for any non-trivial character $\psi$. Recall that the space of Whittaker functionals for an irreducible admissible infinite dimensional representation of ${\rm GL}_{2}(E)$ is one dimensional, an assertion which is known as the uniqueness of Whittaker model. But this need not be true for an irreducible genuine representation of $\widetilde{{\rm GL}_{2}(E)}$. \\
%However, we have uniqueness if we fix a character of $\tilde{Z}$ too.  

Let $\omega_{\pi}$ denote the central character of an irreducible genuine representation $(\pi, W)$ of $\widetilde{\GL_{2}(E)}$. For a genuine character $\chi$ of $\tilde{Z^2}$, we define a $\tilde{Z}$-module $\Omega(\chi)$ on which $\tilde{Z^2}$ acts by $\chi$ and any genuine character $\mu$ of $\tilde{Z}$ with $\mu|_{\tilde{Z^2}} = \chi$ appears in $\Omega(\chi)$ with multiplicity one. We abuse the notation and write $\mu \in \Omega(\chi)$ if $\mu$ appears in $\Omega(\chi)$, i.e. $\Hom_{\tilde{Z}}(\Omega(\chi), \mu) =1$. Let $\mathcal{L}$ be the space of all $\psi$-Whittaker functionals for $(\pi, W)$. Then $\tilde{Z}$ has a natural action on $\mathcal{L}$ given by $(\tilde{z} \cdot \Lambda) (v) := \Lambda(\pi(\tilde{z})v)$. As the action of $\tilde{Z^2}$ on $\mathcal{L}$ is by $\omega_{\pi}$, a character of $\tilde{Z}$ appearing in $\mathcal{L}$ belongs to $\Omega(\omega_{\pi})$. For $\mu \in \Omega(\omega_{\pi})$, let $\mathcal{L}_{\mu} := \{ \Lambda \in \mathcal{L} \mid \tilde{z} \cdot \Lambda = \mu(\tilde{z}) \Lambda, \, \forall \tilde{z} \in \tilde{Z} \}$. Call $\mathcal{L}_{\mu}$ the space of $(\psi, \mu)$-Whittaker functionals. 
\begin{theorem} {\rm \cite[~Theorem 4.1]{GHP79}} \label{GHP79:uniquness}
For an irreducible admissible genuine representation $\pi$ of $\widetilde{\GL_{2}(E)}$, we have $\dim \mathcal{L}_{\mu} \leq 1$ for all $\mu \in \Omega(\omega_{\pi})$. 
\end{theorem}
\begin{definition}
Let $N$ be a group, $\pi$ a representation of $N$ and $\psi$ a character of $N$. Let $\pi(N, \psi)$ be the vector space spanned by $\{ \pi(n)v - \psi(n)v \mid n \in N \text{ and } v \in \pi \}$. Then $\pi_{N, \psi} := \pi / \pi(N, \psi)$ is called $\psi$-twisted Jacquet module of $\pi$. If $\psi =1$ then we write $\pi_{N}$ for $\pi_{N, \psi}$ and call it the Jacquet module of $\pi$. 
\end{definition}
If $B$ is a group, $N \subset B$ a normal subgroup and $\pi$ a representation of $B$ then $\pi_{N}$ has an induced action of $B/N$ hence is a $B/N$-module . Thus $\pi \mapsto \pi_{N}$ defines a functor from the category of $B$-modules to the category of $B/N$-modules. For a non-trivial character $\psi$ of $N$, $\pi_{N, \psi}$ has an induced action of ${\rm Norm}(N, \psi)/N$, where ${\rm Norm}(N, \psi) = \{ b \in B \mid \psi(bnb^{-1}) = \psi(n) \}$ and hence $\pi \mapsto \pi_{N, \psi}$ defines a functor from the category of $B$-modules to the category of ${\rm Norm}(N, \psi)/N$-modules.\\

Note that, $\mathcal{L}$, as a vector space, is dual of $\pi_{N(E), \psi}$. From Theorem \ref{GHP79:uniquness} it follows that the multiplicity of a character $\mu \in \Omega(\omega_{\pi})$ in $\pi_{N(E), \psi}$ is at most one, i.e. $\dim \Hom_{\tilde{Z}}(\pi_{N(E), \psi}, \mu) \leq 1$. As a $\tilde{Z}$-module we have 
\[
\pi_{N(E), \psi} \subset \Omega(\omega_{\pi}).
\]
\begin{proposition} \label{whittaker models of principal series}
Let $\pi$ be a principal series representation of $\widetilde{{\rm GL}_2(E)}$ with central character $\omega_{\pi} : \tilde{Z^2} \rightarrow \C^{\times}$. Let $\psi$ be a non-trivial additive character of $E$. Then all the character of $\tilde{Z}$ which extend $\omega_{\pi}$ appear in $\pi_{N(E), \psi}$, i.e. as a $\tilde{Z}$-module 
\[
 \pi_{N(E),\psi} \cong \Omega(\omega_{\pi}).
\]
\end{proposition}
We prove this proposition in the next few lemmas. 
\begin{lemma}
Let $N(E)^{-}$ be the group of lower triangular unipotent matrices. Let $V_0$ be the subspace of functions in the space of the principal series representation $V(\tilde{\tau}) = \Ind_{\widetilde{B(E)}}^{\widetilde{{\rm GL}_2(E)}}(\tilde{\tau})$ which have compact support when restricted to $N(E)^-$, where $(\tilde{\tau}, V)$ is an irreducible genuine representation of $\tilde{A}$. Then $V_0$ is of finite codimension in $V(\tilde{\tau})$. Moreover, we have $V(\tilde{\tau})/V_0 \cong V$. On this quotient space $V$ the induced action of $N(E)^-$ is trivial and the induced action of $\tilde{A}$ is $\tilde{\tau}^{w}$.
\end{lemma}
\begin{proof}
Recall that $V(\tilde{\tau})$ is space of $V$ valued functions $f$ on $\widetilde{{\rm GL}_2(E)}$ which are locally constant and satisfy $f(\tilde{b}g) = \tilde{\tau}(\tilde{b})f(g)$ for all $\tilde{b} \in \widetilde{B(E)}$ and $g \in \widetilde{{\rm GL}_2(E)}$. Because of the Bruhat decomposition $\widetilde{\GL_{2}(E)} = w \widetilde{B(E)} \sqcup N^{-}(E) \widetilde{B(E)}$, a function $f \in V(\tilde{\tau})$ is determined by its values on $w = \left( \begin{matrix} 0 & -1 \\ 1 & 0 \end{matrix} \right)$ and on the set $N(E)^-$. Define the evaluation map at $w$, $e : V(\tilde{\tau}) \longrightarrow V$, by $f \mapsto f(w)$.
%Now we claim that a function $f$, which is in the kernel of $e$, has compact support in $N(E)^-$. Let $f \in ker(e)$, then $f(w)=0$. Because $f$ is locally constant there is an open neighbourhood $U$ around $w$ such that $f(x)=0$ for all $x \in U$. Then there is an $\epsilon > 0$ such that $f\left( \begin{matrix} 0 & -1 \\ 1 & y \end{matrix} \right) =0$ for all $|y| < \epsilon$. But for $y \neq 0$ we have
%\[
% \left( \begin{matrix} 0 & -1 \\ 1 & y \end{matrix} \right) = \left( \begin{matrix} -y^{-1} & 1 \\ 0 & y \end{matrix} \right) \left( \begin{matrix} 1 & 0 \\ y^{-1} & 1 \end{matrix} \right).
%\]
%Therefore $f\left( \begin{matrix} 1 & 0 \\ y^{-1} & 1 \end{matrix} \right) =0$ for all y such that $|y| < \epsilon$. Hence $f|_{N^-}$ is 0 outside a compact set. Moreover if a function $f$ in the principal series is such that $f(w) \neq 0$ then by the same argument it does not have compact support in $N(E)^-$. Thus we get $V_0 = ker(e)$. As $V$ is finite dimensional hence the space $V(\tilde{\tau})/V_0$ is finite dimensional. Observe that $V_0$ is $N(E)^-$ invariant as well as $\tilde{Z}$ invariant. 
It is easy to verify that $V_{0} = \ker (e)$. Note that $V_{0}$ is stable under the action of $\widetilde{B(E)^{-}} := \tilde{A}N(E)^{-}$, so we have the following short exact sequence on $\tilde{A}N(E)^-$-modules
\[
 0 \longrightarrow V_0 \longrightarrow V(\tilde{\tau}) \longrightarrow V(\tilde{\tau})/V_0 \cong V \longrightarrow 0.
\]
The induced action of $N(E)^{-}$ is trivial on the quotient $V$. For $\tilde{a} \in \tilde{A}$ we have
\[
(\pi(\tilde{a})f)(w) = f(w \tilde{a}) = f(w \tilde{a} w^{-1} w) = \tilde{\tau}(w \tilde{a} w^{-1})f(w) = \tilde{\tau}^{w}(\tilde{a}) f(w)
\]
proving that the action of $\tilde{A}$ on the quotient $V$ is same as $\tilde{\tau}^{w}$.
\end{proof}
The next two lemmas are immediate and these will complete the proof of Proposition \ref{whittaker models of principal series}.
\begin{lemma}
Following the notation of the above lemma, $V_{0}$ and $V(\tilde{\tau})$ are genuine $\widetilde{B(E)^{-}}$-modules. Let $\psi$ be a non-trivial character of $E$. Then
\[
V(\tilde{\tau})_{N(E)^{-}, \psi} \cong (V_{0})_{N(E)^{-}, \psi} \cong \Omega(\omega_{\pi}).
\]
\end{lemma}
%\begin{proof}
%Since the action of $N(E)^{-}$ on $V$ is trivial, one has $V_{N(E)^{-}, \psi} = 0$ and hence the lemma follows from exactness of the $\psi$-twisted Jacquet functor.
%\end{proof}
\begin{lemma}
If $\psi_{-1}$ is given by $x \mapsto \psi(-x)$ then as $\tilde{Z}$-modules we have
\[
(V_{0})_{N(E), \psi_{-1}} \cong (V_{0})_{N^{-}(E), \psi}.
\]
\end{lemma}
%\begin{proof}
%As $N^{-}(E)$ action on $V_{0}$ is by translation, we have $V_{0} \cong \ind_{ \{1\} }^{N^{-}(E)} (\tilde{\tau})$ as an $N^{-}(E)$-module where $\tilde{\tau}$ is considered to be trivial $N^{-}(E)$-module. By Frobenius reciprocity, we have 
%\[
%(V_{0})_{N^{-}(E), \psi} = \Hom_{N^{-}(E)}( \ind_{ \{1\}}^{N^{-}(E)}(\tilde{\tau}), \psi) = \Hom_{\{1\}}(\tilde{\tau}, \psi) = \tilde{\tau}. 
%\]
%Thus as $\tilde{Z}$-module, $(V_{0})_{N^{-}(E), \psi} \cong \tilde{\tau}$.
%This proves the lemma.
%\end{proof}

\section{The Jacquet module and the Kirillov model for $\widetilde{\GL_{2}(E)}$} 

\subsection{The Jacquet module with respect to $N(E)$}
Let $\pi$ be an irreducible admissible genuine representation of $\widetilde{\GL_{2}(E)}$. We will describe $\pi_{N(E)}$ in this section. It is well known that $\pi_{N(E)} = 0$ if and only if $\pi$ is a supercuspidal representation, i.e. it does not appear as a subquotient of a principal series representation. So we need to consider only those representations which arise as  Jordan-H{\"o}lder factors of principal series representations. Let $\pi = \ind_{\widetilde{B(E)}}^{\widetilde{\GL_{2}(E)}} (\tilde{\tau})$ be a principal series representation. As $\widetilde{\GL_{2}(E)} = \widetilde{B(E)} \sqcup \widetilde{B(E)}w\widetilde{B(E)}$ and $\widetilde{B(E)}w\widetilde{B(E)}$ is open in $\widetilde{\GL_{2}(E)}$, we have the following filtration of $\widetilde{B(E)}$-modules $
0 \subsetneq \pi_{w} \subsetneq \pi$, where $\pi_{w}$ is space of functions supported on $\widetilde{B(E)}w\widetilde{B(E)}$. This gives us the following filtration of the Jacquet modules $\pi_{N(E)}$ (as $\tilde{A}$-modules): 
\begin{equation} \label{ses of JM of PS}
0 \longrightarrow (\tilde{\tau})^{w} \cdot \delta^{1/2} \longrightarrow \pi_{N(E)} \longrightarrow \tilde{\tau} \cdot \delta^{1/2} \longrightarrow 0.
\end{equation}
Both the Jordan-H{\"o}lder factors are genuine representations of $\tilde{A}$ of dimension $\dim(\tilde{\tau})$. By Lemma \ref{gen rep of tilde{A}}, both are irreducible $\tilde{A}$-modules. Its semi-simplification $\pi_{N(E)}^{ss}$ equals $\tilde{\tau} \cdot \delta^{1/2} \oplus \tilde{\tau}^{w} \cdot \delta^{1/2}$. Note that $\tilde{\tau}$ is determined by its restriction to $\tilde{A^2}$, i.e. a pair $(\chi_{1}^{2}, \chi_{2}^{2})$, where $\chi_{1}, \chi_{2}$ are characters of $E^{\times}$. The restriction of $\tilde{\tau}^{w}$ to $\tilde{A^2}$ is $(\chi_{2}^{2}, \chi_{1}^{2})$. The two Jordan-H{\"o}lder factors are isomorphic to each other if and only if $\chi_{1}^{2} = \chi_{2}^{2}$. So the short exact sequence of $\tilde{A}$-modules in Equation \ref{ses of JM of PS} splits whenever $\chi_{1}^{2} \neq \chi_{2}^{2}$. In particular, when $\pi$ is a reducible principal series representation, the short exact sequence in Equation \ref{ses of JM of PS} splits as $\chi_{1}^{2}/\chi_{2}^{2} = \mid \cdot \mid^{ \pm 1}$. \\
If $\pi$ is an irreducible principal series, then we know its Jacquet module $\pi_{N(E)}$ in the sense that we know its Jordan-H{\"o}lder factors. Moreove $\pi_{N(E)}$ is of length two as $\tilde{A}$-module. Let us assume that the principal series representation $\pi$ is reducible and its Jordan-H{\"o}lder factors are $\pi_{1}$ and $\pi_{2}$ giving rise to the following exact sequence of $\widetilde{\GL_{2}(E)}$-modules
\[
0 \longrightarrow \pi_{1} \longrightarrow \pi \longrightarrow \pi_{2} \longrightarrow 0.
\]
As the Jacquet functor is exact \cite[~Proposition 2.35]{BerZel76}, we get the following short exact sequence of $\tilde{A}$-modules
\[
0 \longrightarrow (\pi_{1})_{N(E)} \longrightarrow \pi_{N(E)} \longrightarrow (\pi_{2})_{N(E)} \longrightarrow 0.
\]
As we know that $(\pi_{1})_{N(E)}$ and $(\pi_{2})_{N(E)}$ are non-zero, one of these is $\tilde{\tau} \cdot \delta^{1/2}$ and the other is $\tilde{\tau}^{w} \cdot \delta^{1/2}$. As $\pi_{1}$ is a subrepresentation of $\pi$, by Frobenius reciprocity we have
\[
\Hom_{\widetilde{\GL_{2}(E)}} (\pi_{1}, \pi) = \Hom_{\tilde{A}} ((\pi_{1})_{N(E)}, \tilde{\tau} \cdot \delta^{1/2}),
\]
therefore $(\pi_{1})_{N(E)} \cong \tilde{\tau} \cdot \delta^{1/2}$ and hence $(\pi_{2})_{N(E)} = \tilde{\tau}^{w} \cdot \delta^{1/2}$.

\subsection{The Kirillov model}
Now we describe the Kirillov model \index{Kirillov model} of an irreducible admissible genuine representation $\pi$ of $\widetilde{\GL_{2}(E)}$ \cite{Gelbart80}.
Recall $\pi_{N(E), \psi} = \pi / \pi(N(E), \psi)$. Let $l : \pi \rightarrow \pi_{N(E), \psi}$ be the canonical map. Let $\mathcal{C}^{\infty}(E^{\times}, \pi_{N(E), \psi})$ denote the space of smooth functions on $E^{\times}$ with values in $\pi_{N(E), \psi}$. Define the Kirillov mapping
\[
\mathtt{K} : \pi \longrightarrow \mathcal{C}^{\infty}(E^{\times}, \pi_{N(E), \psi})
\]
given by $v \mapsto \xi_{v}$ where $\xi_{v}(x) = l \left( \pi \left( \left( \begin{matrix} x & 0 \\ 0 & 1 \end{matrix} \right), 1 \right) v \right)$. We summarize some of the properties of the Kirillov mapping in the following proposition. 
\begin{proposition} \label{prop:Kirillov model}
\begin{enumerate}
\item If $v' = \pi \left( \left( \begin{matrix} a & b \\ 0 & d \end{matrix} \right), 1 \right) v$, then
\[
\xi_{v'}(x) = (x, d) \psi(bd^{-1}x) \pi \left( \left( \begin{matrix} d & 0 \\ 0 & d \end{matrix} \right), 1 \right) \xi_{v}(ad^{-1}x).
\]
\item For $v \in W$ the function $\xi_{v}$ is a locally constant function on $E^{\times}$ which vanishes outside a compact subset of $E$.
\item The map $\mathtt{K}$ is an injective linear map. 
\item The image $\mathtt{K}(\pi)$ of the map $\mathtt{K}$ contains the space $\mathcal{S}(E^{\times}, \pi_{N(E), \psi})$ of smooth functions with compact support in $E^{\times}$.
\item The Jacquet module $\pi_{N(E)}$ of $\pi$ is isomorphic to $\mathtt{K}(\pi)/\mathcal{S}(E^{\times}, \pi_{N(E), \psi})$.
\item The representation $\pi$ is supercuspidal if and only if $\mathtt{K}(\pi) = \mathcal{S}(E^{\times}, \pi_{N(E), \psi})$.
\end{enumerate}
\end{proposition}
\begin{proof}
Part 1 follows from the definition. The proofs of part 2 and 3 are verbatim those of Lemma 2 and Lemma 3 in \cite{God70}. The proofs of part 4, 5 and 6 follow from the proofs of the corresponding statements of \cite[~Theorem 3.1]{PrRa2000}.
\end{proof}
Since the map $\mathtt{K}$ is injective, we can transfer the action of $\widetilde{\GL_{2}(E)}$ on $W$ (via $\pi$) to $\mathtt{K}(\pi)$ using the map $\mathtt{K}$. The realization of $(\pi, W)$ on $\mathtt{K}(\pi)$ is called the Kirillov model, on which the action of $\widetilde{B(E)}$ is explicitly given by part 1 in Proposition \ref{prop:Kirillov model}. It is clear that $\mathcal{S}(E^{\times}, \pi_{N(E), \psi})$ is $\widetilde{B(E)}$ stable, which gives rise to the following short exact sequence of $\widetilde{B(E)}$-modules
\begin{equation} \label{Kirillov ses}
0 \rightarrow \mathcal{S}(E^{\times}, \pi_{N(E), \psi}) \rightarrow \mathtt{K}(\pi) \rightarrow \pi_{N(E)} \rightarrow 0.
\end{equation}

\subsection{The Jacquet module with respect to $N(F)$}
Now restrict an irreducible admissible genuine representation $\pi$ of $\widetilde{\GL_{2}(E)}$ to $B(F)$. $N(F) \subset B(F)$ is a normal subgroup. To simplify notation we write $N$ for $N(F)$ in the rest of this section. We describe the Jacquet module $\pi_{N}$ of $\pi$, which we will need in Chapter 5. We utilize the short exact sequence in Equation \ref{Kirillov ses} of $\widetilde{B(E)}$-modules arising from the Kirillov model of $\pi$, which is also a short exact sequence of $B(F)$-modules. By the exactness of the Jacquet functor with respect to $N$, we get the following short exact sequence from Equation \ref{Kirillov ses}, 
\[
0 \rightarrow \mathcal{S}(E^{\times}, \pi_{N(E), \psi})_{N} \rightarrow \mathtt{K}(\pi)_{N} \rightarrow (\pi_{N(E), \psi})_{N} ( \cong \pi_{N(E)}) \rightarrow 0.
\]
Let us first describe $\mathcal{S}(E^{\times}, \pi_{N(E), \psi})_{N}$, the Jacquet module of $\mathcal{S}(E^{\times}, \pi_{N(E), \psi})$ with respect to $N=N(F)$. Let $\mathcal{S}(F^{\times}, \pi_{N(E), \psi})$ be the space of locally constant functions with compact support from $F^{\times}$ with values in $\pi_{N(E), \psi}$ be trivial $N(F)$-module.  
\begin{proposition} \label{restriction jacquet}
$\mathcal{S}(E^{\times}, \pi_{N, \psi})_{N} \cong \mathcal{S}(F^{\times}, \pi_{N(E), \psi})$.
\end{proposition}
The Proposition \ref{restriction jacquet} follows from the proposition below. The author thanks Professor D. Prasad for suggesting the proof below. 
%The proof below was indicated to the author by Professor D. Prasad.
\begin{proposition} 
Let $\mathcal{S}(E^{\times})$ be a representation space for $N \cong E$ with the action of $N$ given by $(n \cdot f)(x) = \psi(nx) f(x)$ for all $x \in E^{\times}$ where $\psi$ is a non-trivial additive character of $E$ such that $\psi|_{F} =1$. Then the restriction map
\begin{equation} \label{restriction map}
\mathcal{S}(E^{\times}) \longrightarrow \mathcal{S}(F^{\times})
\end{equation}
gives the Jacquet module, i.e. the above map realizes $\mathcal{S}(E^{\times})_{N}$ as $\mathcal{S}(F^{\times})$.
\end{proposition}
\begin{proof}
Note that $\mathcal{S}(E^{\times}) \hookrightarrow \mathcal{S}(E)$. For a fixed Haar measure $dw$ on $E$, we define the Fourier transform $\mathcal{F}_{\psi} : \mathcal{S}(E) \rightarrow \mathcal{S}(E)$ with respect to the character $\psi$ by
\[
\mathcal{F}_{\psi}(f)(z) := \int_{E} f(w) \psi(zw) \, dw.
\]
$\mathcal{F}_{\psi}$ is an isomorphism of vector spaces and image of $\mathcal{S}(E^{\times})$ can be identified with those functions whose integral on $E$ is zero. The Fourier transform takes the action of $N(E)$ on $\mathcal{S}(E^{\times})$ to the restriction of the action of $N(E)$ on $\mathcal{S}(E)$ given by $(n \cdot f) (x) = f(x+n)$. Here we have identified $N(E)$ with $E$. Thus the maximal quotient of $\mathcal{S}(E)$ on which $N(F)$ acts trivially can be identified with $\mathcal{S}(F)$ by integrating along the fibres (defined below) of the mapping $\phi : E \rightarrow F$ given by $\phi(e) = \frac{e - \bar{e}}{2 \sqrt{d}}$ if $E = F(\sqrt{d})$. Note that $\phi(z+x) = \phi(z)$ for all $z \in E$ and $x \in F$. We define the integration along the fibres of the map $\phi$, $I : \mathcal{S}(E) \rightarrow \mathcal{S}(F)$ as follows:
\[
I(f)(y) := \int_{F} f(x + \sqrt{d} y) \, dx \text{ for all } y \in F.
\] 
It can be checked that $I(f)$ belongs to $\mathcal{S}(F)$. Note that $\psi_{\sqrt{d}} = \psi_{\sqrt{d}}|_{F} : x \mapsto \psi(\sqrt{d}x)$ is a non-trivial character of $F$. The proposition will follow if we prove the commutativity of the following diagram:
\begin{displaymath}
\xymatrix{ 
\mathcal{S}(E) \ar[r]^{\mathcal{F}_{\psi}} \ar[d]_{{\rm Res}} & \mathcal{S}(E) \ar[d]^{I} \\
\mathcal{S}(F) \ar[r]^{\mathcal{F}_{\psi_{\sqrt{d}}}} & \mathcal{S}(F)   }
\end{displaymath} 
where $\mathcal{F}_{\psi}$ (respectively, $\mathcal{F}_{\psi_{\sqrt{d}}}$) is the Fourier transform on $\mathcal{S}(E)$ (respectively, $\mathcal{S}(F)$) with respect to the character $\psi$ (respectively, $\psi_{\sqrt{d}} =(\psi_{\sqrt{d}})|_{F}$), ${\rm Res}$ denotes the restriction mapping and $I$ denote the integration along the fibres mentioned above. $\mathcal{F}_{\psi_{\sqrt{d}}} : \mathcal{S}(F) \rightarrow \mathcal{S}(F)$ is defined by $\mathcal{F}_{\psi_{\sqrt{d}}} (\phi) (x) := \int_{F} \phi(y) \psi_{\sqrt{d}}(xy) dy$ for all $x \in F$.  We claim that the above diagram is commutative. Let $f \in \mathcal{S}(E)$. We want to show that $I \circ \mathcal{F}_{\psi} (f) (y) = \mathcal{F}_{\psi_{\sqrt{d}}} \circ {\rm Res}(f) (y)$ for all $y \in F$. We write an element of $E$ as $x + \sqrt{d}y$ with $x, y \in F$. We choose a measure $dx$ on $F$ which is self dual with respect to $\psi_{\sqrt{d}}$ in the sense that $\mathcal{F}_{\psi_{\sqrt{d}}} (\mathcal{F}_{\psi_{\sqrt{d}}}(\phi)) (x) = \phi(-x)$ for all $\phi \in \mathcal{S}(F)$ and $x \in F$. We identify $E$ with $F \times F$ as vector space. Consider the product measure $dx \, dy$ on $E = F \times F$. Then using Fubini's theorem and $\int_{F} \int_{F} \phi(z_{2}) \psi_{\sqrt{d}}(xz_{2}) dz_{2} \, dx = \mathcal{F}_{\psi_{\sqrt{d}}} ( \mathcal{F}_{\psi_{\sqrt{d}}} (\phi))(0) = \phi(0)$  for $\phi \in \mathcal{S}(F)$, we get the following:
\begin{displaymath}
\begin{array}{lcl}
I \circ \mathcal{F}_{\psi} (f) (y) & = & \int_{F} \mathcal{F}_{\psi}(f)(x+\sqrt{d}y)  dx\\
                                               & = & \int_{F} \int_{E=F \times F} f(z_{1} + \sqrt{d} z_{2}) \psi((x + \sqrt{d}y)(z_{1}+ \sqrt{d}z_{2})) dz_{1} \, dz_{2} \, dx \\
                                               & =& \int_{F}  \int_{F} \int_{F} f(z_{1} + \sqrt{d} z_{2}) \psi_{\sqrt{d}} (yz_{1} + x z_{2}) dz_{1} \, dz_{2} \, dx \\
                                               & =& \int_{F} \left( \int_{F} \int_{F} f(z_{1} + \sqrt{d} z_{2}) \psi_{\sqrt{d}} (xz_{2}) dz_{2} \, dx \right) \psi_{\sqrt{d}}(yz_{1}) dz_{1} \\
                                               &=& \int_{F} f(z_{1}) \psi_{\sqrt{d}}(yz_{1}) dz_{1}  \\
                                               &=&  \mathcal{F}_{\psi_{\sqrt{d}}} \circ {\rm Res}(f) (y).
\end{array} 
\end{displaymath}  
This proves the commutativity of the above diagram. 
\end{proof}

\chapter{Splitting questions} \label{splitting questions}
\section{Introduction}
Let $E$ be a non-Archimedian local field. This chapter will be concerned with a specific 2-fold covers of ${\rm GL}_{2}(E)$, to be called the metaplectic covering of ${\rm GL}_{2}(E)$, which was defined in Section \ref{definition: 2-fold cover}. We recall that there is a unique (up to isomorphism) non-trivial 2-fold cover of $\SL_{2}(E)$ called the metaplectic cover and denoted by $\widetilde{\SL_{2}(E)}$, but there are many inequivalent 2-fold coverings of ${\rm GL}_{2}(E)$ which extend this 2-fold covering of ${\rm SL}_{2}(E)$.  The covering $\widetilde{\GL_{2}(E)}$ of $\GL_{2}(E)$ can be described as follows. Observe that ${\rm GL}_{2}(E)$ is the semi-direct product of ${\rm SL}_{2}(E)$ and $E^{\times}$, where $E^{\times}$ sits inside ${\rm GL}_{2}(E)$ as $e \mapsto \left( \begin{matrix} e & 0 \\ 0 & 1 \end{matrix} \right)$. This action of $E^{\times}$ on ${\rm SL}_{2}(E)$ lifts uniquely to an action of $E^{\times}$ on  $\widetilde{{\rm SL}_{2}(E)}$. The group $\widetilde{{\rm SL}_{2}(E)}  \rtimes E^{\times}$ is the metaplectic cover $\widetilde{\GL_{2}(E)}$ of ${\rm GL}_{2}(E)$. Thus the metaplectic cover of ${\rm GL}_{2}(E)$ that we consider in this chapter is that cover of ${\rm GL}_{2}(E)$ which extends the metaplectic cover of ${\rm SL}_{2}(E)$ and is further split on the subgroup $\left\{ \left( \begin{matrix} e & 0 \\ 0 & 1 \end{matrix} \right) : e \in E^{\times} \right\}$.\\
Given a central extension of a group $G$ by $\Z/2\Z$, say 
\[
0 \longrightarrow \Z/2\Z \longrightarrow G' \longrightarrow G \longrightarrow 1
\]
there is a natural central extension, say $G''$, of $G$ by $\C^{\times}$, given by
\[
G'' := G' \times_{\Z/2\Z} \C^{\times} :=\dfrac{G' \times \C^{\times}}{<(-1,-1)>},
\]
which sits in the following exact sequence
\begin{displaymath}
\xymatrix{ \{1\} \ar[r] & \Z/2\Z \ar[r] \ar@{^{(}->}[d] & G' \ar[r] \ar@{^{(}->}[d]  & G \ar[r] \ar@{=}[d] & \{1\} \\ 
 \{1\} \ar[r] & \C^{\times} \ar[r] & G'' \ar[r] & G \ar[r] & \{1\} }  
\end{displaymath}
This $\C^{\times}$-central extension of $G$ is said to be obtained from the 2-fold cover $G' \rightarrow G$ of $G$. It is well known that $\C^{\times}$-covers tend to be easier to analyse and this is what we shall do in this chapter. \\
Let $F$ be a non-Archimedian local field of characteristic zero. Let $D_{F}$ denote the unique quaternion division algebra \index{quaternion division algebra} with center $F$. Note that $D_{F}^{\times} \hookrightarrow \GL_{2}(E)$ given by fixing an isomorphism $D_{F} \otimes E \cong M_{2}(E)$. By Skolem-Noether theorem, such an embedding is uniquely determined upto conjugation by elements of $\GL_{2}(E)$. The main theorem of this chapter is the following:
\begin{theorem} \label{main theorem: splitting}
Let $E$ be a quadratic extension of a non-Archimedian local field $F$ and $\widetilde{{\rm GL}_{2}(E)}$ the two-fold metaplectic covering of ${\rm GL}_{2}(E)$.  Then:
\begin{enumerate}
\item The two-fold metaplectic covering splits over the subgroup ${\rm GL}_{2}(F)$.
\item The $\C^{\times}$-covering obtained from $\widetilde{{\rm GL}_{2}(E)}$ splits over the subgroup $D_{F}^{\times}$. 
\end{enumerate} 
\end{theorem}
From now onward we abuse the notation and write $\widetilde{\GL_{2}(E)}$ for the $\C^{\times}$-covering obtained from the metaplectic cover $\widetilde{\GL_{2}(E)}$. 
Note that a quadratic extension $L$ of $F$ gives rise to two embeddings of $L$ in $M(2,E)$ as in the diagram below:
\begin{displaymath}
\xymatrix{ & D_{F} \ar@{^{(}->}[rd] & \\
L \ar@{^{(}->}[rd] \ar@{^{(}->}[ru] & & M(2,E). \\
& M(2,F) \ar@{^{(}->}[ru] & }
\end{displaymath}
 By Skolem-Noether theorem, any two embeddings of $L \otimes E$ in $M(2,E)$ and hence of $L$ are conjugate in $M(2,E)$ by ${\rm GL}_{2}(E)$.  \\
Let $\widetilde{\GL_{2}(E)}_{\C^{\times}}$ denote the $\C^{\times}$-covering of $\GL_{2}(E)$ obtained from 2-fold cover $\widetilde{\GL_{2}(E)}$.\\

\begin{Refined Question}
Does there exist a natural identification of the set of splittings of the $\C^{\times}$-cover $\widetilde{\GL_{2}(E)}_{\C^{\times}}$ of ${\rm GL}_{2}(E)$ restricted to ${\rm GL}_{2}(F)$ and set of splittings restricted to $D_{F}^{\times}$ (in either of the two cases the set of splittings is a principal homogeneous space \index{principal homogeneous space} over the character group of $F^{\times}$) such that for any quadratic extension $L$ of $F$, the two embeddings of $L^{\times}$ in $\widetilde{{\rm GL}_{2}(E)}_{\C^{\times}}$ 
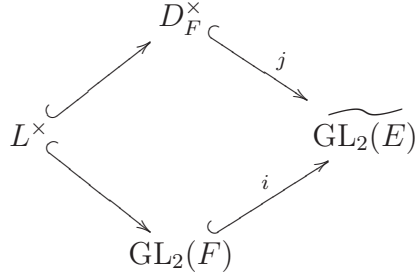
\begin{displaymath}
\xymatrix{ & D_{F}^{\times} \ar@{^{(}->}[rd]^{j} & \\
L^{\times} \ar@{^{(}->}[rd] \ar@{^{(}->}[ru] & & \widetilde{{\rm GL}_{2}(E)} \\
& {\rm GL}_{2}(F) \ar@{^{(}->}[ru]^{i} & }
\end{displaymath}
are conjugate in $\widetilde{{\rm GL}_{2}(E)}_{\C^{\times}}$ ? \\
\end{Refined Question}
We are not able to handle the refined question, and will only content with the proof of the existence of a splitting of the  metaplectic cover of ${\rm GL}_{2}(E)$ restricted to $D_{F}^{\times}$. However the above refined question plays an important role in harmonic analysis relating the pair $(\widetilde{{\rm GL}_{2}(E)}, {\rm GL}_{2}(F))$ with the pair $(\widetilde{{\rm GL}_{2}(E)}, D_{F}^{\times})$. \\ 

We briefly say a few words about the proofs. The proof for ${\rm GL}_{2}(F)$ is straightforward from the explicit knowledge of the cocycle defining the metaplectic cover. For any quadratic extension $L$ of $F$, we know that the embedding $L^{\times} \hookrightarrow D_{F}^{\times}$ is conjugate inside ${\rm GL}_{2}(E)$ to the embedding of $L^{\times}$ inside ${\rm GL}_{2}(E)$ realized as $L^{\times} \hookrightarrow {\rm GL}_{2}(F) \hookrightarrow {\rm GL}_{2}(E)$ (Skolem-Noether theorem). Since the metaplectic cover of ${\rm GL}_{2}(E)$ splits when restricted to ${\rm GL}_{2}(F)$, it is split in particular over $L^{\times}$ for any quadratic extension $L$ of $F$. Thus we know that the restriction of the metaplectic cover of ${\rm GL}_{2}(E)$ to $D_{F}^{\times}$ has the property that it splits over $L^{\times}$ for any quadratic extension $L$ of $F$. This is the key property to be used in the proofs below.

\section{Splitting over ${\rm GL}_{2}(F)$} \index{Splitting over ${\rm GL}_{2}(F)$} \label{section:GL2}
We prove the following proposition:
\begin{proposition} \label{prop:A}
Let $E$ be a quadratic extension of a non-Archimedian local field $F$. Then the metaplectic 2-fold cover $\widetilde{{\rm GL}_{2}(E)}$ of ${\rm GL}_{2}(E)$, as described in the introduction, splits over the subgroup ${\rm GL}_{2}(F)$.
\end{proposition}
\begin{proof}
 To prove that the covering $\widetilde{{\rm GL}_{2}(E)}$ of ${\rm GL}_{2}(E)$ splits over ${\rm GL}_{2}(F)$, it suffices to show that the 2-cocycle $\beta$ which defines the 2-fold metaplectic cover satisfies $\beta(\sigma, \tau) = 1$ for all $\sigma, \tau \in {\rm GL}_{2}(F)$, i.e., the cocycle is identically 1 when restricted to ${\rm GL}_{2}(F)$. One knows that the defining expression of the cocycle $\beta$ involves only quadratic Hilbert symbols of the field $E$. The proposition will follow once we prove that the restriction of the quadratic Hilbert symbol of $E$ to $F$ is identically 1, which is the content of the next lemma.
\end{proof}
\begin{lemma} \label{lemma:B}
If we denote the quadratic Hilbert symbol \index{Hilbert symbol} of the field $E$ by $( \cdot, \cdot)_{E}$, then
\[
 (a,b)_{E} = 1 \mbox{ for all } a,b \in F^{\times}.
\]
\end{lemma}
\begin{proof}
Let $(\cdot, \cdot)_{F}$ denotes the quadratic Hilbert symbol of the field $F$. Then it is well known that for $a \in F^{\times}$ and $b \in E^{\times}$, we have
\[
(a,b)_{E} = (a, Nb)_{F}.
\]
Hence for $a, b \in F^{\times}$ we have
\[
  (a,b)_{E} = (a, Nb)_{F} = (a, b^{2})_{F} = 1. \qedhere
\]
\end{proof}
\section{Splitting over ${\rm SL}_{1}(D_{F})$} \index{Splitting over ${\rm SL}_{1}(D_{F})$} \label{section:SL1D}
Recall that $D_{F}$ denotes the unique quaternion division algebra over the field $F$ and that ${\rm SL}_{1}(D_{F})$ is the subgroup of norm 1 elements in $D_{F}^{\times}$. Fix an embedding $E \hookrightarrow D_{F}$ through which $D_{F}$ can be realized as a two dimensional vector space over $E$ with $E$ acting on $D_{F}$ on the left and $D_{F}$ acting on itself on the right. This gives rise to an embedding $D_{F}^{\times} \hookrightarrow {\rm GL}_{2}(E)$. Since  ${\rm SL}_{1}(D_{F})$ is compact, we can assume that ${\rm SL}_{1}(D_{F}) \subset {\rm GL}_{2}(\mathcal{O}_{E})$. It is well known that if the residue characteristic of $F$ is odd, then the two-fold metaplectic cover $\widetilde{{\rm GL}_2(E)}$ of ${\rm GL}_2(E)$ splits over ${\rm GL}_{2}(\mathcal{O}_{E})$ and hence over ${\rm SL}_{1}(D_{F})$. Such a simple minded proof does not work for $p=2$. However, we prove in this section that the $\C^{\times}$-metaplectic cover of ${\rm SL}_{2}(E)$ does split when restricted to ${\rm SL}_{1}(D_{F})$.

\begin{proposition} \label{prop:C}
The restriction of the non-trivial 2-fold cover of ${\rm SL}_{4}(F)$ to ${\rm SL}_{2}(E)$ remains non-trivial, hence gives the unique non-trivial 2-fold cover of ${\rm SL}_{2}(E)$.
\end{proposition}
\begin{proof}
The proposition amounts to the assertion that there is a commutative diagram involving the unique 2-fold covers of ${\rm SL}_{2}(E)$ and ${\rm SL}_{4}(F)$ as follows:
\begin{displaymath}
\xymatrix{ 0 \ar[r] & \Z/2\Z \ar[r] \ar@{=}[d] & \widetilde{{\rm SL}_{2}(E)} \ar[r] \ar@{^{(}->}[d] & {\rm SL}_{2}(E) \ar[r] \ar@{^{(}->}[d] & 1 \\
                 0 \ar[r] & \Z/2\Z \ar[r]  & \widetilde{{\rm SL}_{4}(F)} \ar[r] & {\rm SL}_{4}(F) \ar[r]  & 1 }
\end{displaymath}
This follows from the generality that the transfer map 
\[
tr : K_{2}(E)/2K_{2}(E) \longrightarrow K_{2}(F)/2K_{2}(F)
\] 
is an isomorphism \cite{Mil71}. 
\end{proof}

\begin{corollary} \label{corollary:D}
 The $\C^{\times}$-cover of  ${\rm SL}_{2}(E)$ obtained from $\widetilde{{\rm SL}_{2}(E)}$ splits over ${\rm SL}_{1}(D_{F})$. \qedhere
\end{corollary}
\begin{proof}
From Proposition \ref{prop:C}, the restriction of the 2-fold cover from ${\rm SL}_{4}(F)$ to ${\rm SL}_{2}(E)$ remains non-trivial. Since we have an inclusion of groups
\begin{displaymath}
\xymatrix{ {\rm SL}_{2}(E) \ar@{^{(}->}[r] & {\rm Sp}_{4}(F) \ar@{^{(}->}[r] & {\rm SL}_{4}(F), }
\end{displaymath}
and all these groups have a unique non-trivial 2-fold cover, we deduce that the unique non-trivial 2-fold cover of ${\rm SL}_{4}(F)$ restricts to give the unique non-trivial 2-fold cover of ${\rm Sp}_{4}(F)$ which in turn restricts to the unique non-trivial 2-fold cover of ${\rm SL}_{2}(E)$. Now we use the inclusion of the groups
\begin{displaymath}
\xymatrix{ {\rm SL}_{1}(D_{F}) \ar@{^{(}->}[r] & {\rm SL}_{2}(E) \ar@{^{(}->}[r] & {\rm Sp}_{4}(F) }
\end{displaymath}
and use a result of Kudla \cite[~Theorem 3.1]{Kud94} according to which the restriction of the $\C^{\times}$-covering of ${\rm Sp}_{4}(F)$ to ${\rm U}(2)$ splits. (The result of Kudla is valid for any unitary group ${\rm U}(n)$ defined by a skew hermitian form in $n$ variables over $E$ and hence comes with a natural embedding in ${\rm Sp}_{2n}(F)$). If we take a non-degenerate hermitian form in  2 variables which is anisotropic, then the corresponding unitary group is ${\rm U}(2)$ and, ${\rm SU}(2) \cong {\rm SL}_{1}(D_{F})$.  As a result, the restriction of the $\C^{\times}$-covering from ${\rm SL}_{2}(E)$ to ${\rm SL}_{1}(D)=SU(2) \subset U(2)$ splits.	\qedhere
\end{proof}

\section{Splitting over $D_{F}^{\times}$} \index{Splitting over $D_{F}^{\times}$} \label{section:D*}
In this section we prove the splitting of the $\C^{\times}$-cover of ${\rm GL}_{2}(E)$ obtained from $\widetilde{{\rm GL}_{2}(E)}$ over $D_{F}^{\times}$. 
\subsection{The case of even residue characteristic} 
Note the following short exact sequence
\[
 1 \longrightarrow {\rm SL}_{1}(D_{F}) \longrightarrow D_{F}^{\times} \longrightarrow F^{\times} \longrightarrow 1. \label{exact:A} \tag{A}
\]
Let $\C^{\times}$ be the trivial $D_{F}^{\times}$-module. Then $H^{2}(D_{F}^{\times}, \C^{\times})$ classifies central extensions of $D_{F}^{\times}$ by the group $\C^{\times}$.  The Hochschild-Serre spectral sequence \index{Hochschild-Serre spectral sequence} arising from (\ref{exact:A}) gives a filtration on $H^{2}(D_{F}^{\times}, \C^{\times}):$
\[
 H^{2}(D_{F}^{\times}, \C^{\times}) = F^{0} \supseteq F^{1} \supseteq F^{2} \supseteq 0
\]
with $F^{0}/F^{1} = E_{\infty}^{0,2}$, $F^{1}/F^{2}= E_{\infty}^{1,1}$ and $F^{2} = E_{\infty}^{2,0}$, where
\begin{center} 
 $E_{2}^{0,2} = H^{0}(F^{\times}, H^{2}({\rm SL}_{1}(D_{F}), \C^{\times})), $ \\
 $E_{2}^{1,1} = H^{1}(F^{\times}, H^{1}({\rm SL}_{1}(D_{F}), \C^{\times})), $ \\
 $E_{2}^{2,0} = H^{2}(F^{\times}, H^{0}({\rm SL}_{1}(D_{F}), \C^{\times})) $.
\end{center}
Consider the embedding $D_{F}^{\times} \hookrightarrow {\rm GL}_{2}(E)$ and denote the restriction of the central extension of ${\rm GL}_{2}(E)$ to $D_{F}^{\times}$ as well as the corresponding element of $H^{2}(D_{F}^{\times}, \C^{\times})$ by $\beta$. In Section \ref{section:SL1D} we proved that $\beta$ restricted to ${\rm SL}_{1}(D_{F})$ is trivial, therefore $\beta \in F^{1}$. In even residue characteristic, since we are dealing with a cohomology class of order 2 (or 1), the following result of C. Riehm \cite{Riehm70} implies that $\beta$ must be trivial in $F^{1}/F^{2}$.
\begin{proposition} \label{prop:Riehm}
Let $G_{0} = {\rm SL}_{1}(D_{F})$ and for $i \geq 1$, let $G_{i}$ denote the $i$-th standard congruence subgroup of $G_{0}$. Then
\[
[G_{0}, G_{0}] = G_{1}.
\]
In particular, the character group of ${\rm SL}_{1}(D_{F})$ is a finite cyclic group of order prime to $p$.
\end{proposition}
Thus, in the case of even residue characteristic, an element of $H^{2}(D_{F}^{\times}, \C^{\times})$ of order 2 which is trivial when restricted to ${\rm SL}_{1}(D_{F})$ arises by inflation from an element of $H^{2}(F^{\times}, \C^{\times})$.
An element of $H^{2}(F^{\times}, \C^{\times})$ is represented by a central extension 
\[
1 \rightarrow \C^{\times} \rightarrow \widetilde{F^{\times}} \rightarrow F^{\times} \rightarrow 1.
\]
The proof of the splitting of the $\C^{\times}$-metaplectic cover of  ${\rm GL}_{2}(E)$ restricted to $D_{F}^{\times}$ will be completed in the case of even residue characteristic once we prove the following lemma.
\begin{lemma} \label{lemma:F}
A $\C^{\times}$-covering of $D_{F}^{\times}$ coming from a $\C^{\times}$-covering of $F^{\times}$ via the norm map, which is trivial on $L^{\times}$ for all quadratic extensions $L$ of $F$, is trivial.
\end{lemma}
\begin{proof}
Suppose there exists a {\it non-trivial} $\C^{\times}$-covering of $D_{F}^{\times}$ coming from a $\C^{\times}$-central extension $\widetilde{F^{\times}}$ of $F^{\times}$ via the norm map, which is trivial on $L^{\times}$ for all quadratic extension $L$ of $F$. If the cover $\widetilde{F^{\times}}$ is non-trivial, then it is non-abelian. Thus there are two elements $e_1, e_2 \in \widetilde{F^{\times}}$ which do not commute. Look at the images, say, $f_1, f_2$ of $e_1, e_2$ in $F^{\times}$. Let $\bar{f_1}$, $\bar{f_2}$ be images of $f_1$, $f_2$ in $F^{\times}/F^{\times 2}$. Since the residue characteristic of $F$ is even, $F^{\times}/F^{\times 2}$ is a vector space over $\Z/2\Z$ of dimension $\geq 3$. Therefore given any two elements $\bar{f_1}, \bar{f_2} \in F^{\times}/F^{\times 2}$, there exist a subgroup $F_{1} \hookrightarrow F^{\times}$ of index 2 containing $f_1, f_2$. By local class field theory, there exists a unique quadratic extension $M$ of $F$ with ${\rm Norm}_{M/F}(M^{\times}) = F_{1}$. Now we use the fact given to us that the central extension of $D_{F}^{\times}$ that we are considering is trivial on $L^{\times}$ for any quadratic extension $L$ of $F$, in particular on $M^{\times}$. Hence the inverse image of $M^{\times}$ in the central extension must be abelian, a contradiction to the construction of $M$.
\end{proof} 
\subsection{The case of odd residue characteristic}
In this subsection we assume that the residue characteristic $p$ of $F$ is odd. We first introduce more notation. Let $\mathcal{O}_{D_{F}}$ be the maximal compact subring of $D_{F}$ and $\mathcal{P}_{D_{F}}$ the maximal ideal of $\mathcal{O}_{D_{F}}$. Let $D_{F}^{\times}(1) := 1 + \mathcal{P}_{D_{F}}$. Note that $D_{F}^{\times}(1)$ is a normal pro-$p$ subgroup in $D_{F}^{\times}$. Since $p$ is odd and $D_{F}^{\times}(1)$ is a normal pro-$p$ subgroup
\[
H^{2}(D_{F}^{\times}, \Z/2\Z) \cong H^{2}(D_{F}^{\times}/D_{F}^{\times}(1), \Z/2\Z).
\]
In other words, every 2-fold central extension of $D_{F}^{\times}$ arises as a pull back of a 2-fold central extension $D_{F}^{\times}/D_{F}^{\times}(1)$. The group $D_{F}^{\times}/D_{F}^{\times}(1)$ is (non-canonically) isomorphic to $\mathbb{F}_{q^2}^{\times} \rtimes \Z$, where $\mathbb{F}_{q^2}$ is the finite field with $q^2$ elements and $\Z$ operates on $\F_{q^2}^{\times}$ by powers of the Frobenius map $x \mapsto x^q$. This group sits in the following short exact sequence 
\[
0 \rightarrow \mathbb{F}_{q^2}^{\times} \rightarrow G':=D_{F}^{\times}/D_{F}^{\times}(1) \rightarrow \Z \rightarrow 0.
\]
Using this description of the group we prove the following proposition.
\begin{proposition} \label{prop:H}
\begin{enumerate}
\item[(A)] We have
\[
H^{2}(D_{F}^{\times}, \Z/2\Z)  \cong \Z/2\Z \oplus \Z/2\Z.
\]
\item[(B)] If we denote the subgroup of 2-torsion elements of $H^{2}(D_{F}^{\times}, \C^{\times})$ by $H^{2}(D_{F}^{\times}, \C^{\times})[2]$ then
\[
H^{2}(D_{F}^{\times}, \C^{\times})[2] = \Z/2\Z.
\]
\end{enumerate}
\end{proposition}
\begin{proof}
\begin{enumerate} 
\item[(A)]  Since $G' = \F_{q^2}^{\times} \rtimes \Z$ and $\Z$ has cohomological dimension 1, the Hochschild-Serre spectral sequence \index{Hochschild-Serre spectral sequence} $E_{2}^{i,j} = H^{i}(\Z, H^{j}(\F_{q^2}^{\times}, \Z/2\Z))$ calculating the cohomology of $G'$  satisfies $E_{2}^{1,1}=E_{\infty}^{1,1}$, $E_{2}^{0,2} = E_{\infty}^{0,2}$ and $E_{2}^{2,0} = E_{\infty}^{2,0} =0$. Therefore
\[
0 \longrightarrow H^{1}(\Z, H^{1}(\F_{q^2}^{\times}, \Z/2\Z)) \longrightarrow H^{2}(G', \Z/2\Z) \longrightarrow H^{2}(\F_{q^2}^{\times}, \Z/2\Z)^{\Z} \longrightarrow 0.
\]
Since $H^{1}(\F_{q^2}^{\times}, \Z/2\Z) \cong \Z/2\Z$ and $H^{2}(\F_{q^2}^{\times}, \Z/2\Z) \cong \Z/2\Z$, and since $\Z$ must act trivially on $\Z/2\Z$, we get
\[
0 \rightarrow H^{1}(\Z, \Z/2\Z) \rightarrow H^{2}(G', \Z/2\Z) \rightarrow \Z/2\Z \rightarrow 0.
\]
which proves part (A) of the proposition.
\item[(B)] This is evident from the next lemma, namely Lemma \ref{lemma:I}.
\end{enumerate}
This proves the proposition.
\end{proof}
By Proposition \ref{prop:H} there are four non-isomorphic two-fold coverings of the group $D_{F}^{\times}$. The lemma below proves that one of the three non-trivial 2-fold covers becomes trivial as a $\C^{\times}$-cover. 
\begin{lemma} \label{lemma:I}
We have a short exact sequence 
\[
0 \longrightarrow \dfrac{H^{1}(D_{F}^{\times}, \C^{\times})}{2H^{1}(D_{F}^{\times}, \C^{\times})} \longrightarrow H^{2}(D_{F}^{\times}, \Z/2\Z) \longrightarrow H^{2}(D_{F}^{\times}, \C^{\times})[2] \longrightarrow 0
\]
with
\[
\dfrac{H^{1}(D_{F}^{\times}, \C^{\times})}{2H^{1}(D_{F}^{\times}, \C^{\times})} \cong \Z/2\Z
\]
where for any abelian group $A$, $A[2]=\{ a \in A : 2a=0 \}$. 
\end{lemma}
\begin{proof}
The short exact sequence can be deduced from the long exact sequence of cohomology groups of $D_{F}^{\times}$ arising from the following short exact sequence 
\begin{displaymath}
\xymatrix{ 0  \ar[r] & \Z/2\Z \ar[r] & \C^{\times} \ar[r]^{2}  & \C^{\times} \ar[r] & 0 }.
\end{displaymath}
Since $[D_{F}^{\times}, D_{F}^{\times}] = \SL_{1}(D_{F})$ and $D_{F}^{\times}/\SL_{1}(D_{F}) \cong F^{\times}$, the second statement follows from the fact that the character group of $D_{F}^{\times}$, i.e. $H^{1}(D_{F}^{\times}, \C^{\times})$, is the same as the character group of $F^{\times}$, and using that $F$ has odd residue characteristic, it is easy to see that 
\[
\dfrac{H^{1}(F^{\times}, \C^{\times})}{2H^{1}(F^{\times}, \C^{\times})} \cong  \Z/2\Z. \qedhere
\] 
\end{proof}
\begin{proposition} \label{prop:J}
Let $M$ be the quadratic unramified extension of $F$ with $M \hookrightarrow D_{F}$. Then a two-fold cover of $D_{F}^{\times}$ which remains non-trivial with  $\C^{\times}$ coefficients does not split over the subgroup $M^{\times} \hookrightarrow D_{F}^{\times}$.
\end{proposition}
\begin{proof}
Let  $M^{\times}(1)=1+\mathcal{P}_{M}$. As $M$ is a quadratic unramified extension of $F$, we have
\[
M^{\times}/M^{\times}(1) \cong \F_{q^2}^{\times} \times \Z.
\]
Since $\Z$ has cohomological dimension 1, by the Kunneth theorem \index{Kunneth theorem}
\[
\begin{array}{lll}
 H^{2}(M^{\times}/M^{\times}(1), \Z/2\Z) &\cong& H^{2}(\F_{q^2}^{\times}, \Z/2\Z) \oplus \left( H^{1}(\F_{q^2}^{\times}, \Z/2\Z) \otimes H^{1}(\Z, \Z/2\Z) \right) \\
 &=& \Z/2\Z \oplus \Z/2\Z.
\end{array}
\]
Since $M^{\times}/M^{\times}(1) \cong \F_{q^2}^{\times} \times \Z$, its character group is isomorphic to $\widehat{\F_{q^2}^{\times}} \times \C^{\times}$. So once again as in lemma \ref{lemma:I}, we get the following short exact sequence:
\[
0 \rightarrow \Z/2\Z= \dfrac{H^{1}(M^{\times}/M^{\times}(1), \C^{\times})}{2 H^{1}(M^{\times}/M^{\times}(1), \C^{\times})} \rightarrow H^{2}(M^{\times}, \Z/2\Z) \rightarrow H^{2}(M^{\times}, \C^{\times}) \rightarrow 0
\]
By considering the embedding $M^{\times}/M^{\times}(1) \hookrightarrow D_{F}^{\times}/D_{F}^{\times}(1)=G'$, we get the following exact sequences with connecting homomorphisms
\begin{displaymath}
 \xymatrix{ 0 \ar[r] & \dfrac{H^{1}(G', \C^{\times})}{2 H^{1}(G', \C^{\times})} \ar[r] \ar[d]^{f} & H^{2}(G', \Z/2\Z) \ar[r] \ar[d]^{g}  & H^{2}(G', \C^{\times})[2] \ar[r] \ar[d]^{h} & 0 \\
                  0 \ar[r] & \dfrac{H^{1} \left( \frac{M^{\times}}{M^{\times}(1)}, \C^{\times} \right) }{2 H^{1} \left( \frac{M^{\times}}{M^{\times}(1)}, \C^{\times}\right)} \ar[r] & H^{2} \left( \frac{M^{\times}}{M^{\times}(1)}, \Z/2\Z \right) \ar[r] & H^{2} \left( \frac{M^{\times}}{M^{\times}(1)}, \C^{\times} \right) [2] \ar[r] &  0. } \label{Diagram:**} \tag{**}
\end{displaymath}
In the next lemma we prove that $h$ is injective. This proves the proposition.
\end{proof}
\begin{lemma} \label{lemma:L}
The right most vertical map $h : H^{2}(G', \C^{\times})[2] \rightarrow H^{2}(M^{\times}, \C^{\times})[2]$ in the above diagram \ref{Diagram:**} is an isomorphism.
\end{lemma}
 \begin{proof}
Consider the short exact sequence which appeared in the proof of proposition \ref{prop:H} with $\Z/2\Z$ replaced by $\C^{\times}$,
 \[
0 \longrightarrow H^{1}(\Z, H^{1}(\F_{q^2}^{\times}, \C^{\times})) \longrightarrow H^{2}(G', \C^{\times}) \longrightarrow H^{2}(\F_{q^2}^{\times}, \C^{\times})^{\Z} \longrightarrow 0.
\]
This combined with the fact that the second cohomology of a cyclic group with coefficients in $\C^{\times}$ is zero, implies that
\[
H^{2}(G', \C^{\times}) = H^{1}(\Z, H^{1} (\F_{q^2}^{\times}, \C^{\times})) = H^{1}(\Z, \widehat{\F_{q^2}^{\times}}).
\]
Similarly
\[
H^{2}(M^{\times}, \C^{\times}) = H^{1}(2\Z, H^{1} (\F_{q^2}^{\times}, \C^{\times})) = H^{1}(2\Z, \widehat{\F_{q^2}^{\times}}).
\] 
We need to prove that the restriction map 
\[
H^{1}(\Z, \widehat{\F_{q^2}^{\times}})[2] \cong \Z/2\Z \longrightarrow H^{1}(2\Z, \widehat{\F_{q^2}^{\times}})[2] \cong \Z/2\Z
\]
is injective. For this,
consider the following short exact sequence
 \begin{displaymath}
\xymatrix{ 0 \ar[r] & \Z \ar[r]^{2}  &\Z \ar[r] & \Z/2\Z \ar[r] & 0 } 
 \end{displaymath}
The above exact sequence gives rise to the following inflation-restriction exact sequence
\[
0 \longrightarrow H^{1}(\Z/2\Z, \widehat{\F_{q^2}^{\times}}) \longrightarrow H^{1}(\Z, \widehat{\F_{q^2}^{\times}}) \longrightarrow H^{1}(2\Z, \widehat{\F_{q^2}^{\times}}).
\] 
By the next lemma there is an isomorphism of $\widehat{\F_{q^2}^{\times}}$ with $\F_{q^2}^{\times}$ preserving the natural $Gal(\F_{q^2}/\F_{q})$ action on these groups.
Hence by Hilbert's Theorem 90 \index{Hilbert's Theorem 90} we get that $H^{1}(\Z/2\Z, \widehat{\F_{q^2}^{\times}} )= 0$. So the map 
\[
H^{1}(\Z, \widehat{\F_{q^2}^{\times}}) \rightarrow H^{1}(2\Z, \widehat{\F_{q^2}^{\times}})
\]
is injective and hence in particular on 2-torsions 
\[
H^{1}(\Z, \widehat{\F_{q^2}^{\times}})[2] \rightarrow H^{1}(2\Z, \widehat{\F_{q^2}^{\times}})[2].
\]
This proves that the map $h$ is non-zero and an isomorphism.
\end{proof}
\begin{lemma} \label{lemma:M}
There is an isomorphism of $\widehat{\F_{q^d}^{\times}}$ with $\F_{q^d}^{\times}$ such that the natural Galois action of $Gal(\F_{q^d}/\F_{q})$ on $\widehat{\F_{q^d}^{\times}}$ becomes the inverse of the natural action of $Gal(\F_{q^d}/\F_{q})$ on $\F_{q^d}^{\times}$ (where by ``inverse" of an action of an abelian group $G$ on a module $M$, we mean $g * m = (g^{-1})m$).
\end{lemma}
\begin{proof}
Since the $Gal(\F_{q^d}/\F_{q})$ operates by $x \mapsto x^{q}$ on $\F_{q^d}^{\times}$, the proof of the lemma is clear.
\end{proof}

\chapter{A theorem of M{\oe}glin-Waldspurger for covering groups} \label{MW-covering}
%\begin{abstract}
%Let $E$ be a non-Archimedian local field of characteristic zero and residue characteristic $p$. Let ${\bf G}$ be a connected reductive group defined over $E$ and $\pi$ an irreducible admissible representation of $G={\bf G}(E)$. A result of C. M$\oe$glin and J.-L. Waldspurger (for $p \neq 2$) and S. Varma (for $p=2$) states that the leading coefficient in the character expansion of $\pi$ at the identity element of ${\bf G}(E)$ gives the dimension of a certain space of degenerate Whittaker forms. In this paper we generalize this result of M$\oe$glin-Waldspurger to the setting of covering groups $\tilde{G}$ of $G$.
%\end{abstract}

\section{Introduction}  
Let $E$ be a non-Archimedian local field of characteristic zero, ${\bf G}$ a connected split reductive group defined over $E$ and $G= {\bf G}(E)$.  Let $\bg = \text{Lie}({\bf G})$ be the Lie algebra of ${\bf G}$ and $\g = \bg (E)$. Let $(\pi, W)$ be an irreducible admissible representation of $G$. A theorem of F. Rodier, in \cite{Rod75}, relates the dimension of the space of non-degenerate Whittaker functionals of $\pi$ with respect to a non-degenerate Whittaker datum and coefficients in the character expansion of $\pi$ around the identity. More precisely, Rodier proves that if the residue characteristic of $E$ is large enough and the group ${\bf G}$ is split then the dimension of the space of non-degenerate Whittaker functionals for $(\pi, W)$ with respect to any Whittaker datum equals the coefficient in the character expansion of $\pi$ at the identity corresponding to an appropriate maximal nilpotent orbit in the Lie algebra $\g$. Rodier proved his theorem assuming that the residue characteristic of $E$ is large enough, in fact, greater than a constant which depends only on the root datum of ${\bf G}$. A theorem of C. M$\oe$glin and J.-L. Waldspurger \cite{MW87} generalizes this theorem of Rodier, in particular proving the theorem of Rodier for the fields $E$ whose residue characteristic is odd. Their version of the theorem does not require ${\bf G}$ to be split. The theorem of M$\oe$glin-Waldspurger is a more precise statement about the coefficients appearing in the character expansion around the identity and certain spaces of `degenerate' Whittaker forms. \index{degenerate Whittaker forms} In a recent work of S. Varma \cite{San14} this theorem has been proved for fields with even residue characteristic. So the theorem of M$\oe$glin-Waldspurger is true for all connected reductive groups without any restriction on the residue characteristic of the field $E$. We now recall the theorem of M$\oe$glin-Waldspurger. To state the theorem we need to introduce some notation. Let $Y$ be a nilpotent element in $\bg$ and suppose $\varphi : \mathbb{G}_{m} \longrightarrow {\bf G}$ is a one parameter subgroup satisfying
\begin{equation} \label{condition:a}
 Ad(\varphi(t))Y=t^{-2}Y.
\end{equation}
Associated to such a pair $(Y,\varphi)$ one can define a certain space $\mathcal{W}_{(Y, \varphi)}$, called the space of degenerate Whittaker forms of $(\pi, W)$ relative to $(Y, \varphi)$ (see Section \ref{degenerate_W_forms}  for the definition). Define $\mathcal{N}_{Wh}(\pi)$ to be the set of nilpotent orbits \index{nilpotent orbits} $\mathcal{O}$ of $\mathfrak{g}$ for which there exists an element $Y \in \mathcal{O}$ and a $\varphi$ satisfying (\ref{condition:a}) such that the space $\mathcal{W}_{(Y, \varphi)}$ of degenerate Whittaker forms  relative to the pair $(Y, \varphi)$ is non-zero. \\ 

Recall that the character expansion \index{character expansion} of $(\pi, W)$ around the identity is a sum $\sum_{\mathcal{O}} c_{\mathcal{O}} \widehat{\mu_{\mathcal{O}}}$, where $\mathcal{O}$ varies over the set of nilpotent orbits of $\mathfrak{g}$, $c_{\mathcal{O}} \in \C$ and $\widehat{\mu_{\mathcal{O}}}$ is the Fourier transform \index{Fourier transform} of a suitably chosen measure $\mu_{\mathcal{O}}$ on $\mathcal{O}$. One defines $\mathcal{N}_{tr}(\pi)$ to be the set of nilpotent orbits $\mathcal{O}$ of $\mathfrak{g}$ such that the corresponding coefficient $c_{\mathcal{O}}$ in the character expansion of $\pi$ around the identity is non zero. \\ 

We have the standard partial order on the set of nilpotent orbits in $\mathfrak{g}$: $\mathcal{O}_{1} \leq \mathcal{O}_{2}$ if $\mathcal{O}_{1} \subset \overline{\mathcal{O}_{2}}$. Let ${\rm Max}(\mathcal{N}_{Wh}(\pi))$ and ${\rm Max}(\mathcal{N}_{tr}(\pi))$  denote the set of maximal element in $\mathcal{N}_{Wh}(\pi)$ and $\mathcal{N}_{tr}(\pi)$ respectively with respect to this partial order. Then the main theorem  of Chapter I of \cite{MW87} is as follows:
\begin{theorem} [M$\oe$glin-Waldspurger] \label{theorem:M-W}
Let ${\bf G}$ be a connected reductive group defined over $E$. Let $\pi$ be an irreducible admissible representation of $G={\bf G}(E)$. Then
\[
 {\rm Max}(\mathcal{N}_{Wh}(\pi)) =  {\rm Max}(\mathcal{N}_{tr}(\pi)).
\]
Moreover, if $\mathcal{O}$ is an element in either of these sets, then for any $(Y, \varphi)$ as above with $Y \in \mathcal{O}$ we have
\[
c_{\mathcal{O}} = \dim \mathcal{W}_{(Y, \varphi)}.
\]
\end{theorem}

If one considers the case of the pair $(Y, \varphi)$ with $Y$ a `regular' nilpotent element then the above theorem of M$\oe$glin-Waldspurger specializes to Rodier's theorem. \\

In this Chapter, we generalize the theorem of M$\oe$glin-Waldspurger to the setting of a covering group $\tilde{G}$ of $G$. Let $\mu_{r}$ be the group of $r$-th roots of unity in $\C^{\times}$. An $r$-fold covering group $\tilde{G}$ of $G$ is a locally compact topological central extension \index{central extension} of $G$ by $\mu_{r}$ giving rise to the following short exact sequence
\begin{equation} \label{def:cover}
 {1} \longrightarrow \mu_r \longrightarrow \tilde{G} \longrightarrow G \longrightarrow {1}.
\end{equation}

The representations of $\tilde{G}$ on which $\mu_{r}$ acts via the natural embedding $\mu_{r} \hookrightarrow \C^{\times}$ are called genuine representations. The definition of the space of degenerate Whittaker forms of a representation of $G$ involves only unipotent groups. Since the covering $\tilde{G} \longrightarrow G$ splits over any unipotent subgroup of $G$ in a unique way, see \cite{MW95}, this makes it possible to define the space of degenerate Whittaker forms for any genuine smooth representation $(\pi, W)$ of $\tilde{G}$. In particular, it makes sense to talk of the set $\mathcal{N}_{Wh}(\pi)$.\\
 
The existence of a character expansion of an admissible genuine representation $(\pi, W)$ of $\tilde{G}$ has been proved by Wen-Wei Li in \cite[~Theorem 4.1.10]{WWLi}. At the identity, the Harish-Chandra-Howe character expansion of an irreducible genuine representation has the same form as the character expansion of an irreducible admissible representation of a linear group and therefore it makes sense to talk of $\mathcal{N}_{\tr}(\pi)$. This makes it possible to have an analogue of Theorem \ref{theorem:M-W} in the setting of covering groups. The main aim of this paper is to prove the following.
\begin{theorem} \label{main:theorem}
Let $\pi$ be an irreducible admissible genuine representation of $\tilde{G}$. Then
\[
 {\rm Max}(\mathcal{N}_{Wh}(\pi)) =  {\rm Max}(\mathcal{N}_{tr}(\pi)).
\]
Moreover, if $\mathcal{O}$ is an element in either of these sets, then for any $(Y, \varphi)$ as above with $Y \in \mathcal{O}$ we have
\[
c_{\mathcal{O}} = \dim \mathcal{W}_{(Y, \varphi)}.
\]
\end{theorem}
 We will use the work of M$\oe$glin-Waldspurger \cite{MW87}, and to accommodate the case of even residue characteristic, we follow Varma \cite{San14}. Let us describe some of the ideas involved in the proof. Let $Y$ be a nilpotent element in $\g$ and $\varphi$ a one parameter subgroup as above. Let $\bg_{i}$ be the eigenspace of weight $i$ under the action of $\mathbb{G}_{m}$ on $\bg$ via $\Ad \circ \varphi$. One can attach a parabolic subgroup ${\bf P}$ with unipotent radical ${\bf N}$ whose Lie algebras are $\mathfrak{p}:= \oplus_{i \geq 0} \bg_{i}$ and $\mathfrak{n}:= \oplus_{i > 0} \bg_{i}$. The one parameter subgroup $\varphi $ also determines a parabolic subgroup ${\bf P}^{-}$ opposite to ${\bf P}$ with Lie algebra $\mathfrak{p}^{-} := \oplus_{i \leq 0} \bg_{i}$. For simplicity, assume $\bg_{1} =0$ for the purpose of the introduction. Then $\mathfrak{n} = \oplus_{i \geq 2} \bg_{i}$, and $\chi : \gamma \mapsto \psi(B(Y, \log \gamma))$ defines a character of $N= {\bf N}(E)$, where $B$ is an $\Ad(G)$-invariant non-degenerate symmetric bilinear form \index{bilinear form} on $\g$ and $\psi$ is an additive character of $E$. In this case (i.e., $\bg_{1}=0$), the space of degenerate Whittaker forms $\mathcal{W}_{(Y,\varphi)}$ is defined to be the twisted Jacquet module of $\pi$ with respect to $(N, \chi)$. In the case where $\bg_{1} \neq 0$, the definition of $\mathcal{W}_{(Y, \varphi)}$ needs to be appropriately modified (see Section \ref{degenerate_W_forms}).\\
 
 On the other hand, to the pair $(Y, \varphi)$ one attaches certain open compact subgroups $G_{n}$ of $G$ for large $n$ and certain characters $\chi_{n}$ of $G_{n}$. One then proves that the covering $\tilde{G} \longrightarrow G$ splits over $G_{n}$ for large $n$, so that $G_{n}$ can be seen as subgroups of $\tilde{G}$ as well. Let $t := \varphi(\varpi)$ and $\tilde{t}$ be any lift of $t$ in $\tilde{G}$. It turns out that ${\tilde{t}}^{-n} G_{n} {\tilde{t}}^{n} \cap N$ becomes an ``arbitrarily large'' subgroup of $N$ and ${\tilde{t}}^{-n}G_{n}{\tilde{t}}^{n} \cap P^{-}$ an ``arbitrarily small'' subgroup of $P^{-}$, as $n$ becomes large. For large $n$, the characters $\chi_{n}$ have been so defined that the character $\chi_{n}' := \chi_{n} \circ \Int({\tilde{t}}^{n})$ restricted to ${\tilde{t}}^{-n} G_{n} {\tilde{t}}^{n} \cap N$ agrees with $\chi$. Using the Harish-Chandra-Howe character expansion one proves that the dimension of $(G_{n}, \chi_{n})$-isotypic component of $W$ is equal to $c_{\mathcal{O}}$ for large enough $n$, where $\mathcal{O}$ is the nilpotent orbit of $Y$ in $\g$. Note that the $(G_{n}, \chi_{n})$-isotypic component of $W$ and the $(\tilde{t}^{-n}G_{n}\tilde{t}^{n}, \chi_{n} \circ \Int(\tilde{t}^{n}))$-isotypic component of $W$ are isomorphic as vector spaces. Finally one proves that there is a natural isomorphism between $(\tilde{t}^{-n}G_{n}\tilde{t}^{n}, \chi_{n} \circ \Int(\tilde{t}^{n}))$-isotypic component of $W$ and $\mathcal{W}_{(Y, \varphi)}$.
 
 \begin{remark}
 The definition of $\mathcal{W}_{(Y, \varphi)}$ (hence that of $\mathcal{N}_{\Wh}(\pi)$) depends on a choice of an additive character $\psi$  of $E$ and a choice of an $\Ad(G)$-invariant non-degenerate symmetric bilinear form $B$ on $\g$. On the other hand, in the character expansion, the $c_{\mathcal{O}}$'s (hence $\mathcal{N}_{\tr}(\pi)$) depend on $\psi$, $B$, a measure on $\tilde{G}$ and a measure on $\g$. However by requiring the measures on $\tilde{G}$ and $\g$ to be compatible via the exponential map $\exp$ one gets rid of the dependency of the $c_{\mathcal{O}}$ on the measures on $\tilde{G}$ and $\g$. Therefore the $c_{\mathcal{O}}$'s depend only on $\psi$ and $B$. For a more detailed discussion about the dependency on $B$ and $\psi$ on the results here, see Remark 4 in \cite{San14}.
 \end{remark} 
\begin{remark} 
One aspect in Varma's proof for $p=2$, which does not obviously generalise from the case when $p \neq 2$ is the prescription of the character $\chi_{n}$ of $G_{n}$ given in \cite{MW87}, which is due to somewhat bad behaviour of the Campbell-Hausdorff formula in the $p=2$ case. Using Kirillov theory for compact $p$-adic groups Varma prescribes a $\chi_{n}$ (although not unique) which will serve our purpose. 
\end{remark}

Although the methods of the proof of Theorem \ref{main:theorem} are not new and heavily depend on the proofs in the case of linear groups \cite{MW87, San14}, the result is useful in the study of the representation theory of covering groups. We will make use of an application (Theorem \ref{Cass-Prasad-GL2}) to this result in the next chapter, when we generalize a result of D. Prasad \cite{Prasad92} in the setting of covering groups, namely, in the harmonic analysis relating the pairs $(\widetilde{{\rm GL}_{2}(E)}, {\rm GL}_{2}(F))$ and $(\widetilde{{\rm GL}_{2}(E)}, D_{F}^{\times})$, where $E/F$ is a quadratic extension of non-Archimedian local fields, $\widetilde{{\rm GL}_{2}(E)}$ is a certain two fold cover of ${\rm GL}_{2}(E)$ defined in Section \ref{definition: 2-fold cover}, and $D_{F}$ is the quaternion division algebra with center $F$ for suitable embeddings $\GL_{2}(F) \hookrightarrow \widetilde{\GL_{2}(E)}$ and $D_{F}^{\times} \hookrightarrow \widetilde{\GL_{2}(E)}$. 

%{\bf Acknowledgements:}  Author would like to express his gratitude to Professor D. Prasad and Professor Sandeep Varma for their numerous help and suggestions at various points. Without their help and continuous encouragement this paper would not have been possible.

\section{Subgroups $G_{n}$ and characters $\chi_{n}$} \label{G_n ans chi_n}
In this section, we recall a certain sequence of subgroups $G_{n}$ of $G$, which form a basis of neighbourhoods at identity and certain characters $\chi_{n} : G_{n} \longrightarrow \C^{\times}$. Although the objects involved in this section were defined for linear groups in \cite{MW87, San14}, we will lift them to our covering groups in a suitable way in Section \ref{covering_groups} and work with these lifts. \\

Let $\mathfrak{O}_{E}$ denote the ring of integers in $E$. We fix an additive character $\psi$ of $E$ with conductor $\mathfrak{O}_{E}$. Fix an $\Ad(G)$-invariant non-degenerate symmetric bilinear form \index{bilinear form} $B : \g \times \g \longrightarrow E$. Let $Y$ be a nilpotent element in $\mathfrak{g}$. Choose a one parameter subgroup $\varphi : \mathbb{G}_{m} \longrightarrow {\bf G}$ satisfying
\begin{equation} \label{property:1}
\Ad(\varphi(s))Y = s^{-2}Y, \forall s \in \mathbb{G}_{m}.
\end{equation}
We note that for a given nilpotent element $Y \in \g$ the existence of $\varphi$ is guaranteed by the theory of $\mathfrak{sl}_{2}$-triplets but there are examples of $\varphi$ which do not come from $\mathfrak{sl}_{2}$-triplets. \index{$\mathfrak{sl}_{2}$-triplets}\\
For $i \in \Z$, define
\[
\bg_{i} = \{ X \in \bg : \Ad(\varphi(s))X = s^{i}X, \forall s \in \mathbb{G}_{m} \}.
\]
Set 
\[
\boldsymbol{\mathfrak{n}}:=\boldsymbol{\mathfrak{n}}^{+}:= \oplus_{i >0} \bg_{i}, \boldsymbol{\mathfrak{n}}^{-}:= \oplus_{i<0} \bg_{i}, \boldsymbol{\mathfrak{p}}^{-}:= \oplus_{i \leq 0} \bg_{i}.
\]
The parabolic subgroup ${\bf P}^{-}$ of ${\bf G}$ normalizing $\boldsymbol{\mathfrak{n}}^{-}$ has $\boldsymbol{\mathfrak{p}}^{-}$ as its Lie algebra. Let ${\bf N}= {\bf N}^{+}$ be the unipotent subgroup of ${\bf G}$ having $\boldsymbol{\mathfrak{n}}$ as the Lie algebra.\\

Let $G(Y)$ be the centralizer of $Y$ in $G$ and $Y^{\#}$  the centralizer of $Y$ in $\g$. The $G$-orbit $\mathcal{O}_{Y}$ of $Y$ can be identified with $G/G(Y)$ and therefore its tangent space at $Y$ can be identified with $\g/Y^{\#}$. Note that 
\[
\begin{array}{lcl}
Y^{\#} &=& \{ X \in \g : [X, Y]=0 \} \\
           &=& \{ X \in \g : B([X,Y], Z) = 0, \forall Z \in \g \} \\
           &=& \{ X \in \g : B(Y, [X,Z])=0, \forall Z \in \g \}.
\end{array}           
\] 
The bilinear form $B$ gives rise to a non-degenerate alternating form $B_{Y} : \g/Y^{\#} \times \g/Y^{\#} \longrightarrow E$ defined by $B_{Y}(X_{1}, X_{2})=B(Y, [X_{1}, X_{2}])$. \\

Let $L \subset \mathfrak{g}$ be a lattice \index{lattice} satisfying the following conditions:
\begin{enumerate}
\item $[L, L] \subset L$,
\item $L = \oplus_{i \in \Z} L_{i}$, where $L_{i}=L \cap \g_{i}$,
\item The lattice $L/L_{Y}$, where $L_{Y} = L \cap Y^{\#}$, is self dual (i.e. $(L/L_{Y})^{\perp} = L/L_{Y}$) with respect to $B_{Y}$. (For any vector space $V$ with a non-degenerate bilinear form $B'$ and a lattice $M$ in $V$, $M^{\perp} := \{ X \in V : B'(X, Y) \in \mathfrak{O}_{E}, \forall \, Y \in V \}$.)
\end{enumerate}
A lattice $L$ satisfying the above properties can be chosen by taking a suitable basis of all $\g_{i}$'s, see \cite{MW87}. Now we summarize a few well known properties of the \index{exponential map} exponential map, and use them to define subgroups $G_{n}$ and their Iwahori decompositions.
\begin{lemma} \label{subgroup G_n}
\begin{enumerate}
\item There exists a positive integer $A$ such that the exponential map $\exp$ is defined and injective on $\varpi^{A}L$, with inverse $\log$. 
\item $\exp|_{\varpi^{n}L}$ is a homeomorphism of $\varpi^{n}L$ onto its image $G_{n}:=\exp(\varpi^{n}L)$, which is an open subgroup of $G$ for all $n \geq A$. 
\item Set $P_{n}^{-} = \exp(\varpi^{n}L \cap \mathfrak{p}^{-})$ and $N_{n} = \exp(\varpi^{n}L \cap \mathfrak{n})$. Then we have an Iwahori factorization \index{Iwahori factorization}
\[
G_{n} = P_{n}^{-}N_{n}.
\]
\end{enumerate}
\end{lemma}
We will be working with a certain character $\chi_{n}$ of $G_{n}$, which we recall in the next lemma. 
\begin{lemma} \label{character_chi_n}
For large $n$ there exists a character $\chi_{n}$ of $G_{n}$, whose restriction to $\exp((Y^{\#} \cap \varpi^{n}L)+\varpi^{n + \val \, 2}L)$ coincides with $\gamma \mapsto \psi(B(\varpi^{-2n}Y, \log \gamma))$. If $P_{n}^{-}$ is as in Lemma \ref{subgroup G_n}, the character $\chi_{n}$ can be chosen so that 
\[
\chi_{n}(p)=1, \forall p \in P_{n}^{-}.
\]
\end{lemma}
For a proof of this lemma and other properties of this character $\chi_{n}$, see \cite[~Lemma 5]{San14}.

\begin{remark} \normalfont
If $p \neq 2$, then the map $\gamma \mapsto \psi(B(\varpi^{-2n}Y, \log \gamma))$ itself defines a character of $G_{n}$ for large $n$ and satisfies the properties stated in Lemma \ref{character_chi_n}.  But when $p=2$, the number of such characters is greater than one, for more details see \cite{San14}. 
\end{remark}

\section{Covering groups} \label{covering_groups}
Let $\mu_{r} := \{ z \in \C^{\times} \mid z^{r} = 1 \}$. Consider an $r$-fold covering $\tilde{G}$ of $G$ which is a locally compact topological central extension by $\mu_{r}$, giving rise to the following short exact sequence
\[
1 \rightarrow \mu_{r} \rightarrow \tilde{G} \xrightarrow{p} G \rightarrow 1.
\]
As $\mu_{r}$ is central in $\tilde{G}$ for any $x \in G$ and $y \in \tilde{G}$ the element $\tilde{x} y \tilde{x}^{-1}$ does not depend on the choice of $\tilde{x}$ for $\tilde{x} \in p^{-1}(\{ x \})$. We may write this element as $xyx^{-1}$. 
\begin{lemma} \label{splitting of covering}
\begin{enumerate}
\item The covering $\tilde{G} \xrightarrow{p} G$ splits over any unipotent subgroup of $G$ in a unique way.
\item For large enough $n$ the covering $\tilde{G} \xrightarrow{p} G$ splits over $G_{n}$. In fact, for large $n$, there is a splitting $s$ of $\tilde{G} \xrightarrow{p} G$ restricted to $\cup_{g \in G} g G_{n} g^{-1}$ such that $s(hth^{-1}) = h s(t) h^{-1}$ for all $h \in G$.
\end{enumerate}
\end{lemma}
\begin{proof}
\begin{enumerate}
\item This is well known, see \cite{MW95}. For a simpler proof, in the case when $E$ has characteristic zero, see \cite[~Section 2.2]{WWLi14}.
\item Recall that the subgroups $G_{n}$ form a basis of neighbourhoods of the identity. It is well known that the covering $\tilde{G} \xrightarrow{p} G$ splits over a neighbourhood of the identity. Therefore for large enough $n$, the covering splits over $G_{n}$. There is more than one possible splitting for the cover $\tilde{G} \xrightarrow{p} G$ over $G_{n}$.  If a splitting is fixed, then any other splitting over $G_{n}$ will differ from the above splitting by a character $G_{n} \longrightarrow \mu_{r}$. \\
Fix some $m$ such that the covering $\tilde{G} \xrightarrow{p} G$ splits over $G_{m} = \exp(\varpi^{m}L)$. As mentioned above, any two splittings over the subgroup $G_{m}$ will differ by a character $G_{m} \rightarrow \mu_{r}$ and any such character is trivial over
\[
G_{m}^{r} := \{ g^{r} : g \in G_{m} \}.
\] 
Hence all the possible splittings over $G_{m}$ agree on $G_{m}^{r}$. The subset $G_{m}^{r}$ is a subgroup of $G_{m}$ as it equals $\exp(r \cdot \varpi^{m} L)$.   Let $g, h \in G$. Then 
\[
(g G_{m} g^{-1} \cap h G_{m}h^{-1}) \supset (g G_{m}^{r} g^{-1} \cap h G_{m}^{r} h^{-1}).
\]
This implies that any two splittings of $\tilde{G} \xrightarrow{p} G$ restricted to $g G_{m}^{r} g^{-1} \cap h G_{m}^{r} h^{-1}$, one of which comes from the restriction of a splitting of $\tilde{G} \xrightarrow{p} G$ over $g G_{m} g^{-1}$ and the other of which comes from the restriction of a splitting over $h G_{m} h^{-1}$, are the same. Now choose $A'$ so large such that $G_{n} \subset G_{m}^{r}$ for $n \geq A'$. We fix the  splitting of $G_{n}$ which comes from that of the restriction to $G_{m}^{r}$. This gives us a splitting over $\cup_{g \in G} g G_{n} g^{-1}$. \qedhere
\end{enumerate}
\end{proof}
Using this splitting we get that an exponential map is defined from a small enough neighbourhood of $\g$ to $\tilde{G}$, namely the usual exponential map composed with this splitting, which one can use to define the character expansion of an irreducible admissible genuine representation $(\pi, W)$ of $\tilde{G}$, which has been done by Wen-Wei Li in \cite{WWLi}.
\begin{remark} \normalfont
If $r$ is co-prime to $p$, then as $G_{n}$ is a pro-$p$ group and $(r,p)=1$, there is no non-trivial character from $G_{n}$ to $\mu_{r}$. In that situation, there is a unique splitting in the above lemma.
\end{remark}
From now onwards, for large enough $n$, we treat $G_{n}$ not only as a subgroup of $G$ but also as one of $\tilde{G}$, with the above specified splitting. In other words, for the covering group $\tilde{G}$ (as in the linear case) we have a sequence of pairs $(G_{n}, \chi_{n})$ using the splitting specified above which satisfies the properties described in Section 2.
\begin{definition} \label{phi tilde}
\normalfont
 Let $H \subset G$ be an open subgroup and $s : H \hookrightarrow \tilde{G}$ be a splitting. Then for any $\phi \in C_{c}^{\infty}(G)$ with ${\rm supp}(\phi) \subset H$, define $\tilde{\phi}_{s} \in C_{c}^{\infty}(\tilde{G})$  as follows:
\[
 \tilde{\phi}_{s}(g) := \left\{ \begin{array}{ll}
                           \phi(g'), & \text{ if } g = s(g') \in s(H) \\
                           0, & \text{ if } g \in \tilde{G} \backslash s(H)
                           \end{array}
                          \right.
\]
\end{definition}
Note that this definition depends upon the choice of splitting. Whenever the splitting is clear in the context or it has been fixed and there is no confusion we write just $\tilde{\phi}$ instead of $\tilde{\phi}_{s}$ and $H$ for $s(H)$. Recall that the convolution \index{convolution} $\phi \ast \phi'$ for $\phi, \phi' \in C_{c}^{\infty}(G)$ is defined by
\[
 \phi \ast \phi' (x) = \int_{G} \phi(xy^{-1}) \phi'(y) \, dy.
\]
Observe that 
\[
\supp(\phi \ast \phi') \subset \supp(\phi) \cdot \supp(\phi'),
\]
which implies the lemma below.
\begin{lemma} \label{convolution}
Let $H$ be an open subgroup of $G$ such that the covering $\tilde{G} \rightarrow G$ has a splitting over $H$, say, $s : H \hookrightarrow \tilde{G}$, satisfying $s(xy)=s(x)s(y)$ whenever $x, y$ are in $H$. If $\phi, \phi' \in C_{c}^{\infty}(G)$ are such that supp($\phi$) and supp($\phi'$) are contained in $H$, then we have 
\[
 \widetilde{\phi \ast \phi'} =\tilde{\phi} \ast \tilde{\phi'}. 
\]
\end{lemma}

\section{Degenerate Whittaker forms} \label{degenerate_W_forms}
In this section we give the definition of degenerate Whittaker forms \index{degenerate Whittaker forms} for a smooth genuine representation $\pi$ of $\tilde{G}$. This is an adaptation of Section I.7 of \cite{MW87} and Section 5 of \cite{San14}. \\

Define $N:=\exp(\mathfrak{n}) = \exp(\oplus_{i \geq 1} \g_{i}), N^{2}:= \exp(\oplus_{i \geq 2} \g_{i})$ and $N' = \exp(\mathfrak{g}_{1} \cap Y^{\#})N^{2}$. It is easy to see that $N^{2}, N'$ are normal subgroups of $N$. Let $H$ be the Heisenberg group \index{Heisenberg group} defined with $\mathfrak{g}_{1}/(\mathfrak{g}_{1} \cap Y^{\#}) \times E$ as underlying set using the symplectic from induced by $B_{Y}$, i.e. for $X, Z \in \g_{1}/(\g_{1} \cap Y^{\#})$ and $a, b \in E$,
\begin{equation} \label{definition:H}
(X, a)(Y, b) = \left( X+Y, a+b + \frac{1}{2}B_{Y}(X,Z) \right).
\end{equation}
Consider the map $N \longrightarrow H$ given by
\[
\exp(X) \mapsto (\bar{X}, B(Y, X)),
\]
where $\bar{X}$ is the image of the $\mathfrak{g}_{1}$ component of $X$ in $\mathfrak{g}_{1}/(\mathfrak{g}_{1} \cap Y^{\#})$. The Campbell-Hausdorff formula \index{Campbell-Hausdorff formula} implies that the above map is a homomorphism with the following kernel 
\[
N'' = \{ n \in N' : B(Y, \log n)=0 \}.
\] 
Let $\chi : N' \longrightarrow \C^{\times}$ be defined by 
\begin{equation} \label{def of chi}
\chi(\gamma) = \psi \circ B(Y, \log \gamma).
\end{equation}
 Note that $\gamma \mapsto B(Y, \log \gamma) \in E \cong \{0\} \times E \subset H$ induces an isomorphism $N'/N'' \cong E$.	\\

We note that the cover $\tilde{G} \xrightarrow{p} G$ splits uniquely over the subgroups $N, N'$ and $N''$. We denote the images of these splittings inside $\tilde{G}$ by the same letters. For a smooth genuine representation $(\pi, W)$ of $\tilde{G}$ we define
\[
N^{2}_{\chi}W= \{ \pi(n)w - \chi(n)w : w \in W, n \in N^{2} \}
\]
and 
\[
N_{\chi}'W= \{ \pi(n)w - \chi(n)w : w \in W, n \in N' \}.
\]
Note that $N$ normalizes $\chi$. Therefore $H=N/N''$ acts on $W/N_{\chi}'W$ in a natural way. This action restricts to $N'/N''$ ( the center of $N/N''$) as multiplication by the character $\chi$. Let $\mathcal{S}$ be the unique irreducible representation of the Heisenberg group $H$ with central character $\chi$. 
\begin{definition} \normalfont
Define the space of degenerate Whittaker forms for $(\pi, W)$ associated to $(Y, \varphi)$ to be 
\[
\mathcal{W} = \mathcal{W}_{(Y, \varphi)} := \Hom_{H}( \mathcal{S}, W/N_{\chi}'W).
\]
\end{definition}
\begin{remark} \normalfont
If $\g_{1} =0$, then $N = N' = N^{2}$. In this case, $\mathcal{W} \cong W/N_{\chi}W$ is the $(N, \chi)$-twisted Jacquet functor.  \index{twisted Jacquet functor}
\end{remark}
\begin{definition} \normalfont
For a smooth representation $(\pi,W)$ of $\tilde{G}$ define $\mathcal{N}_{\Wh}(\pi)$ to be the set of nilpotent orbits $\mathcal{O}$ of $\mathfrak{g}$ such that there exists $Y \in \mathcal{O}$ and $\varphi$ as in Equation \ref{property:1}, such that the space of degenerate Whittaker forms for $\pi$ associated to $(Y, \varphi)$ is non-zero.
\end{definition}
As $\g_{1}/(\g_{1} \cap Y^{\#})$ is a symplectic vector space and $L/L_{Y}$ is self dual, it follows that $L_{H} := (L \cap \g_{1})/(L \cap \g_{1} \cap Y^{\#})$ is a self dual lattice in the symplectic vector space $H/Z(H) \cong \g_{1}/(\g_{1} \cap Y^{\#})$. \\

Recall the definition of the Heisenberg group $H$ (see Equation \ref{definition:H}) and as $\psi$ is trivial on $\mathfrak{O}_{E}$, it follows that one can extend the character $\psi$ of $E \cong Z(H)$ to a character of the inverse image of $2L_{H}$ under $H \longrightarrow \g_{1} / (\g_{1} \cap Y^{\#})$ by defining it to be trivial on $2L_{H} \times \{0\} \subset H$. From Lemma 4 in \cite{San14}, this character can be extended to a character $\tilde{\chi}$ on the inverse image $H_{0}$ of $L_{H}$ under the natural map $H \longrightarrow \g_{1}/(\g_{1} \cap Y^{\#})$.\\

\begin{remark} \normalfont
There are one parameter subgroups $\varphi$ which do not arise from $\mathfrak{sl}_{2}$-triplets. If $\varphi$ arises from  $\mathfrak{sl}_{2}$-triplets, then it is easy to see that $Y^{\#} \subset \oplus_{i \leq 0} \g_{i}$. In particular, we have $\g_{1} \cap Y^{\#} = \{ 0 \}$. As $\g_{1}$ is a symplectic vector space, the Heisenberg group $H \cong \g_{1} \times E$.
\end{remark}

Then, by Chapter 2, Section I.3 of \cite{MVW}, one knows that $\mathcal{S} = {\rm ind}_{H_{0}}^{H} \tilde{\chi}$, $\ind$ denoting the induction with compact support. Since $H_{0}$ is an open subgroup of the locally profinite group $H$, we have the following form of the Frobenius reciprocity law:
\[
\Hom_{H}(\mathcal{S}, \tau) = \Hom_{H}({\rm ind}_{H_{0}}^{H} \tilde{\chi}, \tau) = \Hom_{H_{0}}(\tilde{\chi}, \tau\mid _{H_{0}})
\] 
for any smooth representation $\tau$ of $H$. Thus, in the category of representations of $N$ on which $N'$ acts via the character $\chi$, the functor $\Hom_{H}(\mathcal{S}, -)$ amounts to taking the $\tilde{\chi}\mid _{H_{0}}$-isotypic component. Since $H_{0}$ is compact modulo the center, this functor is exact. Thus we have
\begin{equation} \label{equation for W}
\mathcal{W} = \Hom_{H}(\mathcal{S}, W/N_{\chi}'W) \cong (W/N_{\chi}'W)^{(H_{0}, \tilde{\chi})},
\end{equation}
where $(W/N_{\chi}'W)^{(H_{0}, \tilde{\chi})}$ denotes the $(H_{0}, \tilde{\chi})$-isotypic component \index{isotypic component} of $W/N_{\chi}'W$. \\

Recall that we have defined certain characters $\chi_{n}$ in Section \ref{G_n ans chi_n}, and now we have a character $\tilde{\chi}$. We need to choose them in a compatible way. First we fix a character $\tilde{\chi}$ as above and consider it as a character of $\exp(\g_{1} \cap L)N'$ in the obvious way (as $\exp(\g_{1} \cap L)N'$ is the inverse image of $H_{0}$ under $N \longrightarrow H$). Let $t:=\varphi(\varpi) \in G$. Let $\tilde{t} \in \tilde{G}$ be any lift of $t$ in $\tilde{G}$. Let 
\[
G_{n}' = \Int({\tilde{t}}^{-n})(G_{n}), P_{n}'= \Int({\tilde{t}}^{-n})(P_{n}^{-}) \, {\rm and} \, V_{n}'= \Int({\tilde{t}}^{-n})(N_{n}).
\] 
It can be easily verified that $V_{n}'$ contains $\exp(\g_{1} \cap L)$. We also have $V_{n}' \subset V_{m}'$ for large $m, n$ with $n \leq m$. Moreover
\[
\exp(\g_{1} \cap L)N^{2} = \bigcup_{n \geq 0} V_{n}'.
\] 
It can also be verified easily that $\tilde{\chi} \circ \Int({\tilde{t}}^{-n})$ restricts to a character of $N_{n}$ that extends the character on $N_{n+ \val 2}N_{n}'$ given by $\gamma \mapsto \psi(B(\varpi^{-2n}Y, \log \gamma))$. Now define 
\begin{equation} \label{chi_n is character}
\chi_{n}(pv) = \tilde{\chi}(\tilde{t}^{-n}v \tilde{t}^{n}), \forall p \in P_{n}^{-} \text{ and } \forall v \in V_{n}'.
\end{equation}
\begin{lemma} [Lemma 6 in \cite{San14}]
Let $\chi_{n}$ be as defined in Equation \ref{chi_n is character}. Then $\chi_{n}$ is a character of $G_{n}$ and satisfies the properties stated in Lemma \ref{character_chi_n}. 
\end{lemma}
Define a character $\chi_{n}'$ on $G_{n}'$ as follows:
\[
\chi_{n}' := \chi_{n} \circ \Int({\tilde{t}}^{n}).
\]
\begin{remark} \normalfont
 The characters $\chi_{n}$ have been so defined that for large $n$, $\chi_{n}'$ agrees with $\chi$ on the intersection of their domains, namely, for large $n$ we have,
\[
\chi_{n}'\mid _{V_{n}'} = \tilde{\chi}\mid _{V_{n}'}.
\]
In particular, $\chi_{n}'\mid _{\exp(L \cap \g_{1})} = \tilde{\chi}\mid _{\exp(L \cap \g_{1})}$. One can also see that $\chi_{n}'$ and $\chi_{m}'$ (for large $n,m$) agree on $G_{n}' \cap G_{m}'$, because they agree on $V_{n}' \cap V_{m}'$ and also on $P_{n}' \cap P_{m}'$ (being trivial on it).
\end{remark}
Set
\begin{equation}
W_{n} := \{ w \in W \mid  \pi(\gamma)w = \chi_{n}(\gamma)w, \forall \gamma \in G_{n} \}
\end{equation}
and 
\begin{equation}
W_{n}' := \{ w \in W \mid  \pi(\gamma)w = \chi_{n}'(\gamma)w, \forall \gamma \in G_{n}' \} = \pi({\tilde{t}}^{-n})W_{n}
\end{equation}
For large $m, n$ define the map $I_{n,m}' : W_{n}' \longrightarrow W_{m}'$ by
\begin{equation}
I_{n,m}'(w) = \int_{G_{m}'} \chi_{m}'(\gamma^{-1}) \pi(\gamma) w \, d\gamma.
\end{equation}
Let $m, n$ be large with $m > n$. Since $\chi_{n}'$ is trivial on $P_{n}' \supset P_{m}'$ and since $G_{m}' = P_{m}'V_{m}'$, for a convenient choice of measures we have
\[
\begin{array}{lll}
I_{n,m}'(w) &=& \int_{V_{m}'} \chi_{m}'(x^{-1}) \pi(x) w \, dx \\
				 &=& \int_{\exp(\g_{1} \cap L)} \tilde{\chi}^{-1}(\exp X) \pi(\exp X) \int_{N^{2} \cap G_{m}'} \chi (x^{-1}) \pi(x) w \, dx \, dX.
\end{array}
\]
Now using the fact that $\exp(\g_{1} \cap L)$ lies in $G_{n}'$ for large $n$ and that it normalizes the character $\chi|_{N^{2}}$, we get
\[
\begin{array}{lll}
I_{n,m}'(w) &=& \int_{N^{2} \cap G_{m}'} \chi(x^{-1}) \pi(x) w \, dx \\
				 &=& \int_{N' \cap G_{m}'} \chi(x^{-1}) \pi(x) w \, dx.
\end{array}
\]
From this the following is clear. For large $n, m$ with $m > n$ and suitable choice of measures we have
\begin{equation} \label{composition of I'}
I_{n,m}' = I_{n+1,m}' \circ I_{n,n+1}'.
\end{equation}
For large $n$, the above equation gives that $\ker I_{n,m}' \subset \ker I_{n,l}'$ for $n < m \leq l$. Set $W_{n, \chi}':= \cup_{m>n} \ker I_{n,m}'$. Recall that for any unipotent subgroup $U$, a character $\chi : U \rightarrow \C^{\times}$ and $w \in W$, $\int_{K} \chi(x)^{-1} \pi(x) w \, dx = 0$ for some open compact subgroup $K$ of $U$ if and only if $w \in U_{\chi}W$, where $U_{\chi}W$ is the span of $\{ \pi(u)w - \chi(u)w \mid w \in W, u \in U \}$. Thus we have $W_{n, \chi}' \subset N_{\chi}^{2}W \subset  N_{\chi}'W$, which gives rise to the following natural maps 
\[j_{n} : W_{n}'/W_{n,\chi}' \rightarrow W/N_{\chi}^{2}W \text{ and } j_{n}' : W_{n}'/W_{n,\chi}' \longrightarrow W/N_{\chi}'W
\] 
and these maps give the following diagram:
\begin{equation} \label{commtative diagram}
\xymatrix{ W_{n}'/W_{n, \chi}' \ar[rr]^{j_{n}'} \ar[dr]^{j_{n}} & & W/N_{\chi}'W \\ & W/N_{\chi}^{2}W \ar@{-->}[ru]_{\exists \, {\rm natural}} & }
\end{equation}
By the compatibility between $\chi_{n}'$ and $\tilde{\chi}$, it is easy to see that the image of $j_{n}'$ is contained in $(W/N_{\chi}'W)^{(H_{0}, \tilde{\chi})}$. Let $w \in W$ be such that the image $\bar{w}$ of $w$ in $W/N_{\chi}'W$ belongs to $(W/N_{\chi}'W)^{(H_{0}, \tilde{\chi})}$. For large $n$, $P_{n}'$ acts trivially on $w$, as $(\pi, W)$ is smooth. Since $G_{n}' = P_{n}'V_{n}'= V_{n}'P_{n}'$, the element
\[
\int_{V_{n}'} \chi_{n}'(x^{-1}) \pi(x) w \, dx
\]
belongs to $W_{n}'$. As $\chi_{n}'$ and $\tilde{\chi}$ are compatible, it can be seen that its image in $W/N_{\chi}'W$ is $\bar{w}$. This gives us the following lemma.

\begin{lemma} \label{W_n non zero}
Let $(Y, \varphi)$ be arbitrary. Then any element of $(W/N_{\chi}'W)^{(H_{0}, \chi)}$ belongs to $j_{n}'(W_{n}')$ for all sufficiently large $n$. In particular, if $\mathcal{W} \neq 0$ then, for large $n$, $W_{n}$ and $W_{n}'$ are non-zero.
\end{lemma}

\section{The main theorem} \label{proof of main theorem}
Now recall that, by the work of Wen-Wei Li \cite[~Theorem 4.1.10]{WWLi}, the Harish-Chandra-Howe character expansion of an irreducible admissible genuine representation of $\tilde{G}$ at the identity element has an expression of the same form as that of an irreducible admissible representation of a linear group. The proof of the following lemma for a covering group is verbatim that of  \cite[~Proposition I.11]{MW87}, or equivalently, of \cite[~Proposition 1]{San14}.
\begin{proposition} \label{W is non-zero} 
Let $\mathcal{W}$ be the space of degenerate Whittaker forms for $\pi$ with respect to a given $(Y, \varphi)$. If $\mathcal{W} \neq 0$ then there exists a nilpotent orbit $\mathcal{O}$ in $\mathcal{N}_{\tr}(\pi)$ such that $\mathcal{O}_{Y} \leq \mathcal{O}$ (i.e.,  $Y \in \bar{\mathcal{O}}$).
\end{proposition}
Let the function $\phi_{n} : G \longrightarrow \C$ be defined by
\[
 \phi_{n}(\gamma) = \left\{ \begin{array}{ll}
 											\chi_{n}(\gamma^{-1}), & \text{ if } \gamma \in G_{n} \\
 											0, & \text{ otherwise. }
 											\end{array} \right.
\]
Consider the corresponding function $\tilde{\phi_{n}} : \tilde{G} \longrightarrow \C$ (see Definition \ref{phi tilde}).
Write the character expansion \index{character expansion} of $\pi$ at the identity element as follows:
\[
\Theta_{\pi} \circ \exp = \sum_{\mathcal{O}} c_{\mathcal{O}} \widehat{\mu_{\mathcal{O}}}.
\]
Choose $n$ large enough so that the above expansion is valid over $G_{n}$ and then evaluate $\Theta_{\pi}$ at the function $\tilde{\phi_{n}}$. As $\pi(\tilde{\phi_{n}})$ is a projection from $W$ to $W_{n}$, by definition we get $\Theta_{\pi}(\tilde{\phi_{n}}) = \trace \, \pi(\tilde{\phi_{n}}) = \dim W_{n}$. Now assume that $(Y, \varphi)$ is such that $O_{Y}$ is a maximal element in $\mathcal{N}_{tr}(\pi)$. On the other hand, if we evaluate $\sum_{\mathcal{O}} c_{\mathcal{O}} \widehat{\mu_{\mathcal{O}}}(\tilde{\phi_{n}})$, it turns out that $\widehat{\mu_{\mathcal{O}}}(\tilde{\phi_{n}})$ is zero unless $\mathcal{O} = \mathcal{O}_{Y}$. In addition, if we fix a $G$-invariant measure \index{invariant measure} on $\mathcal{O}_{Y}$ as in I.8 of \cite{MW87} (for more details about this invariant measure see Section 3 of \cite{San14}), we get the following lemma.
\begin{lemma} \label{dimW_n is c_O} {\rm (\cite[~Lemma I.12]{MW87} and \cite[Lemma 7]{San14})} \\
Suppose $(Y, \varphi)$ is such that $\mathcal{O}_{Y}$ is a maximal element of $\mathcal{N}_{\tr}(\pi)$. Then for large $n$, 
\[
\dim W_{n} = c_{\mathcal{O}_{Y}}.
\]
In particular, the dimension of $W_{n}$ is finite and independent of $n$, for large $n$.
\end{lemma}
From Lemma \ref{W_n non zero} we know that every vector in $\mathcal{W}$ is in the image of $j_{n}'$ for large $n$. In particular, if $W_{n}$ is finite dimensional, we get that the map $j_{n}'$ is surjective. Moreover, we have the following lemma whose proof is verbatim that of Corollary I.14 in \cite{MW87} and Lemma 8 in \cite{San14} in the case of a linear group.
\begin{lemma} \label{j_n and j_n'} 
Let $(Y, \varphi)$ be such that $\mathcal{O}_{Y}$ is a maximal element of $\mathcal{N}_{\tr}(\pi)$. Then for large $n$, the maps $j_{n}$ and $j_{n}'$ are injections and the image of $j_{n}'$ is $(W/N_{\chi}'W)^{(H_{0}, \tilde{\chi})}$.
\end{lemma}

Let $\phi_{n}' : G \longrightarrow \C$ be defined by
\[
 \phi_{n}'(\gamma) = \left\{ \begin{array}{ll}
 											\chi_{n}'(\gamma^{-1}), & \text{ if } \gamma \in G_{n}' \\
 											0, & \text{ otherwise. }
 											\end{array} \right.
\]

Consider the corresponding function $\tilde{\phi}_{n}' : \tilde{G} \longrightarrow \C$. Thus, $\tilde{\phi}_{n}' = \tilde{\phi}_{n} \circ \Int(\tilde{t}^{n})$.
\begin{lemma} \label{injectivity}
Consider a pair $(Y, \varphi)$ such that $\mathcal{O} = \mathcal{O}_{Y}$ is a maximal element of $\mathcal{N}_{\tr}(\pi)$. Then for large enough $n$:
\begin{enumerate}
\item Let $\mathcal{Y}_{n} \subset  G_{n+1}' \cap G(Y)$ be a set of representatives for the $G_{n}'$ double cosets in $G_{n}'(G_{n+1}' \cap G(Y))G_{n}'$. Then for large enough $n$, 
\[
\tilde{\phi}_{n}' \ast \tilde{\phi}_{n+1}' \ast \tilde{\phi}_{n}'(g)= \left\{ \begin{array}{ll}
																					\lambda \cdot (\chi_{n}')^{-1}(h_{1}h_{2}), & \text{ if } g=h_{1}yh_{2} \text{ with } y \in \mathcal{Y}_{n},  h_{1}, h_{2} \in G_{n}' \\
																					0, & \text{ if } g \notin G_{n}'\mathcal{Y}_{n}G_{n}',
																					\end{array} \right.
\]
where $\lambda = \meas(G_{n}' \cap G_{n+1}') \meas(G_{n}')$.
\item For large $n$, $I_{n, n+1}'$ is injective.
\end{enumerate}
\end{lemma}
\begin{proof}
From part (a) of Lemma 9 in \cite{San14}, we have
\[
\phi_{n}' \ast \phi_{n+1}' \ast \phi_{n}' (g)= \left\{ \begin{array}{ll}
																					\lambda \cdot (\chi_{n}')^{-1}(h_{1}h_{2}), & \text{ if } g=h_{1}yh_{2} \text{ with } y \in \mathcal{Y}_{n},  h_{1}, h_{2} \in G_{n}' \\
																					0, & \text{ if } g \notin G_{n}'\mathcal{Y}_{n}G_{n}'
																					\end{array} \right.
\]
where $\lambda = \meas(G_{n}' \cap G_{n+1}') \meas(G_{n}')$. Now part 1 follows from Lemma \ref{convolution}, which gives that for large $n$, 
\begin{equation} 
\tilde{\phi}_{n}' \ast \tilde{\phi}_{n+1}' \ast \tilde{\phi}_{n}' = \oversortoftilde{(\phi_{n}' \ast \phi_{n+1}' \ast \phi_{n}')}.
\end{equation}
Now we prove part 2. It is enough to show that $\pi( \tilde{\phi}_{n}' \ast \tilde{\phi}_{n+1}' \ast \tilde{\phi}_{n}')$ acts by a non-zero multiple of identity on $W_{n}'$, since that would implies that $I_{n+1,n}' \circ I_{n,n+1}'$ is a non-zero multiple of the identity on $W_{n}'$. From part 1 we get that $\tilde{\phi}_{n}' \ast \tilde{\phi}_{n+1}' \ast \tilde{\phi_{n}'}$ is a positive linear combination of functions $\tilde{\phi}_{n,y}' : \gamma \mapsto \tilde{\phi}_{n}'(\gamma y^{-1})$, where $y \in G_{n+1}' \cap G(Y)$ is fixed and $G(Y)$ is centralizer of $Y$ in $G$. Then the lemma follows from the fact that $\pi(y)$ acts trivially on $W_{n}'$ for large $n$, so that
\[
\pi(\tilde{\phi}_{n,y}')|_{W_{n}'} = \pi(\tilde{\phi}_{n}') \pi(y)|_{W_{n}'} = \pi(\tilde{\phi}_{n}')|_{W_{n}'}. \qedhere
\]
\end{proof}

\begin{theorem}
Let $(\pi, W)$ be an irreducible admissible genuine representation of $\tilde{G}$. 
\begin{enumerate}
\item  The set of maximal elements in $\mathcal{N}_{\tr}(\pi)$ coincides with the set of maximal elements in $\mathcal{N}_{Wh}(\pi)$.
\item Let $\mathcal{O}$ be a maximal element in $\mathcal{N}_{\tr}(\pi)$. Then the coefficient $c_{\mathcal{O}}$ equals the dimension of the space of degenerate Whittaker forms with respect to any pair $(Y, \varphi)$ such that $Y \in \mathcal{O}$ and $\varphi : \mathbb{G}_{m} \longrightarrow {\bf G}$ satisfies $\Ad(\varphi(s))Y = s^{-2}Y$ for all $s \in E^{\times}$.
\end{enumerate}
\end{theorem}
\begin{proof}
Let $\mathcal{O}$ be a maximal element in $\mathcal{N}_{\tr}(\pi)$. Choose $(Y, \varphi)$ such that $Y \in \mathcal{O}$ and such that $\varphi : \mathbb{G}_{m} \longrightarrow {\bf G}$ satisfies $\Ad(\varphi(s))Y = s^{-2}Y$. Then, from Lemma \ref{dimW_n is c_O}, for large $n$ we have
\[
\dim W_{n} = c_{\mathcal{O}}.
\]
Therefore $W_{n} \neq 0$ (resp $W_{n}' \neq 0$) for large $n$ . By Lemma \ref{j_n and j_n'}, the map $j_{n}'$ is injective and maps surjectively onto $(W/N_{\chi}'W)^{(H_{0}, \tilde{\chi})}$. But by the second part of Lemma \ref{injectivity} and Equation \ref{composition of I'}, $I_{n,m}'$ is injective for large $n$ and $m>n$ which implies that $W_{n, \chi}' = \cup_{m>n} \ker(I_{n,m}') =0$. From Equation \ref{equation for W}, we have $\mathcal{W} \cong (W/N_{\chi}'W)^{(H_{0}, \tilde{\chi}})$. Hence $\dim \mathcal{W} = \dim W_{n}' = \dim W_{n} = c_{\mathcal{O}}$, which proves part 2 of the theorem.  In particular, $\mathcal{W} \neq 0$ and hence $\mathcal{O} \in  \mathcal{N}_{\Wh}(\pi)$. Now we claim that $\mathcal{O}$ is maximal in $\mathcal{N}_{\Wh}(\pi)$. If not, there is a maximal orbit $\mathcal{O}' \in \mathcal{N}_{\Wh}(\pi)$ such that $\mathcal{O} \lneq \mathcal{O}'$. From Proposition \ref{W is non-zero}, there exists $\mathcal{O}'' \in \mathcal{N}_{\tr}(\pi)$ such that $\mathcal{O}' \leq \mathcal{O}''$. Therefore $\mathcal{O} \lneq \mathcal{O}''$ and $\mathcal{O}, \mathcal{O}'' \in \mathcal{N}_{\tr}(\pi)$, a contradiction to the maximality of $\mathcal{O}$ in $\mathcal{N}_{\tr}(\pi)$. \\

Let $\mathcal{O}$ be a maximal element in $\mathcal{N}_{\Wh}(\pi)$. By Proposition \ref{W is non-zero}, there exists an element in $\mathcal{O}' \in \mathcal{N}_{\tr}(\pi)$ such that $\mathcal{O} \leq \mathcal{O}'$. We may assume $\mathcal{O}'$ to be maximal in $\mathcal{N}_{\tr}(\pi)$. Then by the result in the above paragraph, $\mathcal{O}'$ is a maximal element in $\mathcal{N}_{\Wh}(\pi)$. But $\mathcal{O}$ is also maximal in $\mathcal{N}_{\Wh}(\pi)$. Hence $\mathcal{O}=\mathcal{O}'$. This proves that $\mathcal{O}$ is a maximal element in $\mathcal{N}_{\tr}(\pi)$ too. 
\end{proof}

\section{An application: a theorem of Casselman-Prasad} \label{metaplectic Cass-Prasad}
Let ${\bf G}$ be a connected reductive quasi-split group and ${\bf N}$ a maximal unipotent subgroup of ${\bf G}$. Let $\chi$ be a non-degenerate character of $N = {\bf N}(E)$. Then the pair $(N, \chi)$ is called a non-degenerate Whittaker datum. It is well known that there is a bijection between the set of regular nilpotent $\Ad(G)$-orbits in $\mathfrak{g}$ and the set of conjugacy classes of non-degenerate Whittaker data. We state this bijection explicitly in the case where ${\bf G} = {\rm SL}_{2}$ and $\bg = \mathfrak{sl}_{2}$. For any non-zero nilpotent orbit there is a lower triangular nilpotent matrix $Y = Y_{a} = \left( \begin{matrix} 0 & 0 \\ a & 0 \end{matrix} \right)$ in $\mathfrak{sl}_{2}(E)$ such that $Y_{a}$ belongs to the nilpotent orbit. For a given non-zero nilpotent orbit the element $a$ is uniquely determined modulo $E^{\times 2}$. Then the map $\varphi : \mathbb{G}_{m} \rightarrow \SL_{2}$ defined by $\varphi (s) = \left( \begin{matrix} s & 0 \\ 0 & s^{-1} \end{matrix} \right)$ satisfies $\Ad(\varphi(s))Y = s^{-2}Y$. Then $\g_{i} =0$ for $i  \notin \{0, 2, -2 \}$ and $\g_{2} = \left\{ \left( \begin{matrix} 0 & x \\ 0 & 0 \end{matrix} \right) \mid x \in E \right\}$. We have $N = \exp(\g_{2}) = \left\{ \left( \begin{matrix} 1 & x \\ 0 & 1 \end{matrix} \right) \mid x \in E \right\}$. Recall from Equation \ref{def of chi}, that the character $\chi : N \rightarrow \C^{\times}$ is given by $\gamma \mapsto \psi(B(Y, \log \gamma))$, where $B$ is a $\Ad(\SL_{2}(E))$-invariant non-degenerate symmetric bilinear form on $\mathfrak{sl}_{2}(E)$. For the rest of this section, we fix the symmetric non-degenerate bilinear form $B : \mathfrak{sl}_{2}(E) \times \mathfrak{sl}_{2}(E) \rightarrow E$ given by $B(X, Z) = \tr(XZ)$. If $Y= Y_{a}$ then the character $\chi$ of $N$ is given by $\chi \left( \begin{matrix} 1 & x \\ 0 & 1 \end{matrix} \right) = \psi(ax)$. The non-degenerate Whittaker datum determined by the pair $(Y_{a}, \varphi)$ is $(N, \psi_{a})_{(B, \psi, \varphi)}$, where $\psi_{a}$ is the additive character of $E$ given by $x \mapsto \psi(ax)$ and we use the suffix $(B, \psi, \varphi)$ to emphasize the dependence on $B, \psi$ and $\varphi$ of the association $(Y_{a}, \varphi) \rightsquigarrow (N, \psi_{a})$. The set of conjugacy classes of non-degenerate Whittaker data has a set of representatives $\{ (N, \psi_{a}) \mid a \in E^{\times}/E^{\times 2} \}$. If we write $\mathcal{N}_{a}$ for the nilpotent orbit containing $Y = Y_{a}$ then the set of non-zero nilpotent orbits is $\{ \mathcal{N}_{a} \mid a \in E^{\times}/E^{\times 2} \}$. 
%Note that all the non-zero nilpotent orbits are maximal nilpotent orbits. 
Recall that the character expansion $\Theta_{\tau}$ of a representation $\tau$ at the identity also depends on the choices of $B$ and $\psi$. In the character expansion of $\Theta_{\tau}$, we write the coefficients $c_{\mathcal{N}_{a}}$ as $c_{\mathcal{N}_{a}, B, \psi}$ to emphasize its dependence on $B$ and $\psi$. By Theorem \ref{main:theorem}, the bijection between $\{ \mathcal{N}_{a} \mid a \in E^{\times}/E^{\times 2} \}$ and $\{ (N, \psi_{a}) \mid a \in E^{\times}/E^{\times 2} \}$ given by $\mathcal{N}_{a} \xleftrightarrow{(B, \psi, \varphi)}  (N, \psi_{a})$ satisfies the following property:
$c_{\mathcal{N}_{a}, B, \psi} \neq 0$ if and only if the representation $\tau$ of $\widetilde{\SL_{2}(E)}$ admits a non-zero $(N, \psi_{a})_{B, \psi, \varphi}$-Whittaker functional. \\

Let $\tau$ be an irreducible admissible genuine representation of $\widetilde{{\rm SL}_{2}(E)}$. Recall that for an irreducible admissible genuine representation $\tau$ of $\widetilde{{\rm SL}_{2}(E)}$ the character distribution $\Theta_{\tau}$ is a smooth function on the set of regular semisimple elements. The Harish-Chandra-Howe character expansion of $\Theta_{\tau}$ in a neighbourhood of identity is given  as follows:
\[
\Theta_{\tau} \circ \exp = c_{0}(\tau) + \sum_{a \in E^{\times}/E^{\times 2}} c_{a}(\tau) \cdot \widehat{\mu}_{\mathcal{N}_{a}}
\]
where $c_{0}(\tau), c_{a}(\tau)$ are constants and $\widehat{\mu}_{\mathcal{N}_{a}}$ is the Fourier transform of a suitably chosen $\Ad({\rm SL}_{2}(E)$-invariant measure on $\mathcal{N}_{a}$. It follows from Theorem \ref{GHP79:uniquness} and (\ref{restriction:3}) that for any non-trivial additive character $\psi'$ of $N$, the dimension of the space of $(N,\psi')$-Whittaker functionals for $\tau$ is at most one. Therefore, from the theorem of Rodier, as extended in Theorem \ref{main:theorem}, each $c_{a}(\tau)$ is either 1 or 0 depending on whether $\tau$ admits a non-zero Whittaker functional corresponding to the non-degenerate Whittaker datum $(N, \psi_{a})$ or not. \\

\begin{remark} \label{semisimple G tilde}
For $g \in \tilde{G}$, there exists a semisimple element $g_{s} \in \tilde{G}$ such that $g$ belongs to any conjugation invariant neighbourhood of $g_{s}$ in $\tilde{G}$.
\end{remark}
Let $\tau_{1}$ and $\tau_{2}$ be two irreducible admissible genuine representations of $\widetilde{{\rm SL}_{2}(E)}$. As $\widetilde{\{ \pm 1 \}}$ is the center of $\widetilde{\SL_{2}(E)}$ and $\Theta_{\tau_1}$, $\Theta_{\tau_2}$ are given by smooth functions at regular semisimple points, by Remark \ref{semisimple G tilde}, it follows that if $\Theta_{\tau_{1}} - \Theta_{\tau_{2}}$ is a smooth function in a neighbourhood of the identity then it is smooth function on the whole of $\widetilde{{\rm SL}_{2}(E)}$ provided both $\tau_{1}, \tau_{2}$ have the same central characters.\\

For any non-trivial additive character $\psi'$ of $E$, let us assume that $\tau_{1}$ admits a non-zero Whittaker functional for $(N,\psi')$ if and only if $\tau_{2}$ does so too. Under this assumption $c_{a}(\tau_{1}) = c_{a}(\tau_{2})$ for all $a \in E^{\times}/E^{\times 2}$. Then we have the following result.
\begin{theorem} \label{Cass-Prasad-SL2}
Let $\tau_{1}, \tau_{2}$ be two irreducible admissible genuine representations of $\widetilde{{\rm SL}_{2}(E)}$ with the same central characters. For a non-trivial additive character $\psi'$ of $E$ assume that $\tau_{1}$ admits a non-zero Whittaker functional with respect to $(N,\psi')$ if and only if $\tau_{2}$ admits a non-zero Whittaker functional with respect to $(N,\psi')$. Then $\Theta_{\tau_1} - \Theta_{\tau_2}$ is constant in a neighbourhood of identity and hence smooth on $\widetilde{{\rm SL}_{2}(E)}$.
\end{theorem}
Using Theorem \ref{Cass-Prasad-SL2}, we prove an extension of a theorem of Casselman-Prasad \cite[~Theorem 5.2]{Prasad92}. 
\begin{theorem} \label{Cass-Prasad-GL2}
Let $\psi$ be a non-trivial character of $E$. Let $\pi_{1}$ and $\pi_{2}$ be two irreducible admissible genuine representations of $\widetilde{\GL_{2}(E)}$ with the same central characters such that $(\pi_{1})_{N, \psi} \cong (\pi_{2})_{N, \psi}$ as $\tilde{Z}$-modules. Then $\Theta_{\pi_{1}} - \Theta_{\pi_{2}}$ is a smooth function on $\widetilde{\GL_{2}(E)}$.
\end{theorem}
\begin{proof}
We already know that $\Theta_{\pi_{1}}$ and $\Theta_{\pi_{2}}$ are smooth on the set of regular semisimple elements, so is $\Theta_{\pi_{1}} - \Theta_{\pi_{2}}$. To prove the smoothness on whole of $\widetilde{\GL_{2}(E)}$, we need to prove the smoothness at every point in $\tilde{Z}$. As $\tilde{Z}$ is not the center, the smoothness at the identity is not enough to imply the smoothness at every point in $\tilde{Z}$. Note that $\tilde{Z}$ is the center of $\widetilde{\GL_{2}(E)}_{+}$ and  $\widetilde{\GL_{2}(E)}_{+}$ is an open and normal subgroup of $\widetilde{\GL_{2}(E)}$ of index $[E^{\times} : E^{\times 2}]$, cf. Section \ref{Reps of meta GL2}. Let $\tilde{z} \in \tilde{Z}$. The character expansion of $\pi_{1}$ (respectively, of $\pi_{2}$) at $\tilde{z}$ is the same as that of $\pi_{1}|_{\widetilde{\GL_{2}(E)}_{+}}$ (respectively, of $\pi_{2}|_{\widetilde{\GL_{2}(E)}_{+}}$) at $\tilde{z}$. Let $\mu$ be a genuine character in $\Omega(\omega_{\pi_{1}}) = \Omega(\omega_{\pi_{2}})$. Choose irreducible admissible genuine representations $\tau_{1}$ and $\tau_{2}$ of $\widetilde{\SL_{2}(E)}$ which are compatible with $\mu$ (see Section \ref{Reps of meta GL2}), and such that
\begin{equation}
\pi_{1} = \ind_{\widetilde{\GL_{2}(E)}_{+}}^{\widetilde{\GL_{2}(E)}}(\mu \tau_{1}) \text{ and } \pi_{2} = \ind_{\widetilde{\GL_{2}(E)}_{+}}^{\widetilde{\GL_{2}(E)}}(\mu \tau_{2}).
\end{equation}
By Equation \ref{restriction:2},
\begin{equation} \label{noname-1}
\pi_{1}|_{\widetilde{\GL_{2}(E)}_{+}} = \bigoplus_{a \in E^{\times}/E^{\times 2}} (\mu \tau_{1})^{a} \text{ and } \pi_{2}|_{\widetilde{\GL_{2}(E)}_{+}} = \bigoplus_{a \in E^{\times}/E^{\times 2}} (\mu \tau_{2})^{a},
\end{equation}
where we abuse notation to let $a$ denote the matrix $\left( \begin{matrix} a & 0 \\ 0 & 1 \end{matrix} \right)$. Let $\Theta_{\rho, g}$ denote the character expansion of an irreducible admissible representation $\rho$ in a neighbourhood of the point $g$, then
%\[
%\Theta_{\pi_{1}, \tilde{z}} = \sum_{a \in E^{\times}/E^{\times 2}} \Theta_{(\mu \tau_{1})^{a}, \tilde{z}} \text{ and } \Theta_{\pi_{2}, \tilde{z}} = \sum_{a \in E^{\times}/E^{\times 2}} \Theta_{(\mu \tau_{2})^{a}, \tilde{z}},
%\]
\[
\Theta_{\pi_{1}, \tilde{z}} = \sum_{a \in E^{\times}/E^{\times 2}} \Theta_{(\mu \tau_{1})^{a}, \tilde{z}} = \sum_{a \in E^{\times}/E^{\times 2}} \mu^{a}(\tilde{z}) \Theta_{\tau_{1}^{a}, 1}
\]
and 
\[
\Theta_{\pi_{2}, \tilde{z}} = \sum_{a \in E^{\times}/E^{\times 2}} \Theta_{(\mu \tau_{2})^{a}, \tilde{z}} = \sum_{a \in E^{\times}/E^{\times 2}} \mu^{a}(\tilde{z}) \Theta_{\tau_{2}^{a}, 1},
\]
where $\mu^{a}$ is the character of $\tilde{Z}$ given by $\mu^{a}(\tilde{z}) = (a, z) \mu(\tilde{z})$ where $z=p(\tilde{z})$. The smoothness of $\Theta_{\pi_{1}} - \Theta_{\pi_{2}}$ will follow if we prove the smoothness of $\Theta_{\tau_{1}^{a}} - \Theta_{\tau_{2}^{a}}$ for all $a \in E^{\times}/E^{\times 2}$. For any non-trivial character $\psi'$ of $E$, we note that $\tau_{1}$ (resp. $\tau_{2}$) admits a non-zero $\psi'$-Whittaker function if and only if $\tau_{2}^{a}$ (resp. $\tau_{2}^{a}$) admits a non-zero $\psi_{a}'$-Whittaker  functional. By Theorem \ref{Cass-Prasad-SL2}, the smoothness of $\Theta_{\tau_{1}^{a}} - \Theta_{\tau_{2}^{a}}$ is equivalent to the smoothness of $\Theta_{\tau_{1}} - \Theta_{\tau_{2}}$. From the expression \ref{noname-1}, $\mu^{a} \in (\pi_{1})_{N, \psi}$ (respectively, $\mu^{a} \in (\pi_{2})_{N, \psi}$) if and only if $\tau_{1}^{a}$ (respectively, $\tau_{2}^{a}$) admits a $\psi$-Whittaker functional. It can also be seen easily that $\tau_{1}$ (respectively, $\tau_{2}$) admits a $\psi$-Whittaker functional if and only if $\tau_{1}^{a}$ (respectively $\tau_{2}^{a}$) admits $\psi_{a}$-Whittaker functional. It follows that $(\pi_{1})_{N, \psi} = (\pi_{2})_{N, \psi}$ is equivalent to the following: for all non-trivial characters $\psi'$ of $E$, $\tau_{1}$ has a $\psi'$-Whittaker functional if and only if $\tau_{2}$ have a $\psi'$-Whittaker functional. Now from Theorem \ref{Cass-Prasad-SL2}, $\Theta_{\tau_{1}} - \Theta_{\tau_{2}}$ is smooth on $\widetilde{\SL_{2}(E)}$.
\end{proof}
\begin{corollary}
Let $\pi_{1}, \pi_{2}$ be two irreducible admissible genuine representations of $\widetilde{\GL_{2}(E)}$ with the same central character such that $(\pi_{1})_{N, \psi} \cong (\pi_{2})_{N, \psi}$ as $\tilde{Z}$-modules. Let $H$ be a subgroup of $\widetilde{\GL_{2}(E)}$ that is compact modulo center. Then there exist finite dimensional representations $\sigma_{1}, \sigma_{2}$ of $H$ such that 
\[
\pi_{1}|_{H} \oplus \sigma_{1} \cong \pi_{2}|_{H} \oplus \sigma_{2}.
\]
\end{corollary}
In other words, this corollary says that the virtual representation $(\pi_{1} - \pi_{2})|_{H}$ is finite dimensional and hence the multiplicity of an irreducible representation of $H$ in $(\pi_{1} - \pi_{2})|_{H}$ will be finite.

\chapter{Restriction from $\widetilde{{\rm GL}_{2}(E)}$ to ${\rm GL}_{2}(F)$ and $D_{F}^{\times}$}  \label{chapter:restriction}

%\begin{abstract}
% The aim of the present article is to study the multiplicity of an irreducible admissible representation of ${\rm GL}_2(F)$ occuring in an irreducible admissible genuine representation of non-trivial two fold covering $\widetilde{{\rm GL}_2(E)}$ of ${\rm GL}_2(E)$ (metaplectic group), where $F$ is a non-archimedian local field and $E | F$ is a field extension of degree two. 
%\end{abstract}
 
\section{Introduction}  \label{introduction: DP-metaplectic}
Let $F$ be a non-Archimedian local field of characteristic zero and let $E$ be a quadratic extension of $F$. The problem of decomposing a representation of ${\rm {\rm GL}}_2(E)$ restricted to ${\rm {\rm GL}}_2(F)$ was considered and solved by D. Prasad in \cite{Prasad92}, proving a multiplicity one theorem, and giving an explicit classification of representations $\pi_1$ of ${\rm GL}_2(E)$ and $\pi_2$ of ${\rm GL}_2(F)$ such that there exists a non-zero ${\rm GL}_2(F)$ invariant linear form:
\[
 l : \pi_1 \otimes \pi_2 \rightarrow \C.
\]
This problem is closely related to a similar branching law from $\GL_{2}(E)$ to $D_{F}^{\times}$, where $D_{F}$ is the unique quaternion division algebra \index{quaternion division algebra} which is central over $F$, and $D_{F}^{\times} \hookrightarrow \GL_{2}(E)$. We recall that the embedding $D_{F}^{\times} \hookrightarrow \GL_{2}(E)$ is given by fixing an isomorphism $D_{F} \otimes E \cong M_{2}(E)$, and by the Skolem-Noether \index{Skolem-Noether} theorem, such an embedding of $D_{F}^{\times}$ inside $\GL_{2}(E)$ is unique upto conjugation by elements of $\GL_{2}(E)$. Henceforth, we fix one such embedding of $D_{F}^{\times}$ inside $\GL_{2}(E)$. The restriction problems for the pair $(\GL_{2}(E), \GL_{2}(F))$ and $(\GL_{2}(E), D_{F}^{\times})$ are related by a certain \index{dichotomy} dichotomy. More precisely, the following result was proved in \cite{Prasad92}:
\begin{theorem}[D. Prasad]
Let $\pi_{1}$ and $\pi_2$ be irreducible admissible infinite dimensional representations of ${\rm GL}_2(E)$ and ${\rm GL}(2,F)$ respectively such that the central character of $\pi_1$ restricted to center of ${\rm GL}_2(F)$ is the same as the central character of $\pi_2$.  Then
\begin{enumerate}
 \item For a principal series representation $\pi_2$ of ${\rm GL}_2(F)$, we have
 \[
\dim \Hom_{{\rm GL}_2(F)} \left( \pi_{1},\pi_{2} \right) = 1.
\]
 \item For a discrete series representation $\pi_2$ of ${\rm GL}_2(F)$, let $\pi_2'$ be the finite dimensional representation of $D_{F}^{\times}$ associated to $\pi_2$ by the \index{Jacquet-Langlands correspondence} Jacquet-Langlands correspondence, then
 \[
  \dim \Hom_{{\rm GL}_2(F)} \left( \pi_{1},\pi_{2} \right) + \dim \Hom_{D_{F}^{\times}} \left( \pi_{1}, \pi_{2}' \right) = 1.
 \]
\end{enumerate}
\end{theorem}
In this chapter, we will study the analogous problem in the metaplectic setting. More precisely, instead of considering ${\rm GL}_2(E)$ we will consider the group $\widetilde{{\rm GL}_2(E)}_{\C^{\times}}$ which is a topological central extension of $\GL_{2}(E)$ by $\C^{\times}$, which is obtained from the two fold topological central extension $\widetilde{\GL_{2}(E)}$ that has been defined in Section \ref{definition: 2-fold cover} explicitly. In Chapter 3, we have proved that the covering $\widetilde{{\rm GL}_2(E)}_{\C^{\times}}$ of ${\rm GL}_2(E)$ splits when restricted to ${\rm GL}_2(F)$ or $D_{F}^{\times}$. Recall that the splittings over $\GL_{2}(F)$ and $D_{F}^{\times}$ are not unique. As there is more than one splitting in each case, to study the problem of decomposing a representation of $\widetilde{\GL_{2}(E)}_{\C^{\times}}$ restricted to $\GL_{2}(F)$ and $D_{F}^{\times}$, we will fix one splitting of each of the subgroups ${\rm GL}_2(F)$ and $D_{F}^{\times}$, related to each other as in the following working hypothesis formulated by D. Prasad. \\

\begin{workinghypothesis} \label{working hypothesis}
Let $L$ be a quadratic extension of $F$. The sets of splittings
\begin{displaymath}
\xymatrix{ 
     & \widetilde{\GL_{2}(E)}_{\C^{\times}}  \ar[dd]   &  &   & \widetilde{\GL_{2}(E)}_{\C^{\times}} \ar[dd] \\ 
     &    & {\rm and}  & & \\
     \GL_{2}(F) \ar@{^{(}-->}[ruu]^{i} \ar[r] & \GL_{2}(E)   &  & D_{F}^{\times} \ar@{^{(}-->}[ruu]^{j} \ar[r] & \GL_{2}(E) }
\end{displaymath} 
are principal homogeneous spaces over the Pontrjagin dual of $F^{\times}$. We work under the hypothesis that there is a natural identification between these two sets of splittings in such a way that for any quadratic extension $L$ of $F$, any two embeddings of $L^{\times}$ in $\widetilde{{\rm GL}_{2}(E)}_{\C^{\times}}$ as in the following diagrams are conjugate in $\widetilde{{\rm GL}_{2}(E)}_{\C^{\times}}$. 
\begin{displaymath}
\xymatrix{ 
 &     & \widetilde{\GL_{2}(E)}_{\C^{\times}}  \ar[dd]         &  &   &     & \widetilde{\GL_{2}(E)}_{\C^{\times}} \ar[dd]   \\
 &  &  & {\rm and} & & & \\
L^{\times} \ar@{^{(}-->}[rruu] \ar[r]  &  \GL_{2}(F) \ar@{^{(}-->}[ruu] \ar[r] & \GL_{2}(E)  &  &  L^{\times} \ar@{^{(}-->}[rruu] \ar[r] & D_{F}^{\times} \ar@{^{(}-->}[ruu] \ar[r] & \GL_{2}(E) }
\end{displaymath} 
Here $L^{\times} \hookrightarrow \GL_{2}(F)$ (respectively, $L^{\times} \hookrightarrow D_{F}^{\times}$) are obtained by identifying a suitable maximal torus of $\GL_{2}(F)$ (respectively, $D_{F}^{\times}$ viewed as an algebraic group) with ${\rm Res}_{L/F} \mathbb{G}_{m}$. 
\end{workinghypothesis}
We fix the embedding ${\rm GL}_{2}(F) \hookrightarrow \widetilde{{\rm GL}_{2}(E)}_{\C^{\times}}$ given by $g \mapsto (g,1)$ (see Proposition \ref{section:GL2}) and fix the corresponding splitting $D_{F}^{\times} \hookrightarrow \widetilde{{\rm GL}_{2}(E)}_{\C^{\times}}$ of $D_{F}^{\times}$ so that the above `working hypothesis' is satisfied, event though we have not proved in this thesis that there is such a splitting $D_{F}^{\times} \hookrightarrow \widetilde{\GL_{2}(E)}$. For any subset $X$ of $\GL_{2}(F)$, we regard $X$ also as a subset of $\widetilde{\GL_{2}(E)}_{\C^{\times}}$ using the above specified embedding of $\GL_{2}(F)$ inside $\widetilde{\GL_{2}(E)}_{\C^{\times}}$. We abuse notation and write $\widetilde{\GL_{2}(E)}$ for $\widetilde{\GL_{2}(E)}_{\C^{\times}}$. With the above fixed embeddings of $\GL_{2}(F)$ and $D_{F}^{\times}$ inside $\widetilde{\GL_{2}(E)}_{\C^{\times}}$, we are in a situation analogous to that of \cite{Prasad92} and we study the decomposition of a representation of $\widetilde{{\rm GL}_2(E)}$ restricted to ${\rm GL}_2(F)$. As in the case of ${\rm GL}_2(E)$, this problem is also closely related to a similar branching law from $\widetilde{{\rm GL}_2(E)}$ to $D_{F}^{\times}$. In what follows, we always consider genuine representations of the metaplectic group $\widetilde{{\rm GL}_2(E)}$.\\
% {\color{red} since otherwise the representation of $\widetilde{{\rm GL}_2(E)}$ descends to a representation of ${\rm GL}_2(E)$ and hence one will be in the situation considered by Prasad in \cite{Prasad92}.}\\ 
Let $B(E), A(E)$ and $N(E)$ be the Borel subgroup, maximal torus and maximal unipotent subgroup of ${\rm GL}_2(E)$ consisting of all upper triangular matrices, diagonal matrices and upper triangular unipotent matrices respectively. Let $B(F), A(F)$ and $N(F)$ denote the corresponding subgroups of ${\rm GL}_2(F)$. Let $Z$ be the center of ${\rm GL}_2(E)$ and $\tilde{Z}$ the inverse image of $Z$ in $\widetilde{{\rm GL}_2(E)}$. Note that $\tilde{Z}$ is an abelian subgroup of $\widetilde{\GL_{2}(E)}$ but is not the center of $\widetilde{{\rm GL}_2(E)}$; the center of $\widetilde{{\rm GL}_2(E)}$ is $\tilde{Z^2}$, the inverse image of $Z^2 := \{ z^{2} \mid z \in Z \}$. \\

Let $\psi$ be a non-trivial additive character of $E$. Let $\pi$ be an irreducible admissible genuine representation of $\widetilde{\GL_{2}(E)}$ and recall that the $\psi$-twisted Jacquet module $\pi_{N, \psi}$ is a $\tilde{Z}$-module. Moreover $\pi_{N, \psi}$ is a $\tilde{Z}$-submodule of $\Omega(\omega_{\pi})$. \\
The main theorem proved in this chapter is the following: 
\begin{theorem} \label{DP-metaplctic}
Let $\pi_{1}$ be an irreducible admissible genuine representation of $\widetilde{{\rm GL}_2(E)}$ and let $\pi_2$ be an infinite dimensional irreducible admissible representation of ${\rm GL}_2(F)$. Assume that the central characters $\omega_{\pi_{1}}$ of $\pi_{1}$ and $\omega_{\pi_{2}}$ of $\pi_{2}$ agree on $E^{\times 2} \cap F^{\times}$. Fix a non-trivial additive character $\psi$ of $E$ such that $\psi|_{F} = 1$. Let $Q = (\pi_{1})_{N(E)}$ be the Jacquet module of $\pi_{1}$. Then 
\begin{enumerate}
 \item Let $\pi_{2} = \Ind_{B(F)}^{\GL_{2}(F)}(\chi)$ be a principal series representation of $\GL_{2}(F)$. Assume that $\Hom_{A(F)} \left( Q, \chi \cdot \delta^{1/2} \right) = 0$. Then
 \[
  \dim \Hom_{{\rm GL}_2(F)} \left( \pi_{1}, \pi_{2} \right) = \dim \Hom_{Z(F)}( (\pi_{1})_{N, \psi}, \omega_{\pi_{2}}).
 \]
 \item Let $\pi_{1}= \Ind_{\widetilde{B(E)}}^{\widetilde{\GL_{2}(E)}} (\tilde{\tau})$ be a principal series representation of $\widetilde{{\rm GL}_2(E)}$ and $\pi_2$ a discrete series representation of ${\rm GL}_2(F)$. Let $\pi_2'$ be the finite dimensional representation of $D_{F}^{\times}$ associated to $\pi_{2}$ by the Jacquet-Langlands correspondence. Assume that $\Hom_{\GL_{2}(F)} \left( \Ind_{B(F)}^{\GL_{2}(F)}(\tilde{\tau}), \pi_{2} \right) = 0$. Then
 \[
  \dim \Hom_{{\rm GL}_2(F)} \left( \pi_{1}, \pi_{2} \right) + \dim \Hom_{D_{F}^{\times}} \left( \pi_{1}, \pi_{2}' \right) =  [E^{\times} : F^{\times}E^{\times 2}].
 \]
 \item Let $\pi_1$ an irreducible admissible genuine representation of $\widetilde{{\rm GL}_2(E)}$ and $\pi_2$ a supercuspidal representation of ${\rm GL}_2(F)$. Let $\pi_{1}'$ be a genuine representation of $\widetilde{{\rm GL}_{2}(E)}$ which has the same central character as that of $\pi_{1}$ and as a $\tilde{Z}$-module $(\pi_{1})_{N, \psi} \oplus (\pi_{1}')|_{N, \psi} = \Omega(\omega_{\pi_{1}})$. Let $\pi_2'$ be the finite dimensional representation of $D_{F}^{\times}$ associated to $\pi_2$ by the Jacquet-Langlands correspondence. Then
\[
 \dim \Hom_{{\rm GL}_2(F)} \left( \pi_{1} \oplus \pi_{1}', \pi_{2} \right) + \dim \Hom_{D_{F}^{\times}} \left( \pi_{1} \oplus \pi_{1}', \pi_{2}' \right) =  [E^{\times} : F^{\times}E^{\times 2}].
\] 
\end{enumerate}
\end{theorem}

\section{Part 1 of Theorem \ref{DP-metaplctic}}

Let $\pi_{2} = \Ind_{B(F)}^{{\rm GL}_2(F)}(\chi)$ be a principal series representation of ${\rm GL}_2(F)$ where $\chi$ is a character of $A(F)$. By Frobenius reciprocity \cite[~Theorem 2.28]{BerZel76}, we get 
\[
\begin{array}{lcl}
 \Hom_{{\rm GL}_2(F)}(\pi_{1}, \pi_{2}) & = & \Hom_{{\rm GL}_2(F)}(\pi_{1}, \Ind_{B(F)}^{{\rm GL}_2(F)}(\chi)) \\
                                                  & = & \Hom_{A(F)}((\pi_{1})_{N(F)}, \chi . \delta^{1/2}) 
\end{array}
\]
where $(\pi_{1})_{N(F)}$ is the Jacquet module of $\pi_{1}$ with respect to $N(F)$. Now depending on whether $\pi_1$ is a supercuspidal representation or not, we consider them separately. \\
First we consider the case when $\pi_{1}$ is a supercuspidal representation of $\widetilde{{\rm GL}_2(E)}$. Then one knows that the functions in the Kirillov model \index{Kirillov model} have compact support in $E^{\times}$ and one has $\pi_{1} \cong  \mathcal{S}(E^{\times}, (\pi_{1})_{N, \psi})$, see Theorem \ref{prop:Kirillov model}. Now using Proposition \ref{restriction jacquet} we get the following:
\[
\begin{array}{lcl}
\Hom_{{\rm GL}_2(F)}(\pi_{1}, \pi_{2}) & = & \Hom_{A(F)} \left( (\pi_{1})_{N(F)}, \chi \cdot \delta^{1/2} \right) \\
                        & = & \Hom_{A(F)} \left( \mathcal{S}(E^{\times}, (\pi_{1})_{N, \psi})_{N(F)}, \chi \cdot \delta^{1/2} \right) \\ 
                        & = & \Hom_{A(F)} \left( \mathcal{S}(F^{\times}, (\pi_{1})_{N, \psi}), \chi \cdot \delta^{1/2} \right) \\ 
\end{array}
\]
From the Kirillov model of $\pi_{1}$, it follows that the action of $A(F)$ on $\mathcal{S}(F^{\times}, (\pi_{1})_{N, \psi})$ is given by
\[
\left( \left( \begin{matrix} a & 0 \\0 & d \end{matrix} \right) \cdot \xi \right) (x) = \left( \begin{matrix} d & 0 \\ 0 & d \end{matrix} \right) \cdot (\xi(ad^{-1}x)),
\]
for all $a, d, x \in F^{\times}$ and $\xi \in \mathcal{S}(F^{\times}, (\pi_{1})_{N, \psi})$. From this explicit action of $A(F)$ on $\mathcal{S}(F^{\times}, (\pi_{1})_{N, \psi})$ it can be checked that as an $A(F)$-module $\mathcal{S}(F^{\times}, (\pi_{1})_{N, \psi}) \cong \ind_{Z(F)}^{A(F)}(\pi_{1})_{N, \psi}$. Using Frobenius reciprocity \cite[~Proposition 2.29]{BerZel76}, we get the following:
\[
\begin{array}{lcl}
\Hom_{{\rm GL}_2(F)}(\pi_{1}, \pi_{2}) & = & \Hom_{A(F)} \left( \ind_{Z(F)}^{A(F)}(\pi_{1})_{N, \psi}, \chi \cdot \delta^{1/2} \right) \\
                        & = & \Hom_{Z(F)} \left( (\pi_{1})_{N, \psi}, (\chi \cdot \delta^{1/2})|_{Z(F)} \right) \\
                        & = & \Hom_{Z(F)} \left( (\pi_{1})_{N, \psi}, \omega_{\pi_{2}} \right)
\end{array}
\]
Now we consider the case when $\pi_1$ is not a supercuspidal representation of $\widetilde{{\rm GL}_2(E)}$. Then from Equation \ref{Kirillov ses} we get the following short exact sequence of $A(F)$-modules
\[
 0 \rightarrow \mathcal{S}(F^{\times}, (\pi_{1})_{N, \psi}) \rightarrow (\pi_{1})_{N(F)} \rightarrow Q \longrightarrow 0.
\]
Now applying the functor $\Hom_{A(F)}(-, \chi . \delta^{1/2})$, we get the following long exact sequence 
\[
\begin{array}{llll}
0 & \rightarrow \Hom_{A(F)} \left( Q, \chi . \delta^{1/2} \right) & \rightarrow \Hom_{A(F)} \left( (\pi_{1})_{N(F)}, \chi . \delta^{1/2} \right) \\
  & \rightarrow \Hom_{A(F)} \left( \mathcal{S}(F^{\times}, (\pi_{1})_{N, \psi}), \chi . \delta^{1/2} \right) & \rightarrow \Ext^{1}_{A(F)} \left( Q, \chi . \delta^{1/2} \right) \\
  & \rightarrow \cdots
\end{array}
\]
\begin{lemma}
  $\Hom_{A(F)} \left( Q, \chi . \delta^{1/2} \right) = 0$ if and only if $\Ext^{1}_{A(F)} \left( Q, \chi . \delta^{1/2} \right) = 0$.
\end{lemma}
\begin{proof}
The space $Q$ is finite dimensional and completely reducible. So it is enough to prove the lemma for one dimensional representation, i.e., for characters of $A(F)$. Moreover one can regard these representations as representation of $F^{\times}$ (after tensoring by a suitable character of $A(F)$ so that it descends to a representation of $A(F)/Z(F) \cong F^{\times}$). Then our lemma follows from the following lemma due to D. Prasad:
\end{proof}
\begin{lemma}
If $\chi_1$ and $\chi_2$ are two characters of $F^{\times}$, then
\[
\dim \Hom_{F^{\times}}(\chi_1, \chi_2) = \dim \Ext_{F^{\times}}^{1}(\chi_1, \chi_2). 
\] 
\end{lemma}
\begin{proof}
Since $F^{\times} \cong \mathcal{O}^{\times} \times \varpi^{\Z}$ where $\mathcal{O}$ is the ring of integers in $F$ and $\mathcal{O}^{\times}$ is compact, $\Ext_{F^{\times}}^{i}( \chi_{1}, \chi_{2}) = H^{i} \left( \Z, \Hom_{\mathcal{O}^{\times}}( \chi_{1}, \chi_{2}) \right)$. If $\Hom_{\mathcal{O}^{\times}} (\chi_{1}, \chi_{2}) = 0$, then the lemma is obvious. Hence suppose that $\Hom_{\mathcal{O}^{\times}}( \chi_{1}, \chi_{2}) \neq 0$. Then  $\Hom_{\mathcal{O}^{\times}}( \chi_{1}, \chi_{2})$ is certain one dimensional vector space with an action of $\varpi^{\Z}$. If the action of $\varpi^{\Z}$ on $\Hom_{\mathcal{O}^{\times}} (\chi_{1}, \chi_{2})$ is non-trivial then $H^{i} (\Z, \Hom_{\mathcal{O}^{\times}} (\chi_{1}, \chi_{2})) = 0$ for all $i \geq 0$. Whereas if the action of $\varpi^{\Z}$ on $\Hom_{\mathcal{O}^{\times}} (\chi_{1}, \chi_{2})$ is trivial, then $H^{0}(\Z, \Z) \cong H^{1}(\Z, \Z) \cong \Z$. 
\end{proof}
We have made an assumption that $\Hom_{A(F)}(Q, \chi . \delta^{1/2}) =0$ and hence by the lemma above $\Ext^{1}_{A(F)}(Q, \chi . \delta^{1/2}) = 0$. So in this case 
\[
\begin{array}{lcl}
 \Hom_{A(F)}((\pi_{1})_{N}, \chi . \delta^{1/2}) & \cong & \Hom_{A(F)}(\mathcal{S}(F^{\times}, (\pi_{1})_{N, \psi}), \chi . \delta^{1/2}) \\
    & = & \Hom_{Z(F)} \left( (\pi_{1})_{N, \psi}, \omega_{\pi_{2}} \right).
\end{array}
\]
Hence
\[
\dim \Hom_{\GL_{2}(F)} \left( \pi_{1}, \pi_{2} \right) = \dim \Hom_{Z(F)} \left( (\pi_{1})_{N, \psi}, \omega_{\pi_{2}} \right).
 \]
\begin{remark}
As $Q$ is a finite dimensional representation of $\widetilde{A(E)}$, only finitely many characters of $A(F)$ appear in $Q$. For a given $\pi_{1}$ there are only finitely many characters $\chi$ such that $\Hom_{A(F)}(Q, \chi.\delta^{1/2}) \neq 0$. We are leaving out at most $2 [E^{\times} : E^{\times 2}]$ many principal series representations $\pi_{2}$ for a given $\pi_{1}$. Note that $2 [E^{\times} : E^{\times 2}]$ is the maximum possible dimension of $Q$, i.e. the case of a principal series representation $\pi_{1}$.
\end{remark}

\section{Part 2 of Theorem \ref{DP-metaplctic}} \label{DP-metaplectic-Part2}
In this section, we consider the case when $\pi_1$ is a principal series representation of $\widetilde{{\rm GL}_{2}(E)}$ and $\pi_2$ a discrete series representation of ${\rm GL}_{2}(F)$.

Let $\pi_{1} = \Ind_{\widetilde{B(E)}}^{\widetilde{{\rm GL}_2(E)}}(\tilde{\tau})$, where $(\tilde{\tau}, V)$ is a genuine irreducible representation of $\tilde{A}=\widetilde{A(E)}$. Now as in \cite{Prasad92}, we use Mackey theory to understand its restriction to ${\rm GL}_2(F)$. We have $\widetilde{{\rm GL}_2(E)}/\widetilde{B(E)} \cong \mathbb{P}_{E}^{1}$ and this has two orbits under the left action of ${\rm GL}_2(F)$. One of the orbits is closed, and naturally identified with $\mathbb{P}_{F}^{1} \cong {\rm GL}_2(F)/B(F)$. The other orbit is open, and can be identified with $\mathbb{P}_{E}^{1} - \mathbb{P}_{F}^{1} \cong {\rm GL}_2(F)/E^{\times}$. By \index{Mackey theory} Mackey theory, we get the following exact sequence of ${\rm GL}_2(F)$-modules:
\begin{equation} \label{p.s. ses}
0 \rightarrow \ind_{E^{\times}}^{{\rm GL}_2(F)} (\tilde{\tau}'|_{E^{\times}}) \rightarrow \pi_{1} \rightarrow \Ind_{B(F)}^{{\rm GL}_2(F)}(\tilde{\tau}|_{B(F)}\delta^{1/2}) \rightarrow 0,
\end{equation}
where $\tilde{\tau}'|_{E^{\times}}$ is the representation of $E^{\times}$ obtained from the embedding $E^{\times} \hookrightarrow \tilde{A}$ which comes from conjugating the embedding $E^{\times} \hookrightarrow \GL_{2}(F) \hookrightarrow \widetilde{\GL_{2}(E)}$ by an element in $\widetilde{\SL_{2}(E)}$. We now identify $E^{\times}$ with its image inside $\tilde{A}$ which is given by $x \mapsto \left( \left( \begin{matrix} x & 0 \\ 0 & \bar{x} \end{matrix} \right), \epsilon(x) \right)$ where $\bar{x}$ is the non-trivial $Gal(E/F)$-conjugate of $x$ and $\epsilon(x) \in \{ \pm 1 \}$.  Conjugation by an element in $\widetilde{\SL_{2}(E)}$ ensures that this embedding of $E^{\times}$ in $\tilde{A}$ restricted to $F^{\times}$ is the one given by $f \mapsto (f,1)$ for all $f \in F^{\times}.$
Now let $\pi_{2}$ be any irreducible admissible representation of ${\rm GL}_2(F)$. By applying the functor $\Hom_{{\rm GL}_2(F)}(-, \pi_{2})$ to the short exact sequence (\ref{p.s. ses}), we get the following long exact sequence:
\begin{equation} \label{p.s. les}
\begin{array}{lllll}
0 & \rightarrow & \Hom_{{\rm GL}_2(F)} [\Ind_{B(F)}^{{\rm GL}_2(F)}(\tilde{\tau}|_{B(F)}\delta^{1/2}), \pi_{2}] & \rightarrow & \Hom_{{\rm GL}_2(F)} [\pi_{1}, \pi_{2}]  \\
  & \rightarrow & \Hom_{{\rm GL}_2(F)} [\ind_{E^{\times}}^{{\rm GL}_2(F)}(\tilde{\tau}'|_{E^{\times}}), \pi_{2}] & \rightarrow & \Ext_{{\rm GL}_2(F)}^{1}[\Ind_{B(F)}^{{\rm GL}_2(F)}(\tilde{\tau}|_{B(F)} \delta^{1/2}), \pi_{2}] \\
  & \rightarrow  & \cdots & &
\end{array}
\end{equation}
From \cite[~Corollary 5.9]{Prasad90} we know that 
\begin{center}
$\Hom_{{\rm GL}_2(F)} [\Ind_{B(F)}^{{\rm GL}_2(F)}(\chi.\delta^{1/2}), \pi_{2}] = 0$ \\
$\Updownarrow$ \\
$\Ext_{{\rm GL}_2(F)}^{1}[\Ind_{B(F)}^{{\rm GL}_2(F)}(\chi.\delta^{1/2}), \pi_{2}] = 0.$
\end{center}
Then from the exactness of (\ref{p.s. les}), it follows that 
\begin{center}
$\Hom_{{\rm GL}_2(F)} [\pi_{1},\pi_{2}] = 0$ \\
$\Updownarrow$\\
$\Hom_{{\rm GL}_2(F)} [\Ind_{B(F)}^{{\rm GL}_2(F)}(\tilde{\tau}|_{B(F)} \delta^{1/2}), \pi_{2}] = 0 \text{ and }
\Hom_{{\rm GL}_2(F)} [\ind_{E^{\times}}^{{\rm GL}_2(F)} (\tilde{\tau}'|_{E^{\times}}),\pi_{2}] = 0.$
\end{center}
Note that the representation $\Ind_{B(F)}^{{\rm GL}_2(F)}(\tilde{\tau}|_{B(F)})$ consists of finitely many principal series of ${\rm GL}_2(F)$. We have made the assumption that $ \Hom_{{\rm GL}_2(F)} [\Ind_{B(F)}^{{\rm GL}_2(F)}(\tilde{\tau}|_{B(F)}), \pi_{2}] = 0$, it follows that  
\[
 \Ext_{{\rm GL}_2(F)}^{1}[\Ind_{B(F)}^{{\rm GL}_2(F)}(\tilde{\tau}.\delta^{1/2}), \pi_{2}] = 0.
\]
This gives
\[ 
\begin{array}{lll}
 \Hom_{{\rm GL}_2(F)}[\pi_{1}, \pi_{2}] & \cong & \Hom_{{\rm GL}_2(F)} [\ind_{E^{\times}}^{{\rm GL}_2(F)}(\tilde{\tau}'|_{E^{\times}}), \pi_{2}] \\
                         & \cong & \Hom_{E^{\times}} [\tilde{\tau}'|_{E^{\times}}, \pi_{2}|_{E^{\times}}]         
\end{array}
\]
The following lemma describes $\tilde{\tau}'|_{E^{\times}}$.
\begin{lemma}
If we identify $E^{\times}$ with its image $\left\{ \left( \left( \begin{matrix} x & 0 \\ 0 & \bar{x} \end{matrix} \right), \epsilon(x) \right) \mid x \in E^{\times} \right\}$ inside $\tilde{A}$ as above then the subgroup $E^{\times} \cdot \tilde{A^2}$ inside $\tilde{A}$ is a maximal abelian subgroup. Moreover, $\tilde{\tau}'|_{E^{\times}}$ contains all the characters of $E^{\times}$ which are same as $\omega_{\tilde{\tau}}|_{E^{\times 2}}$ when restricted to $E^{\times 2}$, where $\omega_{\tilde{\tau}}$ is the central character of $\tilde{\tau}$.
\end{lemma}
\begin{proof}
From the explicit cocycle description and the non-degeneracy of quadratic Hilbert symbol, it is easy to verify that $E^{\times} \cdot \tilde{A^2}$ is a maximal abelian subgroup of $\tilde{A}$. The rest of the proof is similar to that of Lemma \ref{tau tilde res 2 Z tilde}.
\end{proof}

As $\pi_{2}$ is a discrete series representation, it is not always true (unlike what happens in case of a principal series representation) that any character of $E^{\times}$, whose restriction to $F^{\times}$ is the same as the central character of $\pi_{2}$, appears in $\pi_{2}$. Let $\pi_{2}'$ be the finite dimensional representation of $D_{F}^{\times}$ associated to $\pi_{2}$ by the Jacquet-Langlands correspondence. Considering the left action of $D_{F}^{\times}$ on $\mathbb{P}_{E}^{1} \cong \widetilde{\GL_{2}(E)}/\widetilde{B(E)}$ induced by $D_{F}^{\times} \hookrightarrow \widetilde{\GL_{2}(E)}$ it is easy to verify that $\mathbb{P}_{E}^{1} \cong D_{F}^{\times} / E^{\times}$. Then by Mackey theory, the principal series representation $\pi_{1}$ when restricted to $D_{F}^{\times}$, becomes isomorphic to $\ind_{E^{\times}}^{D_{F}^{\times}}(\tilde{\tau}'|_{E^{\times}})$. 
\[
\begin{array}{lll}
\Hom_{D_{F^{\times}}}[\pi_{1}, \pi_{2}']  & \cong & \Hom_{D_{F^{\times}}}[\ind_{E^{\times}}^{D_{F}^{\times}}(\tilde{\tau}'|_{E^{\times}}), \pi_{2}'] \\
                                 & \cong & \Hom_{E^{\times}}(\tilde{\tau}'|_{E^{\times}}, \pi_{2}'|_{E^{\times}})
\end{array}
\]

In order to prove 
\begin{equation} 
 \dim \Hom_{{\rm GL}_2(F)}[\pi_{1}, \pi_{2}] + \dim \Hom_{D_{F}^{\times}}[\pi_{1}, \pi_{2}'] = [E^{\times} : F^{\times}E^{\times 2}]
\end{equation}
we shall prove 
\begin{equation} \label{V pi2 pi2'}
 \dim \Hom_{E^{\times}} [\tilde{\tau}'|_{E^{\times}}, \pi_{2}|_{E^{\times}}] + \dim \Hom_{E^{\times}}(\tilde{\tau}'|_{E^{\times}}, \pi_{2}'|_{E^{\times}}) = [E^{\times} : F^{\times}E^{\times 2}].
\end{equation}
By Remark 2.9 in \cite{Prasad92}, a character of $E^{\times}$ whose restriction to $F^{\times}$ is the same as the central character of $\pi_{2}$ appears either in $\pi_{2}$ with multiplicity one or in $\pi_{2}'$ with multiplicity one, and exactly one of the two possibilities hold. Note that we are assuming that the two embeddings of $E^{\times}$, one via $\GL_{2}(F)$ and other via $D_{F}^{\times}$ are conjugate in $\widetilde{\GL_{2}(E)}$. Then the left hand side of Equation \ref{V pi2 pi2'} is the same as the number of characters of $E^{\times}$ appearing in $(\tilde{\tau}, V)$ which upon restriction to $F^{\times}$ coincide with the central character of $\pi_{2}$, which equals $\dim \Hom_{F^{\times}}(\tilde{\tau}|_{F^{\times}}, \omega_{\pi_2})$. We are reduced to the following lemma.
\begin{lemma}
 Let $(\tilde{\tau}, V)$ be an irreducible genuine representation of $\tilde{A}$ and, let $\chi$ be a character of $Z(F) = F^{\times}$ such that $\chi|_{E^{\times 2} \cap F^{\times}} = \tilde{\tau}|_{E^{\times 2} \cap F^{\times}}$. Then 
 \[
  \dim \Hom_{F^{\times}}(\tilde{\tau}, \chi) = [E^{\times} : F^{\times}E^{\times 2}].
 \]
\end{lemma}
\begin{proof}
Note that $E^{\times 2} \cap F^{\times} = Z^{\times 2} \cap F^{\times}$. From Lemma \ref{restriction to tilde Z}, $\tilde{\tau}|_{\tilde{Z}} \cong \Omega(\omega_{\pi_{1}})$. If a character $\mu \in \Omega(\omega_{\pi_{1}})$ is specified on $F^{\times}$ then it is specified on $F^{\times}E^{\times 2}$. Therefore the number of characters in $\Omega(\omega_{\pi_{1}})$ which agree with $\chi$ when restricted to $F^{\times}$ is equal to $[E^{\times} : F^{\times}E^{\times 2}]$. \qedhere

%Recall that if $\mu$ is a genuine character of $\tilde{Z}$ such that $\mu \in \Omega(\pi_{1})$ then all other genuine characters of $\tilde{Z}$ appearing in $\Omega(\pi_{1})$ are given by $\mu^{a}$ for $a \in E^{\times}/E^{\times 2}$, where $\mu^{a}(\tilde{z}) = (a,z) \mu(\tilde{z})$ for all $\tilde{z} \in \tilde{Z}$. Then dimension of $\Hom_{F^{\times}}(\tilde{\tau}, \chi)$ is same as the cardinality of the set $\{ a \in E^{\times}/E^{\times 2} \mid x \mapsto (x,a) \text{ is trivial character of } F^{\times} \}$. Now we calculate the cardinality of this set. Let $a \in E^{\times}$ such that $x \mapsto (x, a)_{E}$ is trivial character of $F^{\times}$. Note that $(x, a)_{E} = (x, \Norm(a))_{F}$ for $x \in F^{\times}$. By non-degeneracy of quadratic Hilbert symbol we get that $\Norm(a) \in F^{\times}$. Note that $\Norm : E^{\times} \longrightarrow F^{\times}$ induces $\overline{\Norm} : E^{\times}/E^{\times 2} \longrightarrow F^{\times}/F^{\times 2}$. The cardinality of the set which we want is same as $\# \ker(\overline{\Norm})$ the cardinality of the kernel of this $\overline{\Norm}$ map. By local class field theory, image of $\overline{\Norm}$ is of index 2 in $F^{\times}/F^{\times 2}$. Therefore
% \[
% \# \ker(\overline{\Norm}) = 2 \frac{[E^{\times} : E^{\times 2}]}{[F^{\times} : F^{\times 2}]}.  \qedhere
% \]
 \end{proof}

\section{Part 3 of Theorem \ref{DP-metaplctic}} \label{DP-metaplectic-Part3}
Let $\pi_{1}$ be an irreducible admissible genuine representation of $\widetilde{\GL_{2}(E)}$. We take another admissible genuine representation $\pi_{1}'$ having the same central character as that of $\pi_{1}$ and satisfying $(\pi_{1})_{N, \psi} \oplus (\pi_{1}')_{N, \psi} \cong \Omega( \omega_{\pi_{1}} )$ as $\tilde{Z}$-modules. From Proposition \ref{whittaker models of principal series}, if $\pi_{1}$ is a principal series representation then we can take $\pi_{1}' = 0$. We will see in Theorem \ref{not all WModel} that if $\pi_{1}$ is not a principal series representation then $(\pi_{1})_{N, \psi}$ is a proper $\tilde{Z}$-submodule of $\Omega( \omega_{\pi_{1}} )$ forcing $\pi_{1}' \neq 0$. In particular, if $\pi_{1}$ is one of the Jordan-H{\"o}lder factors of a reducible principal series representation then one can take $\pi_{1}'$ to be the other Jordan-H{\"o}lder factor of the principal series representation. For a supercuspidal representation $\pi_{1}$ we do not have any obvious choice for $\pi_{1}'$, and this issue will be taken up in the next chapter. \\

Let $\pi_{2}$ be a supercuspidal representation of ${\rm GL}_2(F)$. To prove Theorem \ref{DP-metaplctic} in this case, we use character theory and deduce the result by using the result of restriction of a principal series representation of $\widetilde{\GL_{2}(E)}$ which has already been proved in Section \ref{DP-metaplectic-Part2}. We can assume, if necessary after twisting by a character of $F^{\times}$, that $\pi_{2}$ is minimal representation. The representation $\pi_{2}$ is called minimal if the \index{conductor} conductor of $\pi_{2}$ is less than or equal to the conductor of $\pi_{2} \otimes \chi$ for any character $\chi$ of $F^{\times}$. Then by a theorem of Kutzko \cite{Kutzko78}, we can take $\pi_{2}$ to be $\ind_{\mathcal{K}}^{ {\rm GL}_2(F)}(W_2)$, where $W_2$ is a representation of a maximal compact modulo center subgroup $\mathcal{K}$ of ${\rm GL}_2(F)$. By Frobenius reciprocity, 
\[
\begin{array}{lll}
\Hom_{{\rm GL}_2(F)} \left( \pi_{1} \oplus \pi_{1}', \pi_{2} \right) & = & \Hom_{{\rm GL}_2(F)} \left( \pi_{1} \oplus \pi_{1}', \ind_{\mathcal{K}}^{{\rm GL}_2(F)}(W_2) \right) \\
                        & = & \Hom_{\mathcal{K}} \left( (\pi_{1} \oplus \pi_{1}')|_{\mathcal{K}}, W_2 \right). 
\end{array}
\]
To prove Theorem \ref{DP-metaplctic}, it suffices to prove that:
\[
 \dim \Hom_{\mathcal{K}}[(\pi_{1} \oplus \pi_{1}')|_{\mathcal{K}}, W_2] + \dim \Hom_{D_{F}^{\times}}[\pi_{1} \oplus \pi_{1}', \pi_{2}'] = [E^{\times} : F^{\times}E^{\times 2}].
\]
For any (virtual) representation $\pi$ of $\widetilde{{\rm GL}_2(E)}$, let $m(\pi, W_2) = \dim \Hom_{\mathcal{K}}[\pi|_{\mathcal{K}}, W_2]$ and $m(\pi,\pi_{2}') = \dim \Hom_{D_{F}^{\times}}[\pi, \pi_{2}']$. With these notations we will prove: 
\begin{equation} \label{Equation:Main-1}
 m(\pi_{1} \oplus \pi_{1}',W_2) + m(\pi_{1} \oplus \pi_{1}',\pi_{2}') = [E^{\times} : F^{\times}E^{\times 2}].
\end{equation}

Let $Ps$ be an irreducible principal series representation of $\widetilde{{\rm GL}_{2}(E)}$ whose central character $\omega_{Ps}$ is same as the central character $\omega_{\pi_{1}}$ of $\pi_{1}$ (it is clear that there exists one such). By Proposition \ref{whittaker models of principal series}, we know that $(Ps)_{N, \psi} \cong \Omega(\omega_{Ps})$ as a $\tilde{Z}$-module. On the other hand, the representation $\pi_{1}'$ has been chosen in such a way that $(\pi_{1})_{N, \psi} \oplus (\pi_{1}')_{N, \psi} = \Omega(\omega_{\pi_{1}})$ as $\tilde{Z}$-module. Then, as a $\tilde{Z}$-module we have
\[
(\pi_{1} \oplus \pi_{1}')_{N, \psi} = (\pi_{1})_{N, \psi} \oplus (\pi_{1}')_{N, \psi} = \Omega(\omega_{\pi_{1}}) = \Omega({\omega_{Ps}}) = (Ps)_{N, \psi}.
\]
We have already proved in Section \ref{DP-metaplectic-Part2} that
\[
 m(Ps,W_2) + m(Ps,\pi_{2}') = [E^{\times} : F^{\times}E^{\times 2}].
\]
In order to prove Equation \ref{Equation:Main-1}, we prove 
\begin{equation} \label{Equation:Main-2}
m(\pi_{1} \oplus \pi_{1}' -Ps,W_2) + m(\pi_{1} \oplus \pi_{1}' -Ps,\pi_{2}') = 0.
\end{equation}
The relation in Equation \ref{Equation:Main-2} follows form the following theorem:
\begin{theorem}
Let $\Pi_{1}, \Pi_{2}$ be two genuine representations of $\widetilde{\GL_{2}(E)}$ of finite length with a central character such that $(\Pi_{1})_{N, \psi} \cong (\Pi_{2})_{N, \psi}$ as $\tilde{Z}$-modules for a non-trivial additive character $\psi$ of $E$. Let $\pi_{2}$ be an irreducible supercuspidal representation of $\GL_{2}(F)$ such that the central characters $\omega_{\Pi_{1}}$ of $\Pi_{1}$ and $\omega_{\pi_{2}}$ of $\pi_{2}$ agree on $F^{\times} \cap E^{\times 2}$. Let $\pi_{2}'$ be the finite dimensional representation of $D_{F}^{\times}$ associated to $\pi_{2}$ by the Jacquet-Langlands correspondence. Then
\[
m(\Pi_{1} - \Pi_{2}, \pi_{2}) + m(\Pi_{1} - \Pi_{2}, \pi_{2}') = 0.
\]
\end{theorem}
We will use character theory to prove this relation following \cite{Prasad92} very closely. First of all, by Theorem \ref{Cass-Prasad-GL2}, $\Theta_{\Pi_{1} - \Pi_{2}}$ is given by smooth function on $\widetilde{\GL_{2}(E)}$. Now we recall the Weyl integration formula for ${\rm GL}_{2}(F)$.\\
\subsection{Weyl integration formula} \label{Weyl_I_F} \index{Weyl integration formula}
\begin{lemma} {\rm \cite[Formula 7.2.2]{JL70}} \label{Weyl-Integration-Formula} \\
For a smooth and compactly supported function $f$ on ${\rm GL}_{2}(F)$ we have
\begin{equation} \label{Expression-W-I-F}
\int_{{\rm GL}_{2}(F)} f(y) dy = \sum_{E_{i}} \int_{E_{i}} \bigtriangleup(x) \left( \frac{1}{2} \int_{E_{i} \backslash {\rm GL}_{2}(F)} f(\bar{g}^{-1}x \bar{g}) \, d\bar{g} \right) \, dx
\end{equation}
where the $E_{i}$'s are representatives for the distinct conjugacy classes of maximal tori in ${\rm GL}_{2}(F)$ and 
\[
\bigtriangleup(x) = \left| \left| \dfrac{(x_{1}- x_{2})^{2}}{x_{1}x_{2}} \right| \right|_{F}
\]
where $x_{1}$ and $x_{2}$ are the eigenvalues of $x$. 
\end{lemma}
We will use this formula to integrate the function $f(x) = \Theta_{\Pi_{1} - \Pi_{2}} \cdot \Theta_{W_{2}}(x)$ on $\mathcal{K}$ which is extended to ${\rm GL}_{2}(F)$ by setting it to be zero outside $\mathcal{K}$. In addition, we also need the following result of Harish-Chandra, cf. \cite[~Proposition 4.3.2]{Prasad92}.
\begin{lemma} [Harish-Chandra]  \label{lemma:HC}
Let $F(g) = (g v, v)$ be a matrix coefficient of a supercuspidal representation $\pi$ of a reductive $p$-adic group $G$ with center $Z$. Then the orbital integrals of $F$ at regular non-elliptic elements vanish. Moreover, the orbital integral \index{orbital integral} of $F$at a regular elliptic element $x$ contained in a torus $T$ is given by the formula
\begin{equation} \label{Harish-chandra theorem}
\int_{T \backslash G} F(\bar{g}x\bar{g}) d\bar{g} = \dfrac{(v,v) \cdot \Theta_{\pi} (x)}{d(\pi) \cdot {\rm vol}(T/Z)},
\end{equation}
where $d(\pi)$ denotes the formal degree of the representation $\pi$.
\end{lemma}
Since $\pi_{2}$ is obtained by induction from $W_{2}$, a matrix coefficient \index{matrix coefficient} of $W_{2}$ (extended to ${\rm GL}_{2}(F)$ by setting it to be zero outside $\mathcal{K}$) is also a matrix coefficient of $\pi_{2}$. It follows that 
\begin{enumerate}
\item for the choice of Haar measure on ${\rm GL}_{2}(F)/F^{\times}$ giving $\mathcal{K}/F^{\times}$ measure 1, we have 
\[
\dim W_{2} = d(\pi_{2}),
\]
\item for a separable quadratic field extension $E_{i}$ of $F$ and a regular elliptic element $x$ of $\GL_{2}(E)$ which generates $E_{i}$, and for the above Haar measure $d \bar{g}$,
\begin{equation} \label{noname-2}
\int_{E_{i}^{\times} \backslash {\rm GL}_{2}(F)} \Theta_{W_{2}}(\bar{g}^{-1} x \bar{g}) d\bar{g} = \dfrac{\Theta_{\pi_{2}}(x)}{{\rm vol}(E_{i}^{\times}/F^{\times})}. 
\end{equation}
\end{enumerate}

\subsection{Completion of the proof of Theorem \ref{DP-metaplctic}}
We recall the following important observation from Section \ref{Weyl_I_F} and Theorem \ref{Cass-Prasad-GL2}:
\begin{enumerate}
\item the virtual representation $(\Pi_{1} - \Pi_{2})|_{\mathcal{K}}$ is finite dimensional,
\item $\Theta_{W_{2}}$ is also a matrix coefficient of $\pi_{2}$ (extended to $\GL_{2}(F)$ by zero outside $\mathcal{K}$), 
\item there is Haar measure on $\GL_{2}(F)/F^{\times}$ giving $\vol(\mathcal{K}/F^{\times})=1$ such that the Equation \ref{noname-2} is satisfied.
\item the orbital integral in Equation \ref{Harish-chandra theorem} vanishes if $T$ is maximal split torus.  
\end{enumerate}
Let $E_{i}$'s be the quadratic extensions of $F$. Then these observations together with Lemma \ref{lemma:HC}, imply the following
\[
\begin{array}{rcl}
m(\Pi_{1} - \Pi_{2}, W_2) &=& \dfrac{1}{\vol(\mathcal{K}/F^{\times})} \bigint\limits_{\mathcal{K}/F^{\times}} \Theta_{\Pi_{1} - \Pi_{2}} \cdot \Theta_{W_{2}}(x) \, dx \\
                                                   &=& \dfrac{1}{\vol(\mathcal{K}/F^{\times})} \bigint\limits_{\GL_{2}(F)/F^{\times}} \Theta_{\Pi_{1} - \Pi_{2}} \cdot \Theta_{W_{2}}(x) \, dx \\
                                                   &=& \dfrac{1}{\vol(\mathcal{K}/F^{\times})} \sum\limits_{E_{i}} \bigint\limits_{E_{i}^{\times}/F^{\times}} \bigtriangleup(x) \left[ \dfrac{1}{2} \bigint\limits_{E_{i}^{\times} \backslash \GL_{2}(F)} \Theta_{\Pi_{1} - \Pi_{2}} \cdot \Theta_{W_{2}}(\bar{g}^{-1}x \bar{g}) \, dg \right] dx \\
                                                   &=& \sum\limits_{E_{i}} \dfrac{1}{2 \vol(E_{i}^{\times}/F^{\times})} \bigint\limits_{E_{i}^{\times}/F^{\times}} \left( \bigtriangleup \cdot \Theta_{\Pi_{1} - \Pi_{2}} \cdot \Theta_{\pi_{2}} \right) (x) dx.
\end{array}                                                    
\]
Similarly, we have the equality
\[
 m(\Pi_{1} - \Pi_{2}, \pi_{2}') = \sum_{E_i}\dfrac{1}{2 \vol(E_{i}^{\times}/F^{\times})} \bigint\limits_{E_{i}^{\times}/F^{\times}} \left( \bigtriangleup . \Theta_{\Pi_{1} - \Pi_{2}} . \Theta_{\pi_{2}'} \right) (x)dx.
\]
Note that $E_{i}$'s correspond to quadratic extensions of $F$ and the embeddings of $\GL_{2}(F)$ and $D_{F}^{\times}$ have been fixed so that the working hypothesis (as stated in the introduction of this chapter) is satisfied, i.e. the embeddings of the $E_{i}$'s in $\GL_{2}(F)$ and in $D_{F}^{\times}$ are conjugate in $\widetilde{\GL_{2}(E)}$. Then the value of $\Theta_{\Pi_{1} - \Pi_{2}}(x)$ for $x \in E_{i}$, does not depend on the inclusion of $E_{i}$ inside $\widetilde{\GL_{2}(E)}$, i.e. on whether inclusion is via $\GL_{2}(F)$ or via $D_{F}^{\times}$. Now using the relation $\Theta_{\pi_{2}}(x) = -\Theta_{\pi_{2}'}(x)$ on regular elliptic elements $x$ \cite[~Proposition 15.5]{JL70}, we conclude the following, which proves the Equation \ref{Equation:Main-2}
\[
 m(\Pi_{1} - \Pi_{2}, W_2) + m(\Pi_{1} - \Pi_{2}, \pi_{2}') = 0. 
\]

%\chapter{Representations of Covering groups: An overview}
%The study of representations of covering groups or metaplectic groups started with the paper of Andre Weil \cite{Weil64} where he gave representation theoretic version of modular forms of half-integral weight. One knows modular forms of integral weight correspond to automorphic forms of $\rm{GL}_{2}$. Weil proved that modular forms of half-integral weight correspond to automorphic forms on the two fold cover of ${\rm GL}_{2}$, thus giving rise to a metaplectic group. 

\chapter{Some consequences of Waldspurger's theorem} \label{consequences:Waldspurger}
\section{Introduction}
Let $E$ be a non-Archimedian local field of characteristic zero and $\psi$ a non-trivial character of $E$. Let $\tilde{\pi}_{1}$ and $\tilde{\pi}_{2}$ be two admissible genuine representations of $\widetilde{\GL_{2}(E)}$ with the same central character. By abuse of notation, we say that $\tilde{\pi}_{1}$ and $\tilde{\pi}_{2}$ have ``complementary Whittaker model" if $(\tilde{\pi}_{1})_{N, \psi} \oplus (\tilde{\pi}_{2})|_{N, \psi} = \Omega(\omega_{\tilde{\pi}_{1}})$ as $\tilde{Z}$-module for a non-trivial character $\psi$ of $E$. It can be easily seen that this notion does not depend on the choice of the character $\psi$ of $E$. By another abuse of notation, we say that a representation (of either of $\widetilde{\GL_{2}(E)}$ or of $\widetilde{\SL_{2}(E)}$) admits a $\psi$-Whittaker model if it admits a non-zero $\psi$-Whittaker functional. \\
Let $\tilde{\pi}$ be an irreducible admissible genuine supercuspidal representation of $\widetilde{\GL_{2}(E)}$. In this chapter, we wish to construct another representation $\tilde{\pi}'$ of $\widetilde{\GL_{2}(E)}$ with a central character such that 
\begin{enumerate}
\item the central characters of $\tilde{\pi}$ and $\tilde{\pi}'$ are same, and
\item $\tilde{\pi}$ and $\tilde{\pi}'$ have ``complementary Whittaker models".  
\end{enumerate}
By Proposition \ref{whittaker models of principal series}, if $\tilde{\pi}$ is an irreducible principal series representation of $\widetilde{\GL_{2}(E)}$ then $\tilde{\pi}' = 0$. We will prove in Proposition \ref{not all WModel} that if $\tilde{\pi}$ is an irreducible discrete series representation then $\tilde{\pi}_{N, \psi}$ is a proper $\tilde{Z}$-submodule of $\Omega(\omega_{\tilde{\pi}})$ and hence $\tilde{\pi}' \neq 0$. If $\tilde{\pi}$ is a Jordan-H{\"o}lder factor of a reducible principal series then $\tilde{\pi}'$ can be taken to be the other Jordan-H{\"o}lder factor. If $\tilde{\pi}$ is an irreducible supercuspidal representation then there is no obvious choice for $\tilde{\pi}'$. The main question which we take up in this chapter is the following.
\begin{question} \label{question: compl WM for GL2 tilde} 
Let $\tilde{\pi}$ is an irreducible admissible genuine supercuspidal representation of $\widetilde{\GL_{2}(E)}$ and $\psi$ a non-trivial additive character of $E$. Is there a `natural' choice of a genuine admissible representation of finite length $\tilde{\pi}'$ with a central character as that of $\tilde{\pi}$ (not necessarily irreducible) such that $\tilde{\pi}_{N, \psi} \oplus \tilde{\pi}_{N, \psi}' \cong \Omega( \omega_{\tilde{\pi}})$ as $\tilde{Z}$-modules ?
\end{question}
We are able to answer this question only for a certain class of representations $\tilde{\pi}$ which we will describe in Section \ref{complementary W model}. We are not able to describe a `natural' choice of $\tilde{\pi}'$ for all supercuspidal representations; it is not clear if the inability to do so is a reflection on us, or if there is a more fundamental reason for this inability. Recall from Section \ref{Reps of meta GL2} that $\tilde{\pi} = \ind_{\widetilde{\GL_{2}(E)_{+}}}^{\widetilde{\GL_{2}(E)}} (\mu \tau)$, where $\tau$ is an irreducible admissible genuine representation of $\widetilde{\SL_{2}(E)}$ and $\mu$ is a genuine character of $\tilde{Z}$ which is compatible with $\tau$. Further, from Equation \ref{restriction:2}, it is clear that Question \ref{question: compl WM for GL2 tilde} is equivalent to the following question.
\begin{question} \label{question: compl WM for SL2 tilde}
Let $\tau$ be an irreducible admissible genuine representation of $\widetilde{\SL_{2}(E)}$ and $\psi$ a non-trivial additive character of $E$.  Is there a `natural' choice of an genuine admissible representation of finite length $\tau'$ with a central character same as that of $\tau$ (non necessarily irreducible) such that $\tau$ admits a non-zero $\psi$-Whittaker functional if and only if $\tau'$ does not admit a non-zero $\psi$-Whittaker functional ?
\end{question}
\begin{remark}
Write the Waldspurger involution on $\widetilde{\SL_{2}(E)}$ as $\tau \leftrightarrow \tau_{W}$. It does not fix the isomorphism classes of any discrete series representation of $\widetilde{\SL_{2}(E)}$ (see next Section). Then $\tau$ admits a non-zero $\psi$-Whittaker model if and only if $\tau_{W}$ does not admit a  non-zero $\psi$-Whittaker model, and the central character of $\tau_{W}$ is opposite to that of $\tau$. On the other hand, in Question \ref{question: compl WM for GL2 tilde}, we require the same central character for $\tau$ and $\tau'$.
\end{remark}
\begin{remark} \label{some facts for WM}
 We note the following facts about the space of $\psi$-Whittaker functionals for $\tilde{\pi}$ and $\tau$, which follow easily from Theorem \ref{GHP79:uniquness}, Equation \ref{restriction:2}  and \ref{restriction:3}.
\begin{enumerate}
\item The dimension of the space of $\psi$-Whittaker functionals for an irreducible admissible genuine representation $\tau$ of $\widetilde{{\rm SL}_{2}(E)}$ is at most one dimensional.
\item There exists $a \in E^{\times}$ such that $\tau^{a}$ has a non-zero $\psi$-Whittaker functional.
\item $\tilde{\pi}$ has a non-trivial $(\mu^{a}, \psi)$-Whittaker functional if and only if $\tau^{a}$ has a non-trivial $\psi$-Whittaker functional. 
\end{enumerate}
\end{remark}
Another question which we take up in this chapter is the question of restriction of  an irreducible admissible genuine representation $\tilde{\pi}$ of $\widetilde{\GL_{2}(E)}$ to the subgroup $\widetilde{\SL_{2}(E)}$. We use the identification in Equation (\ref{restriction:3}) to calculate the multiplicity of a representation of $\widetilde{\SL_{2}(E)}$ in an irreducible admissible genuine representation of $\widetilde{\GL_{2}(E)}$. We prove that the multiplicity may be greater than one, in fact one of 1, 2 or 4. We make use of theta correspondence  and the Waldspurger involution  \cite{Wald91} to prove the results in this chapter.

\section{$\theta$-correspondence \index{$\theta$-correspondence} and the Waldspurger involution \index{Waldspurger involution}}
In this section, we recall some results of Waldspurger from \cite{Wald91}, related to the $\theta$-correspondence between $\widetilde{{\rm SL}_{2}(E)}$ and ${\rm PGL}_{2}(E)$ and that between $\widetilde{{\rm SL}_{2}(E)}$ and $PD^{\times}$, where $D$ is the unique quaternion division algebra over $E$. We will use these results repeatedly. \\
For any non-trivial additive character $\psi$ of $E$ one has the Weil index \index{Weil index} $\gamma(\psi)$ which is an eighth root of unity. For $a \in E^{\times}$, let $\psi_{a}$ be the additive character of $E$ defined by $\psi_{a}(x)=\psi(ax)$. For $a, b \in E^{\times}$, the Weil index satisfies the following property:
\begin{equation} \label{property: Weil index}
\gamma(\psi_{a}) \gamma(\psi_{b}) = (a,b) \gamma(\psi_{ab}) \gamma(\psi).
\end{equation}
Set $\gamma(a, \psi)= \gamma(\psi_{a})/\gamma(\psi)$ and define $\chi_{\psi} : \tilde{Z} \longrightarrow \C^{\times}$ by 
\begin{equation}
 \chi_{\psi}(a,\epsilon)= \epsilon \cdot \gamma(a, \psi).
 \end{equation}
By Equation \ref{property: Weil index}, $\chi_{\psi}$ is a genuine character of $\tilde{Z}$. 
\begin{definition}
Let $\tau$ be an irreducible admissible genuine representation of $\widetilde{{\rm SL}_{2}(E)}$. We define the central sign \index{central sign} $z_{\psi}(\tau)$ of $\tau$ by
\begin{equation}
z_{\psi}(\tau) = \omega_{\tau}(\widetilde{-1})/\chi_{\psi}(\widetilde{-1}) \in \{ \pm 1 \},
\end{equation}
where $\omega_{\tau}$ denotes the central character of $\tau$ and $\widetilde{-1}$ denotes any element of $\tilde{Z}$ lying over $-1 \in Z$. Note that the quotient above does not depend on the choice of $\tilde{Z}$. \\
\end{definition}
It follows from the definition of the central sigh that 
\begin{equation} \label{equation: central sign}
z_{\psi}(\tau^{x}) = z_{\psi}(\tau) \chi_{x}(-1) = z_{\psi}(\tau) (x, -1)
\end{equation} 
for any $x \in E^{\times}$, where $\chi_{x}(z):=(x, z)$ for $z \in E^{\times}$.
\begin{definition} 
Let $\tau_{1}$ and $\tau_{2}$ be two irreducible admissible genuine representations of $\widetilde{{\rm SL}_{2}(E)}$. We say that $\tau_{1}$ and $\tau_{2}$ have opposite central characters if $z_{\psi}(\tau_{1}) = - z_{\psi}(\tau_{2})$.
\end{definition}

Now fix a non-trivial additive character $\psi$ of $E$. With respect to this choice of $\psi$, one has a $\theta$-correspondence between the isomorphism classes of irreducible admissible genuine representations of $\widetilde{{\rm SL}_{2}(E)}$ and the isomorphism classes of irreducible admissible representations of ${\rm PGL}_{2}(E)$ 
\[ \xymatrix{
{\rm Irrep}(\widetilde{{\rm SL}_{2}(E)}) \ar[r]^{\theta(-, \psi)} & {\rm Irrep}({\rm PGL}_{2}(E)) }
\]
as well as between irreducible admissible genuine representations of $\widetilde{{\rm SL}_{2}(E)}$ and irreducible admissible representations of $PD^{\times}$ 
\[ \xymatrix{
{\rm Irrep}(\widetilde{{\rm SL}_{2}(E)}) \ar[r]^{\theta(-,\psi)} & {\rm Irrep}(PD^{\times}).  }
\]
Though this correspondence $\tau \mapsto \theta(\tau,\psi)$ depends on the choice of $\psi$, it will be abbreviated to $\tau \mapsto \theta(\tau)$ as $\psi$ has been fixed. The $\theta$-correspondence between $\widetilde{{\rm SL}_{2}(E)}$ and ${\rm PGL}_{2}(E)$ gives a one to one mapping from the set of isomorphism classes of irreducible admissible genuine representations of $\widetilde{{\rm SL}_{2}(E)}$ which have a $\psi$-Whittaker model onto the set of isomorphism classes of all irreducible admissible representations of ${\rm PGL}_{2}(E)$. Similarly, the $\theta$-correspondence between $\widetilde{{\rm SL}_{2}(E)}$ and $PD^{\times}$ gives a one to one mapping from the set of isomorphism classes of irreducible admissible genuine representations of $\widetilde{{\rm SL}_{2}(E)}$ which do not have a $\psi$-Whittaker model onto the set of isomorphism classes of all irreducible representations of $PD^{\times}$. Thus the $\theta$-correspondence defines a bijection (which depends on the choice of $\psi$):
\begin{equation}
{\rm Irrep}(\widetilde{{\rm SL}_{2}(E)}) \longleftrightarrow {\rm Irrep}({\rm PGL}_{2}(E)) \sqcup {\rm Irrep}(PD^{\times}).
\end{equation}
Now we can describe the Waldspurger involution \index{Waldspurger involution} \cite{Wald91} $W : {\rm Irrep}(\widetilde{\SL_{2}(E)}) \rightarrow {\rm Irrep}(\widetilde{\SL_{2}(E)})$ which is defined using 
\begin{enumerate} 
\item the $\theta$-correspondence from $\widetilde{{\rm SL}_{2}(E)}$ to ${\rm PGL}_{2}(E)$, 
\item the $\theta$-correspondence from $\widetilde{{\rm SL}_{2}(E)}$ to $PD^{\times}$ and 
\item the Jacquet-Langlands correspondence \index{Jacquet-Langlands correspondence} viewed as a map from ${\rm Irr}({\rm PGL}_{2}(E)) \sqcup {\rm Irr}(PD^{\times})$ to itself.
\end{enumerate} 
The Waldspurger involution is the unique map $W : {\rm Irrep}(\widetilde{{\rm SL}_{2}(E)}) \rightarrow {\rm Irrep}(\widetilde{{\rm SL}_{2}(E)})$ that makes the following diagram commutative:
\[
\xymatrix{
{\rm Irrep}(\widetilde{{\rm SL}_{2}(E)}) \ar@{<->}[d]_{W} \ar[rr]^{\theta} & & {\rm Irrep}({\rm PGL}_{2}(E)) \sqcup {\rm Irrep}(PD^{\times}) \ar@{<->}[d]^{J-L} \\
{\rm Irrep}(\widetilde{{\rm SL}_{2}(E)}) \ar[rr]^{\theta} & &  {\rm Irrep}({\rm PGL}_{2}(E)) \sqcup {\rm Irrep}(PD^{\times}) 
} \label{diagram:involution}
\]
This involution is defined on the set of all representations of $\widetilde{{\rm SL}_{2}(E)}$, and its fixed points are precisely the irreducible admissible genuine representations which are not discrete series representations. Denote this involution by $\tau \mapsto \tau_{W}$. This involution is independent of the character $\psi$ chosen to define it \cite{Wald91}. For $\tau \in {\rm Irrep}(\widetilde{\SL_{2}(E)})$, we say that $\tau$ has a non-zero $\theta$ lift to ${\rm PGL}_{2}(E)$ (respectively, $PD^{\times}$) if $\theta(\tau) \neq 0$ in ${\rm PGL}_{2}(E)$ (respectively, $PD^{\times}$). For $\pi \in {\rm Irrep}({\rm PGL}_{2}(E))$, let $\epsilon(\pi)$ denote the value at $\frac{1}{2}$ of standard $\epsilon$-factor, \index{$\epsilon$-factor} i.e. $\epsilon(\pi, \frac{1}{2}, \psi)$.

\begin{theorem} [Waldspurger \cite{Wald91}] \label{theorem:B}
Let $\tau$ be an irreducible admissible genuine representation of $\widetilde{{\rm SL}_{2}(E)}$. Let $\psi$ be a non-trivial additive character of $E$.  Then
\begin{enumerate}
\item $\tau$ has a $\psi$-Whittaker model if and only if $\tau_{W}$ does not have a $\psi$-Whittaker model. Moreover, $\tau$ and $\tau_{W}$ have opposite central characters.
\item $\tau$ has a non-zero $\theta$ lift to ${\rm PGL}_{2}(E)$ with respect to $\psi$ if and only if one of the following equivalent conditions is satisfied: 
\begin{enumerate}
\item $z_{\psi}(\tau) = \epsilon(\theta(\tau, \psi))$.
\item $\tau$ has a $\psi$-Whittaker model.
\item $\tau_{W}$ does not have a $\psi$-Whittaker model.
\end{enumerate}
\item $\tau$ has a non-zero $\theta$ lift to $PD^{\times}$ with respect to $\psi$ if and only if one of the following equivalent conditions is satisfied: 
\begin{enumerate}
\item $z_{\psi}(\tau) = - \epsilon(\theta(\tau, \psi))$.
\item $\tau$ does not have $\psi$-Whittaker model.
\item $\tau_{W}$ has $\psi$-Whittaker model.
\end{enumerate}
\end{enumerate}
\end{theorem}

\begin{theorem} [Waldspurger \cite{Wald91}] \label{theorem:A}
Let $\tau$ be an irreducible admissible genuine representation of $\widetilde{{\rm SL}_{2}(E)}$ and $\psi$ a non-trivial additive character of $E$. Then
\begin{enumerate}
\item For $a \in E^{\times}$, let $\chi_{a}$ be the quadratic character of $E^{\times}$ defined by $\chi_{a}(x)=(a,x)$. Both the representations $\tau$ and $\tau^{a}$ of $\widetilde{{\rm SL}_{2}(E)}$ have a non-zero $\theta$ lift (with respect to the character $\psi$) either to ${\rm PGL}_{2}(E)$ or to $PD^{\times}$ if and only if 
\[
\epsilon(\theta(\tau) \otimes \chi_{a}) = \chi_{a}(-1) \epsilon(\theta(\tau)),
\]
 and if this condition is satisfied, 
 \[
 \theta(\tau^{a}) \cong \theta(\tau) \otimes \chi_{a}.
 \]
If $\epsilon(\theta(\tau) \otimes \chi_{a}) = - \chi_{a}(-1) \epsilon(\theta(\tau))$, then $\theta(\tau)$ is a representation of ${\rm PGL}_{2}(E)$ if and only if $\theta(\tau^{a})$ is a representation of $PD^{\times}$; and 
\[
\theta(\tau^{a}) = \theta(\tau)^{JL} \otimes \chi_{a}.
\]
\item For $a \in E^{\times}$, let $\psi_{a}$ be the additive character of $E$ given by $\psi_{a}(x)= \psi(ax)$. Assume that $\tau$ admits a $\psi$-Whittaker model. Then the following conditions are equivalent 
\begin{enumerate}
\item $\tau$ admits a $\psi_{a}$-Whittaker model.
\item $\epsilon(\pi \otimes \chi_{a}) = \chi_{a}(-1) \epsilon(\pi)$.
\item $\theta(\tau^{a}, \psi_{a}) = \theta(\tau, \psi)$.
\end{enumerate}
\end{enumerate}
\end{theorem}

\section{Higher multiplicity in restriction from $\widetilde{{\rm GL}_{2}(E)}$ to $\widetilde{{\rm SL}_{2}(E)}$} \label{higher multiplicity}
Let $\tilde{\pi}$ be an irreducible admissible genuine representation of $\widetilde{{\rm GL}_{2}(E)}$. Let $\mu$ be a character of $\tilde{Z}$ and $\tau$ an irreducible representation of $\widetilde{{\rm SL}_{2}(E)}$, which are compatible, such that $\mu\tau$ appears in $\tilde{\pi}$ restricted to $\widetilde{{\rm GL}_{2}(E)}_{+}$. We have
 \[
 \tilde{\pi}|_{\widetilde{{\rm GL}_{2}(E)}_{+}} = \bigoplus_{a \in E^{\times}/E^{\times 2}} (\mu^{a} \tau^{a})
 \]
 where $a \in E^{\times}/E^{\times 2}$ is regarded as an elements in the split torus ($\cong E^{\times} \times E^{\times}$) of the form $diag(a,1)$. Since the restriction of $\mu \tau$ from $\widetilde{{\rm GL}_{2}(E)}_{+}$ to $\widetilde{{\rm SL}_{2}(E)}$ is $\tau$, the multiplicity with which the representation $\tau$ appears in $\tilde{\pi}$, to be denoted by $m(\tilde{\pi}, \tau)$, is given by
 \[
m(\tilde{\pi}, \tau) = \# \{ a \in E^{\times}/E^{\times 2} : \tau^{a} \cong \tau \}.
 \]
We have the following immediate corollaries to part 1 of  Theorem \ref{theorem:A}. 
\begin{lemma} \label{lemma:A1}
For an irreducible admissible genuine representation $\tau$ of $\widetilde{\SL_{2}(E)}$, and $a \in E^{\times}$, we have
\[
\tau \cong \tau^{a} \Longleftrightarrow \left\{ \begin{array}{lrl}
     (1) & \theta(\tau) \otimes \chi_{a} & \cong  \theta(\tau) \\
     (2) & \chi_{a}(-1) & =  1.
                                                                         \end{array}
                                                                         \right.
\]
\end{lemma}
\begin{proof}
If $\tau \cong \tau^{a}$, then considering the central characters on both sides, we find that $\chi_{a}(-1) = 1$. Further, if $\tau \cong \tau^{a}$, then in particular, either they both have $\theta$ lifts to ${\rm PGL}_{2}(E)$ or they both have $\theta$ lift to $PD^{\times}$, and $\theta(\tau) \cong \theta(\tau^{a})$. Thus from part 1 of Theorem \ref{theorem:A}, we get $\theta(\tau) \otimes \chi_{a}  \cong  \theta(\tau)$. To prove the converse, note that $\theta(\tau) \otimes \chi_{a} \cong \theta(\tau) \Rightarrow \epsilon(\theta(\tau) \otimes \chi_{a}) = \epsilon (\theta(\tau))$. As $\chi_{a}(-1)=1$, we get $ \epsilon(\theta(\tau) \otimes \chi_{a}) = \chi_{a}(-1)\epsilon (\theta(\tau))$. From {\it loc. cit.},  $\theta(\tau) = \theta(\tau^{a})$  and hence $\tau \cong \tau^{a}$.
\end{proof}
\begin{corollary} \label{corollary:A2} The multiplicity \index{multiplicity} of $\tau$ in $\tilde{\pi}$ is given by
\[
m(\tilde{\pi}, \tau) = \# \left\{ a \in E^{\times}/E^{\times 2} : \theta(\tau) \otimes \chi_{a} \cong \theta(\tau) \text{ and } \chi_{a}(-1)=+1 \right\}. 
\]
\end{corollary}
\noindent It is well-known that for a representation $\pi$ of ${\rm GL}_{2}(E)$, cf. \cite[~Section 2]{LL79}
\[
m(\pi) = \# \{ a \in E^{\times}/E^{\times 2} : \pi \cong \pi \otimes \chi_{a} \} \in \{1, 2, 4 \}.
\]
The condition $\chi_{a}(-1) = 1$ is automatic in some situations, for example if $ -1 \in E^{\times 2}$. Thus we get 
\[
m(\tilde{\pi}, \tau) \in \{ 1, 2, 4 \}
\] 
even when the residue characteristic of $E$ is 2.

\section{A lemma on Waldspurger involution}
 We recall that for an irreducible admissible genuine representation $\tau$ of $\widetilde{{\rm SL}_{2}(E)}$, the central characters of $\tau$ and $\tau_{W}$ are different. The group ${\rm GL}_{2}(E)$ acts on the set of isomorphism classes of irreducible admissible genuine representations of $\widetilde{{\rm SL}_{2}(E)}$ by conjugation. This action reduces to an action of $E^{\times}$, by identifying $E^{\times}$ into ${\rm GL}_{2}(E)$ as $\left\{ \left( \begin{matrix}e & 0 \\ 0 & 1 \end{matrix} \right) : e \in E^{\times} \right\}$. We denote this action by $\tau \mapsto \tau^{a}$  for $a \in E^{\times}$. Since a similar action produces an $L$-packet \index{$L$-packet} for $\SL_{2}(E)$, whereas for $\widetilde{\SL_{2}(E)}$, one defines an $L$-packet by taking $\tau$ and $\tau_{W}$, we investigate in this section if it can happen that $\tau_{W} \cong \tau^{a}$ for some $ a \in E^{\times}$ and $\tau$ a discrete series representation of $\widetilde{{\rm SL}_{2}(E)}$.
\begin{lemma} \label{lemma:B}
Let $\tau$ be a discrete series representation of $\widetilde{{\rm SL}_{2}(E)}$. Let $\psi$ be a non-trivial additive character of $E$ such that $\tau$ has a $\psi$-Whittaker model. Then there exists $a \in E^{\times}$ with $\tau^{a} \cong \tau_{W}$ if and only if for $\pi=\theta(\tau, \psi)$, we have 
\begin{enumerate}
\item[(i)] $\pi \cong \pi \otimes \chi_{a}$
\item[(ii)] $\chi_{a}(-1)=-1$.
\end{enumerate}
\end{lemma}
\begin{proof}
Let $\pi = \theta(\tau, \psi)$ and $\theta(\tau_{W}, \psi)= \pi^{JL}$, where $\pi^{JL}$ denotes the representation of $PD^{\times}$ which is associated to $\pi$ via the \index{Jacquet-Langlands correspondence} Jacquet-Langlands correspondence.  From part 2 of Theorem \ref{theorem:A} it follows that if $\epsilon(\pi \otimes \chi_{a}) = \chi_{a}(-1) \epsilon(\pi)$, then $\tau^{a}$ lifts to ${\rm PGL}_{2}(E)$ and not to $PD^{\times}$ and hence $\tau^{a}$ cannot be isomorphic to $\tau_{W}$. Thus if $\tau^{a}$ were isomorphic to $\tau_{W}$, then we must have $\epsilon(\pi \otimes \chi_{a}) = - \chi_{a}(-1) \epsilon(\pi)$. In this case, by Theorem \ref{theorem:A}, $\tau^{a}$ lifts to $PD^{\times}$, and in fact to the representation  $\pi^{JL} \otimes \chi_{a}$ of $PD^{\times}$. Therefore 
\begin{equation} \label{equation:tau iso tau_W}
\tau^{a} \cong \tau_{W} \Longleftrightarrow \left\{ \begin{array}{lrcl}
(i) & \epsilon(\pi \otimes \chi_{a}) &=& - \chi_{a}(-1) \epsilon(\pi) \\
(ii) & \pi^{JL} &\cong& \pi^{JL} \otimes \chi_{a}.
\end{array} \right.
\end{equation}
The conditions $(i)$ and $(ii)$ in \ref{equation:tau iso tau_W} can be combined to say that
\[
\tau^{a} \cong \tau_{W} \Longleftrightarrow \left\{ \begin{array}{lrl}
 (i) & \pi & \cong \pi \otimes \chi_{a} \\
 (ii) & \chi_{a}(-1) & = -1. 
 \end{array}
 \right. 
\]
This completes the proof of the lemma.
\end{proof}
As a consequence of Lemma \ref{lemma:A1} and Lemma \ref{lemma:B}, we obtain:
\begin{corollary}
Let $\tau$ be an irreducible genuine discrete series representation of $\widetilde{\SL}_{2}(E)$. Let $m_{1} = \# \{ \tau^{a}, (\tau_{W})^{a} \mid a \in E^{\times} \}$ and let $m_{2}$ be the cardinality of the $L$-packet of $\SL_{2}(E)$ determined by $\theta(\tau, \psi)$. Then
\[
m_{1} \cdot m_{2} = 2 [E^{\times} : E^{\times 2}].
\]
\end{corollary}
If $\pi$ is a principal series representation of ${\rm PGL}_{2}(E)$ with $\pi \otimes \chi_{a} \cong \pi$, then $\pi$ must be the principal series representation $Ps(\mu, \mu \chi_{a})$ with $\mu^{2} = \chi_{a}$, and as a result 
\[
\chi_{a}(-1) = \mu^{2}(-1) = 1.
\]
\begin{corollary}
Let $\tau$ is an irreducible admissible genuine representation of $\widetilde{\SL}_{2}(E)$ such that $\theta(\tau)$ an irreducible principal series representation of ${\rm PGL}_{2}(E)$. Let $m_{1} = \# \{ \tau^{a} \mid a \in E^{\times} \}$, and $m_{2}$ the cardinality of the $L$-packet of $\SL_{2}(E)$ determined by $\theta(\tau)$,
\[
m_{1} \cdot m_{2} = [E^{\times} : E^{\times 2}]. 
\]
\end{corollary}

\section{Complementary Whittaker models} \label{complementary W model}
Let $\tilde{\pi}$ be an irreducible admissible genuine representation of $\widetilde{{\rm GL}_{2}(E)}$. By Theorem \ref{GHP79:uniquness}, we know that the space of all $\psi$-Whittaker functionals is finite dimensional and that the characters of $\tilde{Z}$ appear with multiplicity at most one in this space of $\psi$-Whittaker functionals. The characters of $\tilde{Z}$ which appear in the space of $\psi$-Whittaker functionals are extensions of the central characters of the representation $\tilde{\pi}$. In other words, $\tilde{\pi}_{N, \psi} \subset \Omega(\omega_{\tilde{\pi}})$. Thus there are at most $\#(E^{\times}/E^{\times 2})$ characters appearing in the space of $\psi$-Whittaker functionals of $\tilde{\pi}$. We know that if  $\tilde{\pi}$ is a principal series representation then $\tilde{\pi}_{N, \psi} \cong \Omega(\omega_{\tilde{\pi}})$, by Theorem \ref{whittaker models of principal series}. 
\begin{proposition} \label{not all WModel}
If $\tilde{\pi}$ is a discrete series representation of $\widetilde{\GL_{2}(E)}$ then as a $\tilde{Z}$-module 
\[
\tilde{\pi}_{N, \psi} \subsetneq \Omega(\omega_{\tilde{\pi}}).
\]
\end{proposition} 
\begin{proof}
%Let $\tau$ be an irreducible genuine representation of $\widetilde{\SL_{2}(E)}$ and $\mu$ be a genuine character of $\tilde{Z}$ such that $\tilde{\pi} = \ind_{\widetilde{\GL_{2}(E)}_{+}}^{\widetilde{\GL_{2}(E)}} (\mu \tau)$. First we prove that $\tilde{\pi}$ is a discrete series representation if and only if $\tau$ is a discrete series representation. To see this we first assume that $\tau$ is a discrete series representation. Then all matrix coefficients of $\tau^{a}$ are compactly supported on $\widetilde{\SL_{2}(E)}$ for all $a \in E^{\times}/E^{\times 2}$ and hence all matrix coefficients of $(\mu \tau)$ are compactly supported modulo $\tilde{Z}$ (and hence modulo $\tilde{Z^{2}}$) on $\widetilde{\GL_{2}(E)_{+}}$. As $[ \widetilde{\GL_{2}(E)} : \widetilde{\GL_{2}(E)_{+}} ] < \infty$, support of a matrix coefficient $f$ of $\tilde{\pi}$ restricted to $\widetilde{\GL_{2}(E)}_{+}$ is finite sum of matrix coefficients of $\mu \tau$ and hence support of $f|_{\widetilde{\GL_{2}(E)}}$ is compact modulo $\tilde{Z^2}$. Support of $f$ is contained in a finite translate of support of $f|_{\widetilde{\GL_{2}(E)}}$ and hence compactly supported modulo $\tilde{Z^{2}}$. Thus $\tilde{\pi}$ is also a discrete series. On the other hand, assume $\tilde{\pi}$ is a discrete series representation. As a matrix coefficient of $\tau$ is also a matrix coefficient of $\tilde{\pi}$ restricted to $\widetilde{\SL_{2}(E)}$ and hence compactly supported. Thus $\tau$ is a discrete series representation. \\ 
Write $\tilde{\pi} = \ind_{\widetilde{\GL_{2}(E)_{+}}}^{\widetilde{\GL_{2}(E)}}(\mu \tau)$. Observe that $\tilde{\pi}$ is a discrete series representation if and only $\tau$ is a discrete series representation. Consider the set $\{ \tau, \tau_{W} \}$. Since the Waldspurger involution does not fix any discrete series representation, $\tau \ncong \tau_{W}$. Let $\psi$ be a non-trivial additive character of $E$ such that $\tau_{W}$ admits a $\psi$-Whittaker model. Then $\tau$ does not have a $\psi$-Whittaker model. Since $\tilde{\pi} = \ind_{\widetilde{\GL_{2}(E)_{+}}}^{\widetilde{\GL_{2}(E)}}(\mu \tau)$, we have $\tilde{\pi}|_{\widetilde{\GL_{2}(E)}_{+}} = \bigoplus_{a \in E^{\times}/E^{\times 2}} \mu^{a} \tau^{a}$ with $\mu^{a} = \mu \cdot \chi_{a}$ where the $\chi_{a}$, defined by $\chi_{a}(x) = (x, a)$,  are {\it distinct} characters of $E^{\times}$. Therefore $\mu$ does not appear  $\tilde{\pi}_{N, \psi}$. 
\end{proof}

\subsection{Case 1: $-1 \in E^{\times 2}$}
In this subsection, suppose $-1 \in E^{\times}$. Suppose that a genuine character $\mu$ of $\tilde{Z}$ and an irreducible admissible genuine representation $\tau$ of $\widetilde{\SL_{2}(E)}$ are compatible. As $-1 \in E^{\times 2}$, $\mu^{a}$ is also compatible with $\tau$ for all $a \in E^{\times}$.
\begin{lemma}
Let $-1 \in E^{\times 2}$ and let $\psi$ be a non-trivial additive character of $E$. Let $\tilde{\pi} = \ind_{\widetilde{\GL_{2}(E)}_{+}}^{\widetilde{\GL_{2}(E)}} (\mu \tau)$. For $a \in E^{\times}/E^{\times 2}$, if we write $\tilde{\pi}_{a} := \ind_{\widetilde{\GL_{2}(E)_{+}}}^{\widetilde{\GL_{2}(E)}} (\mu^{a} \tau)$, then for all $\mu \in \Omega(\omega_{\tilde{\pi}})$, the multiplicity of $\mu$ in $(\oplus_{a \in E^{\times}/E^{\times 2}} \tilde{\pi}_{a})_{N, \psi}$ is $\dim \tilde{\pi}_{N, \psi}$. In particular, if $\tilde{\pi}_{N, \psi}$ is one dimensional then $(\oplus_{a \in E^{\times}/E^{\times 2}} \tilde{\pi}_{a})_{N, \psi} \cong \Omega(\omega_{\tilde{\pi}})$.
\end{lemma}
\begin{proof}
For $\mu \in \Omega(\omega_{\tilde{\pi}})$, it is clear that $\mu$ appears in $\tilde{\pi}_{N, \psi}$ if and only if $\mu^{a} \in (\tilde{\pi}_{a})_{N, \psi}$. The lemma follows easily by Remark \ref{some facts for WM}.
\end{proof}
Now we assume that  residue characteristic of $E$ is odd, so we have $\#(E^{\times}/E^{\times 2})=4$. 
\begin{proposition}
Let $-1 \in E^{\times 2}$ and suppose that the residual characteristic of $E$ is odd. Let $\tilde{\pi}$ be an irreducible admissible genuine representation of $\widetilde{{\rm GL}_{2}(E)}$ such that $\dim \tilde{\pi}_{N, \psi} = 2$. Assume that $\tilde{\pi} := \ind_{\widetilde{{\rm GL}_{2}(E)}_{+}}^{\widetilde{{\rm GL}_{2}(E)}} (\mu \tau)$ for some compatible $\mu$ and $\tau$ such that $\tau$ admits a non-zero $\psi$-Whittaker functional. Then there exists $b \in E^{\times} - E^{\times 2}$ such that for $ \tilde{\pi}' := \ind_{\widetilde{{\rm GL}_{2}(E)}_{+}}^{\widetilde{{\rm GL}_{2}(E)}} (\mu^{b} \tau)$ we have
\[
(\tilde{\pi})_{N, \psi} \oplus (\tilde{\pi}_{b})_{N, \psi} \cong \Omega(\omega_{\tilde{\pi}}).
\]
\end{proposition}
\begin{proof}
Write $E^{\times}/E^{\times 2} = \{ 1, a, b, ab \}$. Assume $(\tilde{\pi})_{N, \psi} =  \mu \oplus \mu^{a}$. Then each of $\tau^{a}$ and $\tau$ admits a non-zero $\psi$-Whittaker functional. Equivalently, $\tau$ admits a non-zero $\psi$-Whittaker functional as well as a non-zero $\psi_{a}$-Whittaker functional. Therefore $\tau^{b}$ and $\tau^{ab}$ have $\psi_{b}$ and $\psi_{ab}$-Whittaker models. Therefore for
\[
\tilde{\pi}' := ind_{\widetilde{{\rm GL}_{2}(E)}_{+} }^{\widetilde{{\rm GL}_{2}(E)}} (\mu^{b} \tau),
\]
Then, by Remark \ref{some facts for WM}, we have
$(\tilde{\pi}')_{N, \psi} = \mu^{b} \oplus \mu^{ab}$. Therefore we have 
\[
(\tilde{\pi})_{N, \psi} \oplus (\tilde{\pi}_{b})_{N, \psi} \cong (\mu \oplus \mu^{a}) \oplus (\mu^{b} \oplus \mu^{ab}) = \Omega(\omega_{\tilde{\pi}}).
\]
Thus $\tilde{\pi}'$ is a representation of $\widetilde{{\rm GL}_{2}(E)}$ which has the same central character as that of $\tilde{\pi}$ and complementary Whittaker model to that of $\tilde{\pi}$.
\end{proof}

\subsection{Case 2: $-1 \notin E^{\times 2}$}
In this subsection we assume that $-1 \notin E^{\times 2}$.
% Let $\tau_{1}$ and $\tau_{2}$ be two irreducible admissible genuine representations of $\widetilde{\SL_{2}(E)}$. We abuse the notation and say that $\tau_{1}$ and $\tau_{2}$ have the same Whittaker model if for any non-trivial character $\psi$, $\tau_{1}$ admits a non-zero $\psi$-Whittaker functional if and only if $\tau_{2}$ admits a non-zero $\psi$-Whittaker functional. We also say that 
\begin{proposition} \label{propsitionB}
 Let $\tau$ be an irreducible admissible supercuspidal genuine representation of $\widetilde{{\rm SL}_{2}(E)}$. Assume that $p$ is odd and that $-1$ is not a square in $E$. Let $\psi$ be a non-trivial character of $E$ such that $\tau$ admits $\psi$-Whittaker model. Assume that for $\pi = \theta(\tau, \psi)$, $\pi \cong \pi \otimes \chi_{b}$ where $\chi_{b}$ corresponds to a quadratic ramified extension of $E$. Then for $a=-b$, the representations $\tau$ and $\tau^{a}$ have opposite central characters. Moreover, for any non-trivial character $\psi'$ of $E$, $\tau$ admits a non-zero $\psi'$-Whittaker functional if and only if $\tau^{a}$ admits a non-zero $\psi'$-Whittaker functional. 
 %Thus, $\tau^{W}$ and $\tau^{a}$ have the same central character and complementary Whittaker models.
\end{proposition}
\begin{proof}
As $\chi_{b}$ corresponds to a quadratic ramified extension of $E$, for $a=-b$ we have $\chi_{a}(-1) = -1$. Hence, $\tau$ and $\tau^{a}$ have opposite central character. Therefore, it is enough to show that the following holds for all $x \in E^{\times}$:
\begin{equation} \label{case2:A} 
\tau \text{ has } \psi^{x}\text{-Whittaker model } \Longleftrightarrow \tau^{a} \text{ has } \psi^{x}\text{-Whittaker model. } 
\end{equation}
The condition in \ref{case2:A} translates into
\[
\begin{array}{rcl}
\theta(\tau, \psi^{x}) \neq 0 & \Longleftrightarrow & \theta(\tau^{a}, \psi^{x}) \neq 0, \hspace{1cm} \forall x \in E^{\times} \\
\text{i.e.,}  \hspace{1cm}\theta(\tau^{x}, \psi) \neq 0 & \Longleftrightarrow & \theta(\tau^{ax}, \psi) \neq 0, \hspace{1cm} \forall x \in E^{\times}. \label{case2:B} 
\end{array}
\]
Let $V^{+}$ and $V^{-}$ be 3-dimensional quadratic spaces such that $O(V^{+}) = {\rm PGL}_{2}(E) \times \{ \pm 1 \}$ and $O(V^{-}) = PD^{\times} \times \{ \pm 1 \}$. Set $\epsilon(V^{+})=1$ and $\epsilon(V^{-})=-1$. Let $\epsilon \in \{ \pm \}$ be such that the theta lift $\theta(\tau^{x}, \psi)$ is non-zero on $O(V^{\epsilon})$. By parts 2, 3 of Theorem \ref{theorem:B} and Theorem \ref{theorem:A}, 
%the assertion (\ref{case2:B}) is equivalent to 
\begin{equation} \label{relation:z and epsilon}
\dfrac{z_{\psi}(\tau^{x})}{z_{\psi}(\tau^{ax})} = \dfrac{\epsilon(\pi \otimes \chi_{x}) \epsilon(V^{\epsilon})}{\epsilon(\pi \otimes \chi_{ax}) \epsilon(V^{\epsilon})} = \dfrac{\epsilon(\pi \otimes \chi_{x})}{\epsilon(\pi \otimes \chi_{ax})}.
\end{equation}
Recall that in odd residue characteristic
\begin{equation}
\chi_{x}(-1) = (x, -1) = (-1)^{\text{val}(x)},
\end{equation}
therefore in our case, $\chi_{a}(-1) = -1$. By Equation \ref{equation: central sign}, $z_{\psi}(\tau^{ax}) = z_{\psi}(\tau^{x}) \chi_{a}(-1)$, and hence we have $z_{\psi}(\tau^{ax}) = -z_{\psi}(\tau^{x})$. Therefore the Equation \ref{relation:z and epsilon} simplifies to
\begin{equation}   \label{*} 
 \epsilon(\pi \otimes \chi_{x}) = - \epsilon(\pi \otimes \chi_{ax}) \hspace{1cm} \forall x \in E^{\times}.
\end{equation}
Let $\chi_{u}=\chi_{-1}$ be the unramified quadratic character of $E^{\times}$. Let cond($\pi$) denote the conductor \index{conductor} of $\pi$.  By \cite[~Equation 3.2.1]{Tunnel78}, we have
\[
\epsilon(\pi \otimes \chi_{u}) = (-1)^{\text{cond}(\pi)} \epsilon(\pi).
\]
Thus if the conductor of $\pi$, is odd
\begin{equation}
\epsilon(\pi \otimes \chi_{u}) = - \epsilon(\pi). \label{case2:C}
\end{equation}
By \cite[~proposition 3.5]{Tunnel78}, it follows that if $\pi = \pi \otimes \chi_{b}$ then the conductor of $\pi$ is odd and hence Equation (\ref{case2:C}) is satisfied. The assumption $\pi \cong \pi \otimes \chi_{b}$ is equivalent to
\begin{equation}
\pi \otimes \chi_{-1} \cong \pi \otimes \chi_{a}.
\end{equation}
It follows from (\ref{case2:C}) that 
\begin{equation}
 \epsilon(\pi \otimes \chi_{a}) = \epsilon(\pi \otimes \chi_{-1}) = - \epsilon(\pi). \label{case2:D} 
\end{equation}
Now (\ref{*}) follows from (\ref{case2:D}) by direct verification for each element
\[
x \in E^{\times}/E^{\times 2} = \{ 1, -1, a, b=-a \}. \qedhere
\]
\end{proof}
\begin{corollary}
Assume that the residue characteristic of $E$ is odd and that $-1 \notin E^{\times 2}$. Let $\tilde{\pi} = ind_{\widetilde{{\rm GL}_{2}(E)}_{+}}^{\widetilde{{\rm GL}_{2}(E)}}(\mu \tau)$ where $\mu$ and $\tau$ are as before. Assume that for $\pi=\theta(\tau, \psi)$, $\pi \cong \pi \otimes \chi$ for some quadratic character $\chi$ of $E^{\times}$ corresponding to a quadratic ramified extension of $E$. Then there exists $a \in E^{\times}$ such that $\tau$ and $\tau_{W}^{a}$ have same central character, and for any non-trivial character $\psi'$ of $E$, $\tau$ admits a non-zero $\psi'$-Whittaker functional if and only if $\tau_{W}^{a}$ does not admit a non-zero $\psi'$-Whittake functional. Thus,
\[
\tilde{\pi}' := ind_{\widetilde{{\rm GL}_{2}(E)}_{+}}^{\widetilde{{\rm GL}_{2}(E)}}(\mu \tau_{W}^{a})
\]
has a complementary set of Whittaker models to that of $\tilde{\pi}$.
\end{corollary}
\begin{proof}
By Proposition \ref{propsitionB}, there exists an $a \in E^{\times}/E^{\times 2}$ such that $\tau$ and $\tau^{a}$ have opposite central characters, and for any non-trivial character $\psi'$ of $E$, $\tau$ admits a non-zero $\psi'$-Whittaker functional if and only if $\tau^{a}$ admits a non-zero $\psi'$-Whittaker functional. By part 1 of Theorem \ref{theorem:B}, $\tau$ and $\tau_{W}^{a}$ have the same central character, and for any non-trivial character $\psi'$ of $E$, $\tau$ admits a non-zero $\psi'$-Whittaker functional if and only if $\tau_{W}^{a}$ does not admit a non-zero $\psi'$-Whittaker functional. Now the corollary follows immediately.
\end{proof}
%\begin{corollary}
%Assume that the residue characteristic of $E$ is odd and that $-1 \notin E^{\times 2}$. Let $\tau$ be an irreducible admissible genuine supercuspidal representation of $\widetilde{{\rm SL}_{2}(E)}$. Assume that for $\pi=\theta(\tau, \psi)$, $\pi \cong \pi \otimes \chi$ for some quadratic character $\chi$ of $E^{\times}$ corresponding to a quadratic ramified extension of $E$. Let $\mu$ be a genuine character of $\tilde{Z}$ which is compatible with $\tau$ and
%\[
%\tilde{\pi} := ind_{\widetilde{{\rm GL}_{2}(E)}_{+}}^{\widetilde{{\rm GL}_{2}(E)}}(\mu \tau).
%\]
%Then the restriction of $\tilde{\pi}$ from $\widetilde{{\rm GL}_{2}(E)}$ to $\widetilde{{\rm SL}_{2}(E)}$ is multiplicity free.
%\end{corollary}
%\begin{proof}
%From Proposition \ref{propsitionB} there exists an $a \in E^{\times}/E^{\times 2}$ such that $\tau^{a}$ and $\tau$ have opposite central character so $\tau^{a} \ncong \tau$. Write $E^{\times}/E^{\times 2} = \{ 1, a, b, ab \}$. If multiplicity of the restriction of $\tilde{\pi}$ from $\widetilde{{\rm GL}_{2}(E)}$ to $\widetilde{{\rm SL}_{2}(E)}$ is more than one then we may assume that $\tau^{b} \cong \tau$. Then $\tau$ and $\tau^{b}$ have same Whittaker model. As $\tau$ and $\tau^{a}$ have same Whittaker models, it follows that all $\tau, \tau^{a}, \tau^{b}, \tau^{ab}$ have same Whittaker model. Thus $\Omega(\tilde{\pi}, \psi) = \Omega(\pi)$. But $\tilde{\pi}$ is not a principal series, contradiction. 
%\end{proof}

\addcontentsline{toc}{chapter}{Index}
\printindex


\addcontentsline{toc}{chapter}{Bibliography}

\begin{thebibliography}{99}

\bibitem{BerZel76}
 I. N. Bernshtein and A. V. Zelevinskii;
 \emph{Representations of the group $\GL(n,F)$ where $F$ is a non-Archimedian local field},
 Russian Math. Surveys, {\bf 31:3}, 5-76, 1976.

%\bibitem{BD01}
% Jean-Luc Brylinsky and Pierre Deligne: 
% \emph{Central extensions of reductive groups by $K_{2}$}.
% Publications Math{\'e}matiques de I.H.E.S., {\bf 94}, 5-85, 2001.

%\bibitem{Clozel91}
% Laurent Clozel;
% \emph{Invariant harmonic analysis on Schwartz space of a reductive $p$-adic group}.
% Bowdoin Conference, Progress in Mathematics, {\bf 101}, 1991.

%\bibitem{DS07}
% Stephen Delvin and Jason Schultz;
% \emph{Waldspurger's Involution and Lifting of Characters}.
%  Proceedings of the American Mathematical Society, {\bf 135}, No 3, 911-919, 2007.

% \bibitem{Flicker80}
%   Flicker, Yuval Z.:
%   \emph{Automorphic forms on covering group of $GL(2)$}.
%   Inventiones mathematicae, {\bf 57}, 119-182 (1980).

\bibitem{GGP12}
 W. T. Gan, B. H. Gross and D. Prasad:
 \emph{Symplectic local root numbers, central critical $L$-values, and restriction problems in the representation theory of classical groups}.
  Ast{\'e}risque, {\bf 346}, 1-109, (2012).


 \bibitem{Gelbart76}
  Stephan S. Gelbart:
  \emph{Weil's representation and the spectrum of metaplectic group}.
  Springer Lecture Notes in Mathematics, {\bf 532}, 1976.

% \bibitem{Gelbart78}
%   Gelbart, Stephen S. and Piatetski-Shapiro, I. I.:
%   \emph{Automorphic $L$-functions of half-integral weight}.
%   Proc. Natl. Acad. Sci. USA, vol. 75, No. 4, pp. 1620-1623, April 1978.
  
 \bibitem{Gelbart80}
   Stephan S. Gelbart:
   \emph{Distinguished Representations and Modular Forms of Half-integral Weight}.
   Inventiones Mathematicae, {\bf 59}, 145-188, 1980.
  
  \bibitem{GHP79}
 Stephan S. Gelbart, Roger Howe and I. I. Piatetski-Shapiro:
  \emph{Uniqueness and Existence of Wittaker models for metaplectic group}.
  Israel Journal of Mathematics, {\bf 34} Nos. 1-2, 21-37, 1979.
   
 \bibitem{GPS83}
   Stephan S. Gelbart and I. I. Piatetski-Shapiro:
   \emph{Some remarks on metaplectic cusp forms and the correspondences of Shimura and Waldspurger}.
   Israel Journal of Mathematics, {\bf 44} No. 2, 97-126, 1983. 
 
 \bibitem{God70}
  Roger Godement:
   \emph{Notes on Jacquet-Langlands theory}.
   The Institute for advanced study, 1970.
   
\bibitem{GrossPrasad92}
 Gross, Benedict H.; Prasad, Dipendra:
 \emph{On the decomposition of a representation of SO(n) when restricted to SO(n-1)}. 
 Canadian J. of Maths, {\bf 44}, 974-1002, 1992.

 \bibitem{JL70}
  Herve Jacquet and R. P. Langlands:
  \emph{Automorphic forms on $\GL_{2}$}.
  Springer lecture notes in mathematics, {\bf 114}, 1970.
% \bibitem{Kaz-Pat84}
%   Kazhdan, D.A. and Patterson, S.J.:
%   \emph{Metaplectic Forms}.
%    Publications Mathematiques, IHES, 59, 1984.
%   
\bibitem{KP84}
  D. A. Kazhdan and S. J. Patterson:
  \emph{Metaplectic Forms}.
   Publications Math{\'e}matiques, IHES, {\bf 59}, 35-142, 1984.

\bibitem{Kubota69}
  Tomia Kubota:
  \emph{On automorphic functions and the reciprocity law in a number field},
   Lectures in Mathematics, Department of Mathematics, Kyoto University, No. 2 Kinokuniya Book-Store Co., Ltd., Tokyo.

\bibitem{Kud94}
  Stephen S. Kudla:
  \emph{Splitting metaplectic covers of dual reductive pairs}.
  Israel Journal of Mathematics, {\bf 87}, 361-401, 1994.

\bibitem{Kutzko78}
 P. Kutzko
  \emph{On the supercuspidal representations of $\GL_{2}$, I}.
  American Journal of Mathematics, {\bf 100}, 43-60, 1978.
 
\bibitem{LL79}
J.-P. Labesse and R. P. Langlands:
 \emph{$L$-Indistinguishability for $\SL(2)$},
 Canad. J. Math. {\bf 31}  No. 4, 726–785, 1979. 

  \bibitem{WWLi}
Wen-Wei Li:
  \emph{La formule des traces pour les rev\^{e}tements de groupes r\'{e}ductifs connexes. II. Analyse harmonique locale},
  Ann. Scient. \'{E}c. Norm. Sup. 4 e s\'{e}rie, {\bf 45}, 787-859, 2012.

\bibitem{WWLi14}
  Wen-Wei Li:
  \emph{La formule des traces pour les revêtements de groupes r\'{e}ductifs connexes. I. Le d\'{e}veloppement g\'{e}om\'{e}trique fin},                    
  Journal f{\"u}r die reine und angewandte Mathematik. {\bf 686}, 37-109, (2014).

\bibitem{Mc2012}  
Peter J. McNamara:
  \emph{Principl series representation of metaplectic groups over local fields},
  Multiple Dirichlet series, L-functions and automorphic forms, Progr. Math., {\bf 300}, Birkh{\"a}user/Springer, New York, 299-327, 2012.

\bibitem{Mil71}
  John Milnor:
 \emph{Introduction to algebraic K-theory}. 
 Annals of Mathematics Studies, {\bf No. 72}, Princeton University Press, Princeton, N. J.; 1971.

\bibitem{MVW}
  C. M{\oe}glin, M. F. Vigneras and J.-L. Waldspurger:
  \emph{Correspondances de Howe sur un corps $p$-adique},
  Lecture notes in Mathematics,  Springer-Verlag, Berlin, {\bf 1291}, 1987.

\bibitem{MW87}
  C. M{\oe}glin, C. and J.-L. Waldspurger:
  \emph{Modeles de Whittaker degeneres pour de groupes p-adiques},
  Mathematics Zeitschrift, {\bf 196}, 427-452, 1987.
   
\bibitem{MW95}
 C. Moeglin and J.-L. Waldspurger:
  \emph{Spectral decomposition and Eisenstein series},
  Appendix I, Lifting of Unipotent subgroups into a central extension, Cambridge tracts in Mathematics, Cambridge University press, Cambridge, {\bf 113}, 273-277, 1995.

\bibitem{Prasad90}
 Dipendra Prasad:
  \emph{Trilinear forms for representations of $GL_2$ and local $\epsilon$-factors}.
  Compositio Math., {\bf 75}, 1-46, 1990.
  
\bibitem{Prasad92}
  Dipendra Prasad:
  \emph{Invariant forms for the representations of $GL_2$ over a local field}.
  American Journal of Mathematics, {\bf 114}, 1317-1363, 1992.
  
\bibitem{PrRa2000}
 Dipendra Prasad and A. Raghuram:
  \emph{Kirillov theory for $\GL_{2}(D)$ where $D$ is a division algebra over a non-Archimedian local field}.
  Duke Mathematics Journal, {\bf 104} No. 1, 19-44, 2000.

\bibitem{Rod75}
  Francois Rodier:
  \emph{Mod\'{e}les de Whittaker et caract\'{e}res de repr\'{e}esentations},
  Non commutative harmonic analysis (Actes Colloq., Marseille-Luminy, 1974), Lecture Notes in Mathematics, {\bf 466}, 151-171, 1975.

\bibitem{Riehm70}
  Carl Riehm:
  \emph{The norm 1 group of $p$-adic division algebra}.
   American Journal of Mathematics, {\bf 92} No. 2, 499-523, April 1970.

\bibitem{San14}
Sandeep V. Varma:
  \emph{On a result of Moeglin and Waldspurger in residual characteristic 2},
  To appear in Mathematische Zeitschrift.
  
\bibitem{Serre65}
Jean Pierre Serre:
 \emph{Lie algebras and Lie groups},
 Lecture notes in Mathematics, 1500. Springer-Verlag Berlin Heidelberg.
  
\bibitem{Shimura73}
 Goro Shimura:
 \emph{On modular forms of half integral weight}. 
 Annals of Mathematics (2) {\bf 97}, 440-481, 1973. 
 
\bibitem{Tunnel78}
 Jerrold B. Tunnel:
 \emph{On the local Langlands conjecture for $\GL(2)$};
 Inventiones Math. {\bf 46}, 179-200, 1978. 
  
\bibitem{Wald91}
  J.-L. Waldspurger:
  \emph{Correspondance de Shimura et quaternions},
   Forum Mathematicum, {\bf 3} No-3, 219-307, 1991.

\bibitem{Weil64}
 Andre Weil:
  \emph{Sur certains groupes d{'o}p{\'e}rateurs unitaires}, Acta Mathematica, {\bf 111}, 143-211, 1964.

\bibitem{Weissman09}
 Martin H. Weissman:
  \emph{Metaplectic tori over local fields}, 
 Pacific Journal of Mathematics, {\bf 241}, no. 1, 169-200, 2009. 

\bibitem{Weissman13}
 Martin H. Weissman:
 \emph{Split metaplectic groups and their $L$-groups}, \\
  http://arxiv.org/pdf/1108.1413v1.pdf

\end{thebibliography}
\end{document}